\documentclass[a4paper,reqno,11pt]{amsart}
\usepackage{amsmath, amsfonts, amssymb, amsthm, amscd, fge}
\usepackage[margin=1in,marginratio={1:1}]{geometry}
\usepackage{tabularx}
\usepackage{subfigure}
\usepackage{graphicx}
\usepackage{psfrag}
\usepackage{perpage}
\usepackage{url}
\usepackage{color}
\usepackage{mathrsfs}
\usepackage{dsfont} 
\usepackage{mathrsfs}
\usepackage[usenames,dvipsnames]{xcolor}
\usepackage{hyperref}
\usepackage{upgreek}
 
\hypersetup{
	colorlinks=true, linkcolor=NavyBlue,
	citecolor=ForestGreen
}
\usepackage[low-sup]{subdepth}
\usepackage{color}
\usepackage{multicol}
\usepackage{mathtools}
\usepackage{mathdots}
\usepackage{tikz}
\usepackage{mathabx}
\usetikzlibrary{arrows,decorations.pathmorphing,backgrounds,positioning,fit,petri} 
\usepackage{tikz,tikz-cd}
\usepackage{caption}
\usepackage{ragged2e}
\usepackage[utf8]{inputenc}
\usepackage[T1]{fontenc}

\usepackage{microtype}
\usepackage{accents}
\usepackage{stackengine}
\usepackage{hyperref}
\usepackage[vcentermath]{youngtab}
\justifying
\usepackage{comment}
\usepackage{color,soul}
\usepackage{xcolor}
\usepackage{standalone}
\usepackage{caption}
\usepackage[french,english]{babel}

\makeatletter
\def\@secnumfont{\bfseries\scshape}

\def\section{\@startsection{section}{1}%
  \z@{.7\linespacing\@plus\linespacing}{.5\linespacing}%
  {\normalfont\large\bfseries\scshape\centering}}

\def\subsection{\@startsection{subsection}{2}%
  \z@{.5\linespacing\@plus.7\linespacing}{-.5em}%
  {\normalfont\bfseries\scshape}}
  
\def\subsubsection{\@startsection{subsubsection}{3}%
  \z@{.5\linespacing\@plus.7\linespacing}{-.5em}%
  {\normalfont\scshape}}

\def\specialsection{\@startsection{section}{1}%
  \z@{\linespacing\@plus\linespacing}{.5\linespacing}%
  {\normalfont\centering\large\bfseries\scshape}}
\makeatother

\makeatletter

\renewenvironment{proof}[1][\proofname]{\par
\pushQED{\qed}%
\normalfont \topsep4\p@\@plus4\p@\relax
\trivlist
\item[\hskip\labelsep
\bfseries
#1\@addpunct{.}]\ignorespaces
}{%
\popQED\endtrivlist\@endpefalse
}
\makeatother

\makeatletter
\newcommand \Dotfill {\leavevmode \leaders \hb@xt@ 6pt{\hss .\hss }\hfill \kern \z@}
\makeatother

\makeatletter
\def\@tocline#1#2#3#4#5#6#7{\relax
  \ifnum #1>\c@tocdepth 
  \else
    \par \addpenalty\@secpenalty\addvspace{#2}%
    \begingroup \hyphenpenalty\@M
    \@ifempty{#4}{%
      \@tempdima\csname r@tocindent\number#1\endcsname\relax
    }{%
      \@tempdima#4\relax
    }%
    \parindent\z@ \leftskip#3\relax \advance\leftskip\@tempdima\relax
    \rightskip\@pnumwidth plus4em \parfillskip-\@pnumwidth
    #5\leavevmode\hskip-\@tempdima
      \ifcase #1
       \or\or \hskip 1.65em \or \hskip 3.3em \else \hskip 4.95em \fi%
      #6\nobreak\relax
    \Dotfill
    \hbox to\@pnumwidth{\@tocpagenum{#7}}\par
    \nobreak
    \endgroup
  \fi}
\makeatother

\makeatletter
\def\l@section{\@tocline{1}{0pt}{1pc}{}{\scshape}}
\renewcommand{\tocsection}[3]{%
\indentlabel{\@ifnotempty{#2}{\ignorespaces#1 #2.\hskip 0.7em}}#3}
\def\l@subsection{\@tocline{2}{0pt}{1pc}{5pc}{}}

\def\l@subsubsection{\@tocline{3}{0pt}{1pc}{7pc}{}}

\makeatother

\numberwithin{equation}{section}

\newtheoremstyle{mytheorem}{.7\linespacing\@plus.3\linespacing}{.7\linespacing\@plus.3\linespacing}%
     {\itshape}
     {}
     {\bfseries}
     {. }
     {0.3ex}
     {\thmname{{\bfseries #1}}\thmnumber{ {\bfseries #2}}\thmnote{ (#3)}}  

\theoremstyle{mytheorem}

\newtheorem{theorem}{Theorem}[section]
\newtheorem{lemma}[theorem]{Lemma}
\newtheorem{proposition}[theorem]{Proposition}

\newcommand{\bbE}{{\ensuremath{\mathbb E}} }

\newcommand{\bbP}{{\ensuremath{\mathbb P}} }

\newcommand{\cC}{{\ensuremath{\mathcal C}} }

\newcommand{\cF}{{\ensuremath{\mathcal F}} }
\newcommand{\cG}{{\ensuremath{\mathcal G}} }

\newcommand{\cL}{{\ensuremath{\mathcal L}} }

\newcommand{\cZ}{{\ensuremath{\mathcal Z}} }

\newcommand\sfH{\mathsf H}

\newcommand\sfT{\mathsf T}

\newcommand\sfW{\mathsf W}

\newcommand\sfa{\mathsf a}
\newcommand\sfb{\mathsf b}
\newcommand\sfc{\mathsf c}
\newcommand\sfd{\mathsf d}
\newcommand\sfe{\mathsf e}

\newcommand\sfg{\mathsf g}
\newcommand\sfh{\mathsf h}
\newcommand\sfi{\mathsf i}

\newcommand\sfl{\mathsf l}
\newcommand\sfm{\mathsf m}
\newcommand\sfn{\mathsf n}
\newcommand\sfo{\mathsf o}
\newcommand\sfp{\mathsf p}

\newcommand\sfr{\mathsf r}
\newcommand\sfs{\mathsf s}
\newcommand\sft{\mathsf t}

\newcommand\Strip{B^{\sfs \sft \sfr \sfi \sfp}_N}

\renewcommand{\tilde}{\widetilde}          
\DeclareMathSymbol{\leqslant}{\mathalpha}{AMSa}{"36} 
\DeclareMathSymbol{\geqslant}{\mathalpha}{AMSa}{"3E} 
\DeclareMathSymbol{\eset}{\mathalpha}{AMSb}{"3F}     
\newcommand{\dd}{\text{\rm d}}             

\newcommand{\sumtwo}[2]{\sum_{\substack{#1 \\ #2}}} 
\newcommand{\sumthree}[3]{\sum_{\substack{#1 \\ #2 \\ #3}}} 
\newcommand{\prodtwo}[2]{\prod_{\substack{#1 \\ #2}}}     

\newcommand{\R}{\mathbb{R}}

\newcommand{\Z}{\mathbb{Z}}
\newcommand{\N}{\mathbb{N}}

\newcommand{\PEfont}{\mathrm}

\DeclareMathOperator{\cov}{\ensuremath{\PEfont Cov}}
\newcommand{\p}{\ensuremath{\PEfont P}}
\newcommand{\e}{\ensuremath{\PEfont E}}

\newcommand{\E}{\e}
\renewcommand{\P}{\p}

\DeclareMathOperator{\bbvar}{\ensuremath{\mathbb{V}ar}}

\newcommand{\ind}{\mathds{1}}

\renewcommand{\epsilon}{\varepsilon}
\renewcommand{\theta}{\vartheta}
\renewcommand{\rho}{\varrho}

\newcommand{\ms}{\scriptscriptstyle}

\newenvironment{myenumerate}{%
\renewcommand{\theenumi}{\arabic{enumi}}%
\renewcommand{\labelenumi}{{\rm(\theenumi)}}%
\begin{list}{\labelenumi}
	{%
	\setlength{\itemsep}{0.4em}%
	\setlength{\topsep}{0.5em}%
	\setlength\leftmargin{2.45em}%
	\setlength\labelwidth{2.05em}%
	\setlength{\labelsep}{0.4em}%
	\usecounter{enumi}%
	}%
	}%
{\end{list}
}

{\end{list}
}

{\end{list}
}

{\end{myenumerate}}

\newenvironment{myitemize}{%
\begin{list}{$\bullet$}%
 	{%
	\setlength{\itemsep}{0.4em}%
	\setlength{\topsep}{0.5em}%
	\setlength\leftmargin{2.45em}%
	\setlength\labelwidth{2.05em}%
	\setlength{\labelsep}{0.4em}%
	}%
	}%
{\end{list}}

\renewenvironment{itemize}{
\begin{myitemize}}%
{\end{myitemize}}

\MakePerPage[2]{footnote} 

\def\dd{\mathrm{d}}

\definecolor{cadmiumgreen}{rgb}{0.0, 0.42, 0.24}
\definecolor{red(munsell)}{rgb}{0.95, 0.0, 0.24}

\usepackage{bbold} 
\renewcommand{\ind}{\mathbb{1}}

\newcommand{\norm}[1]{\left\lVert#1\right\rVert} 
\DeclareRobustCommand{\ubar}[1]{\underaccent{\bar}{#1}}     
 
\makeatletter
\newenvironment{abstracts}{%
  \ifx\maketitle\relax
    \ClassWarning{\@classname}{Abstract should precede
      \protect\maketitle\space in AMS document classes; reported}%
  \fi
  \global\setbox\abstractbox=\vtop \bgroup
    \normalfont\Small
    \list{}{\labelwidth\z@
      \leftmargin3pc \rightmargin\leftmargin
      \listparindent\normalparindent \itemindent\z@
      \parsep\z@ \@plus\p@
      
      \itemsep\medskipamount
    }%
}{%
  \endlist\egroup
  \ifx\@setabstract\relax \@setabstracta \fi
}

\newcommand{\abstractin}[1]{%
  \otherlanguage{#1}%
  \item[\hskip\labelsep\scshape\abstractname.]%
}
\makeatother

\begin{document}

\title[]{Edwards-Wilkinson fluctuations for the directed polymer in the full $L^2$-regime for dimensions $d \geq 3$}

\author[D. Lygkonis, N. Zygouras]{Dimitris Lygkonis, Nikos Zygouras }
\address{Department of Mathematics\\
	University of Warwick\\
	Coventry CV4 7AL, UK}
\email{Dimitris.Lygkonis@warwick.ac.uk, N.Zygouras@warwick.ac.uk}
\begin{abstracts}
	\abstractin{english}
	We prove that in the full $L^2$-regime the partition function of the directed polymer model in
	dimensions $d\geq 3$, if centered, {scaled} and  averaged with respect to a test
	function $\varphi \in C_c(\R^d)$, converges in distribution to a Gaussian
	random variable with explicit variance.
	Introducing a new idea in this context of a martingale difference representation, we also prove that
	the log-partition function, which can be viewed as a discretisation of the KPZ equation,  exhibits the same fluctuations, when centered and averaged with respect to a test function.
	Thus, the two models fall within the Edwards-Wilkinson universality class in the full $L^2$-regime, a result that was only established,  so far,  for a strict subset of this regime in
	$d\geq 3$.

	\vspace{0.5cm}
	\abstractin{french}
	Nous démontrons que dans tout le régime $L^2$, la fonction de partition du modèle de polymères dirigés en dimension $d \geq 3$, si elle est centrée, normalisée et moyennée par rapport à une fonction de test $\phi \in C_c(\mathbb{R}^d)$, converge en distribution vers une variable aléatoire gaussienne dont la variance est explicite. En introduisant une nouvelle idée dans ce contexte de la représentation de différence de martingale, nous démontrons également que le logarithme de la fonction de partition, qui peut être vu comme une discrétisation de l'équation KPZ, possède les mêmes fluctuations, lorsqu'il est centré et moyenné par rapport à une fonction de test. En conséquence, tous les deux modèles se trouvent dans la classe d'universalité d'Edwards-Wilkinson dans tout le régime $L^2$, un résultat qui a seulement été établi jusque-là dans un sous-ensemble strict de ce régime en dimension $d \geq 3$.
\end{abstracts}
\selectlanguage{english}

\date{\today}

\keywords{Directed Polymer Model, Edwards-Wilkinson Fluctuations, Kardar-Parisi-Zhang universality}
\subjclass{Primary: 82B44}

\maketitle
{
	\hypersetup{linkcolor=black}
	\tableofcontents
}

\section{Introduction and results}

In this paper, we study the directed polymer in dimensions $d \geq 3$. The directed polymer model is defined as a coupling of the simple random walk with a random environment given by i.i.d. random variables, whose strength is tuned by a parameter $\beta$, corresponding to the inverse temperature. In particular, let $\displaystyle (\omega_{n,x})_{(n,x)\in \N \times \Z^d}$ be a collection of i.i.d. random variables with  law $\bbP$ such that
\begin{align*}
	\bbE[\omega]=0, \hspace{1cm} \bbE[\omega^2]=1, \hspace{1cm} \lambda(\beta):=\log \bbE[e^{\beta \omega}] < \infty\,, \hspace{0.3cm} \forall  \beta \in (0,\infty) .
\end{align*}
We also consider a simple random walk, whose distribution we denote by $\P_x$ when starting from $x\in \Z^d$. When starting from $0$ we will refrain from using the subscript and just write $\P$. We will use the notation $q_n(x):=\P(S_n=x)$ for the transition kernel of the random walk.
The directed polymer measure on polymer paths of length $N$, starting from position $x$ and at inverse temperature $\beta \in (0,\infty)$
is defined as
\begin{align} \label{radnik}
	\frac{\dd \P_{N,\beta,x}}{\dd \P_x}(S):=\frac{1}{Z_{N,\beta}(x)}\exp\Big(\sum_{n=1}^{N} \big(\hspace{0.2mm} \beta \omega_{n,S_n} -\lambda(\beta)\hspace{0.2mm}\big) \Big),
\end{align}
where
\begin{align} \label{partitionfunction}
	Z_{N,\beta}(x):=\E_{x}\Big[\exp\Big(\sum_{n=1}^{N} \big(\beta \omega_{n,S_n} -\lambda(\beta) \big) \Big)   \Big ] ,
\end{align}
is a {random} normalising constant which makes the polymer measure a probability measure. This is the so-called \textit{partition function} of the model
and will be the object of our main interest in this paper. When the starting point of the random walk is the origin we will simply write $Z_{N,\beta}$ instead of $Z_{N,\beta}(0)$.

The directed polymer model has, by now, a long history starting with the works of Imbrie-Spencer \cite{IS88} and Bolthausen \cite{B89},
who showed the existence of a {\it weak disorder regime} in dimension $d\geq 3$ and when $\beta$ is small enough.
It was then shown that  paths weighted
by the polymer measure exhibit diffusive behaviour. The regime of $\beta$ that was considered in these works was what we name here the ``$L^2$-regime'', which is characterised by the boundedness of the $L^2(\bbP)$ norm of the partition function $Z_{N,\beta}$.
This regime can be explicitly characterised: if we denote by $\lambda_2(\beta):=\lambda(2\beta)-2\lambda(\beta)$ and by $\pi_d$ the
probability that a $d$-dimensional simple random walk, starting from the origin, will return to the origin, then
\begin{align*}
	\beta_{L^2}:=\beta_{L^2}(d):=\sup\big\{\beta \colon \lambda_2(\beta)<\log\big(\tfrac{1}{\pi_d}) \,\big\}.
\end{align*}
This characterisation is achieved via the simple and standard computation
\begin{align} \label{loctime}
	\bbE\big[(Z_{N,\beta}(x))^2\big]=
	\E^{\otimes 2}\,\big[e^{\lambda_2(\beta)\sum_{n=1}^N \ind_{S_n^1=S_n^2} }\big]=
	\E\big[e^{\lambda_2(\beta)\mathcal{L}_N }\big] ,
\end{align}
where $S_n^1, S_n^2$ are two independent copies of the simple random walk, starting from the origin, with joint law denoted by
$\P^{\otimes 2}$. Moreover, $ \mathcal{L}_N:=\sum_{n=1}^N \ind_{S_{2n}=0}$ denotes the number of times that a $d$-dimensional simple random walk returns to zero and for the second equality we made use of the equality in law $\sum_{n=1}^N \ind_{S^1_n=S^2_n} \overset{\text{law}}{=} \sum_{n=1}^N \ind_{S_{2n}=0}$.
Since the simple random walk is transient in dimensions $d \geq 3$, one can see that $\mathcal{L}_N$ converges almost surely to a random variable $\mathcal{L}_{\infty}$ as $N \to \infty$ and the limiting random variable $\mathcal{L}_{\infty}$ follows a geometric distribution with success probability equal to $\pi_{d}<1$. In particular, we have that $\bbE\big[(Z_{N,\beta}(x))^2\big] \xrightarrow{N \to \infty} \E\big[e^{\lambda_2(\beta)\mathcal{L}_\infty }\big] $ and
\begin{align} \label{secondmoment}
	\E\big[e^{\lambda_2(\beta)\mathcal{L}_\infty }\big] =
	\left\{\begin{matrix}
		\frac{1-\pi_d}{1-\pi_d e^{\lambda_2(\beta)}}, & \text{ \quad  if  } \lambda_2(\beta)<\log(\frac{1}{\pi_d}) \\
		\infty \, \, ,                                & \text{otherwise.}
	\end{matrix}\right.
\end{align}
The weak disorder regime was subsequently characterised as the regime $\beta<\beta_c(d)$ where $Z_{N,\beta}$ converges almost surely to a strictly positive random variable. Clearly $\beta_c(d) \geq \beta_{L^2}(d)$ but a concrete characterisation of $\beta_c$ is still missing and in fact it took
some time to resolve the nontriviality {of the interval $(\beta_{L^2}(d), \beta_c(d))$ for $d\geq 3$, \cite{BS10, BS11, BT10, BGH11}}.
The formulation of the weak disorder regime as the regime where $Z_{N,\beta}\xrightarrow{\text{a.s.}} Z_{\infty, \beta}>0$ is largely due to the works
of Comets, Shiga, Yoshida \cite{CSY03, CSY04, CY06}, see also the recent monograph \cite{C17} for a more detailed bibliographical account with respect to these issues.

The above works (as well as several other relevant ones e.g. \cite{CL17, CN19, MSZ16} etc.)
have  focused on studying the partition function at a fixed starting point. Here, on the other hand, we are interested in the spatial fluctuations
of the field of partition functions $\big(Z_{N,\beta}(x)\big)_{x\in\Z^d}$, when then initial point varies, and we will show it exhibits Edwards-Wilkinson (EW) fluctuations in the $L^2$-regime.
{
Let us recall that the Edwards-Wilkinson fluctuations are determined as the fluctuations of the field
that arises as the solution to the additive stochastic heat equation
\begin{equation}\label{eq:ASHE}
	\left\{
	\begin{aligned}
		\partial_t v^{(c)}(t,x) & = \frac{1}{2} \Delta v^{(c)} (t,x) + c \, \xi(t,x) \\
		v^{(c)} (0,x)           & \equiv 0
	\end{aligned}
	\right.
\end{equation}
where $c$ is a model related constant} {and $\xi$ denotes space-time white noise, that is the Gaussian process with covariance structure
$\bbE[\xi(t,x)\xi(s,y)]=\delta(t-s)\delta(x-y)$ for $t,s>0$ and $x,y \in \R^d$.}
Our first result is the following theorem:
\begin{theorem} \label{gaussianity}
	Let $d \geq 3$,  $\beta \in (0,\beta_{L^2}(d))$ and consider the field of
	partition functions of the d-dimensional directed polymer $\big(Z_{N,\beta}(x)\big)_{x\in\Z^d}$. If
	$\varphi \in C_c(\R^d)$ is a test function, denote by
	\begin{align} \label{ZNBF}
		Z_{N,\beta}(\varphi):= \sum_{x \in \Z^d} \Big(Z_{N,\beta}(x)-\bbE\big[Z_{N,\beta}(x)\big]\Big) \frac{\varphi\big(\frac{x}{\sqrt{N}}\big)}{N^{\frac{d}{2}}}= \sum_{x \in \Z^d} \big(Z_{N,\beta}(x)-1\big) \frac{\varphi\big(\frac{x}{\sqrt{N}}\big)}{N^{\frac{d}{2}}} ,
	\end{align}
	the averaged partition function over $\varphi$. Then the rescaled
	sequence $(N^{\frac{d-2}{4}}Z_{N,\beta}(\varphi))_{N \geq 1}$ converges in distribution to a centered Gaussian random variable $\cZ_\beta (\varphi)$ with variance
	given by
	\begin{align} \label{variance}
		{\bbvar[\cZ_\beta(\varphi)] = \cC_{\beta} \int_{0}^{1} \dd t \int_{\R^{d}\times \R^{d}}  \dd x \dd y \: \varphi(x)g_{\frac{2t}{d}}(x-y)\varphi(y) \, ,}
	\end{align}
	{where} $g(\cdot)$  is the $d$-dimensional heat kernel,
	{$\cC_{\beta}= \: \sigma^2(\beta) \: \E[e^{\lambda_2 (\beta) \mathcal{L}_\infty}]$ and $\sigma^2(\beta)=e^{\lambda_2 (\beta)}-1$.}
\end{theorem}
Besides the interest stemming from understanding spatial correlations in the polymer model, the above result is motivated by intense recent activity in the field of singular stochastic PDEs. The field of partition functions ${\big(Z_{N,\beta}(x)\big)_{x\in\Z^d}}$
of the directed polymer model can be seen as a discretisation (via the stochastic Feynman-Kac formula \cite{BC95})  of the stochastic heat
equation (SHE) with multiplicative noise:
\begin{align} \label{SHE}
	\partial_t u= \frac{1}{2} \Delta u+ \beta \xi(t,x) u,\qquad t>0,x \in\R^d,
\end{align}
with flat $u(0,\cdot)\equiv 1$ initial condition.
Contrary to the case of dimension $d=1$, where one can make sense of~\eqref{SHE} by using classical It\^{o} theory, in dimensions $d \geq 2$ this is not possible due to the lack of regularity of the space-time white noise, which makes the product $u \cdot \xi$ ill defined.
Recent works \cite{MSZ16, GRZ18, CCM19} have shown that a meaning to \eqref{SHE} for $d\geq 3$ can be provided when $\beta$ is small
(a strict subset of the $L^2$-regime) by smoothing out the noise via
spatial mollification with a smooth density $j(\cdot)$ as $\xi_\epsilon(t,x):=\epsilon^{-d}\int_{\R^d} \xi(t,x) j(x/\epsilon) \dd x$ and solving
first the regularised equation
\begin{align} \label{smoothSHE}
	\partial_t u_\epsilon= \frac{1}{2} \Delta u_\epsilon+ \beta \epsilon^{\tfrac{d-2}{2}} \,\xi_\epsilon(t,x) u_\epsilon,\qquad t>0,x \in\R^d.
\end{align}
As $\epsilon$ tends to zero, the solution $u_\epsilon(t,\cdot)$, when centered and scaled, converges (as a field), for $\beta$ small,
to the solution of the additive stochastic heat equation, whose statistics determine the Edwards-Wilkinson class. Our result, Theorem \ref{gaussianity},
viewed as a different type of approximation to the SHE, provides the extension of the meaning of \eqref{SHE} to the whole $L^2$ regime.
We also establish a similar result for the field of log-partition functions. In this case we will additionally require that the disorder satisfies a (mild) concentration property \eqref{concentration}. More precisely,
\begin{theorem} \label{loggaussianity}
	Let $d \geq 3$, $\beta \in (0,\beta_{L^2}(d))$ and consider the fields of
	log-partition functions of the $d$-dimensional directed polymer $\big(\log Z_{N,\beta}(x)\big)_{x\in\Z^d}$, with disorder that satisfies
	concentration property \eqref{concentration}. If
	$\varphi \in C_c(\R^d)$ is a test function, we have that
	\begin{align}\label{KPZavg}
		N^{\frac{d-2}{4}} \sum_{x \in \Z^d} \Big(\log Z_{N,\beta}(x) -\bbE\big[\log Z_{N,\beta}(x) \big]\Big) \frac{\varphi\big(\frac{x}{\sqrt{N}}\big)}{N^{\frac{d}{2}}} \, ,
	\end{align}
	converges in distribution to the centered Gaussian random variable $\cZ_\beta (\varphi)$ defined in Theorem \ref{gaussianity}.
\end{theorem}
Given that $h(t,x):=\log u(t,x)$, with $u(t,x)$ the solution to the SHE, is formally the solution to the KPZ equation
\begin{align} \label{KPZ}
	\partial_t h=\frac{1}{2} \Delta h + \frac{1}{2} |\nabla h|^2 +\beta \xi ,
\end{align}
the field of log-partition functions can be viewed as a discretization of the KPZ equation. Dimensions $d\geq 3$ are known in the recent theory of SPDEs as {\it supercritical} dimensions and thus the theories of regularity structures \cite{H14}, paracontrolled distributions \cite{GIP17}, energy solutions \cite{GJ14}  do not apply. Alternatively, Edwards-Wilkinson limiting fluctuations
for the regularised KPZ
\begin{align} \label{smoothKPZ}
	\partial_t h_\epsilon=\frac{1}{2} \Delta h_\epsilon + \frac{1}{2} |\nabla h_\epsilon|^2 +\beta \epsilon^{\tfrac{d-2}{2}} \xi_\epsilon -
	\frac{1}{2} \beta^2\epsilon^{-2} \|j\|^2_{L^2(\R^d)}\, ,
\end{align}
were recently established  in \cite{GRZ18,DGRZ18} through Malliavin calculus techniques, for small $\beta$. Moreover, in
\cite{MU18} renormalisation and perturbation arguments were used to establish Edwards-Wilkinson fluctuations for small $\beta$, when the mollification is performed in both space and time.
\cite{CCM19b} also studied the one-point limit fluctuations of \eqref{smoothKPZ} in a subset of the $L^2$ regime.

Before closing this introduction we mention that analogous results to Theorems \ref{gaussianity} and \ref{loggaussianity}, for regularisations of
SHE and KPZ as in \eqref{smoothSHE}, \eqref{smoothKPZ} were simultaneously and independently established by Cosco-Nakajima-Nakashima
\cite{CNN20} via quite different methods than ours, based on stochastic calculus and local limit theorems for polymers inspired by earlier works
of Comets-Neveu \cite{CN95} and of Sinai \cite{S95} (see also \cite{V06, CN19, CCM19b}).
Our methods, as we will explain in more detail in the next section, are based on analysis of chaos
expansions inspired by works on scaling limits of disordered systems \cite{CSZ17a, CSZ16} and two dimensional polymers, SHE and KPZ \cite{CSZ17b,CSZ18b} (alternative methods to the two dimensional case, which however do not cover the whole
$L^2$ - in this case also subcritical - regime, are those of \cite{CD18, G18}). A very interesting, open problem is to go beyond
the $L^2$ regimes. Currently the only works in this direction are \cite{CSZ18a, CSZ19, GQT19} on the moments of polymers and SHE {\it on
		the critical temperature} in dimension two. However, these moment estimates are not enough to determine the distribution.

\section{Outline, {main ideas} and comparison to the literature}

We will describe in this section the method we follow as well as the new {ideas} required. The basis of our  analysis is the chaos expansion of the polymer
partition function as
\begin{align}\label{intro-chaos}
	Z_{N,\beta}(x)=1+\sum_{k=1}^N \sigma^k \sumtwo{1 \leq n_1 < \ldots < n_k \le N}
	{z_1, \ldots, z_k \in \Z^d} q_{n_1}(z_1-x)\prod_{i=2}^k q_{n_i-n_{i-1}}(z_i-z_{i-1}) \prod_{i=1}^k \eta_{n_i,z_i} \, ,
\end{align}
where $q_n(x)=\P(S_n=x),\, \sigma=\sigma(\beta):= \sqrt{e^{\lambda_2(\beta)}-1}$ and $\eta_{n,z}:=\sigma^{-1}\big( e^{\beta\omega_{n,z}-\lambda(\beta)}-1\big)$, see \eqref{chaos1} for the details of this derivation.

To prove the central limit theorem for $(N^{\frac{d-2}{4}}Z_{N,\beta}(\varphi))_{N \geq 1}$ we make use of the so called Fourth Moment Theorem \cite{dJ87, NP05, NPR10, CSZ17b}, which states that a sequence of random variables in a fixed Wiener chaos, {normalised to have mean zero and variance one,} converges to a standard normal random variable if its fourth moment {converges} to $3$.
Of course, in order to be able to reduce ourselves to a fixed chaos, we need to perform truncation and for this, the assumption
of bounded second moments ($L^2$ regime) plays an important role. This approach of analysing chaos expansions of partition functions was first used
in \cite{CSZ17b} in a framework that also included the analysis of the two dimensional directed polymer and SHE.  The work, which is needed to carry out this approach in $d\geq 3$, is actually easier than the $d=2$ case in \cite{CSZ17b}. The reason for this is that the variance of $Z_{N,\beta}$ is a functional of the local time
$\cL_N$, see \eqref{loctime}, which stays bounded in $d\geq 3$ but grows logarithmically in $d=2$, introducing,  in the latter case,  a certain
multiscale structure. Still, a careful combinatorial accounting and analytical estimates, which actually deviate from those in \cite{CSZ17b}, are needed to handle the $d\geq 3$ case.
The detailed analysis of such expansion is what allows to go all the way to the $L^2$ critical temperature,
as compared to the previous works \cite{GRZ18}, \cite{MU18}. The work \cite{GRZ18} established the central limit theorem via a ``linearisation'' through Malliavin calculus (Clark-Ocone formula) and homogenisation / mixing estimates only for sufficiently small $\beta$.
On the other hand, the renormalisation methods employed in \cite{MU18} are necessarily restricted to a perturbative (small $\beta$) regime.

For the Edwards-Wilkinson fluctuations of the log-partition function, namely Theorem \ref{loggaussianity}, we also adapt the approach  of
``{linearisation}'' via chaos expansion proposed in \cite{CSZ18b}.
However, the analysis in $d\geq 3$, required to achieve the goal of going all the way to $\beta_{L^2}(d)$, is rather more subtle. The reason is that the power law prefactor $N^{\frac{d-2}{4}}$ in \eqref{KPZavg} (as opposed to the corresponding $\log N$ prefactor in \cite{CSZ18b}) does not allow for any ``soft'' (or even more intricate) bounds \`a la Cauchy-Schwarz or triangle inequalities in the approximations. Instead, we have
to look carefully at the correlation structure that will cancel the $N^{\frac{d-2}{4}}$. This correlation structure is rather obvious in the case of the partition function and can be already understood by looking at the first term of the chaos expansion of $N^{\frac{d-2}{4}} Z_{N,\beta}(\varphi)$ as derived from
\eqref{intro-chaos}, which is
\begin{align*}
	\hspace{-1cm}
	N^{\frac{d-2}{4}} \sum_{x\in\Z^d}  \frac{\varphi(\tfrac{x}{\sqrt{N}})}{N^{\frac{d}{2}}}  \sum_{z \in \Z^d, 1\leq n\leq N} q_n(z-x) \eta_{n,z} \, ,
\end{align*}
and whose variance is easily computed as
\begin{align*}
	\hspace{1cm} & N^{\frac{d-2}{2}} \sum_{x,y\,\in\Z^d}  \frac{\varphi(\tfrac{x}{\sqrt{N}}) \varphi(\tfrac{y}{\sqrt{N}}) }{N^{d}}  \sum_{z \in \Z^d, 1\leq n\leq N} q_n(z-x) q_{n}(z-y) \\
	=            & N^{\frac{d-2}{2}} \sum_{x,y\,\in\Z^d}  \frac{\varphi(\tfrac{x}{\sqrt{N}}) \varphi(\tfrac{y}{\sqrt{N}}) }{N^{d}}  \sum_{1\leq n\leq N} q_{2n}(x-y)\, .
\end{align*}
The factor {$N^{\frac{d-2}{2}}$} is then absorbed by the sum $\sum_n q_{2n}(x-y)$ in a Riemann sum approximation. What underlies the above computation is that correlations are captured by two independent copies of the random walk, one starting at $x$ and another at $y$, meeting at some point by time $N$. The probability of such a coincidence event compensates for the {$N^{\frac{d-2}{2}}$}.

When considering the log-partition functions, the above described mechanism is not obvious, as $\log Z_{N,\beta}$ does not admit an equally nice
and tractable chaos expansion. Nevertheless, it is necessary (which was not the case in \cite{CSZ18b}) to tease out the aforementioned correlation structure, in order to absorb $N^{\frac{d-2}{4}}$ and carry out the approximation. The way we do this is by writing $\log Z_{N,\beta}$ (or
more accurately a certain approximation, which we call $\log Z^A_{N,\beta}$, see \eqref{Z-decomp}) as a martingale difference:
\begin{align*}
	{\log Z_{N,\beta} - \bbE\big[ \log Z_{N,\beta} \big] =\sum_{j\geq 1} \Big(  \bbE\big[ \log Z_{N,\beta}\, |\, \cF_j \big] - \bbE\big[ \log Z_{N,\beta}\, |\, \cF_{j-1} \big] \Big) \, ,}
\end{align*}
where {$\{\cF_j\colon j\geq 1\}, \cF_0=\{\emptyset,\Omega\}$} is a filtration generated as {$\cF_j=\sigma( \omega_{a_i}\colon i=1,...,j)$} with $\{a_1,a_2,...\}$ an
enumeration of $\N \times \Z^d$. By adding the information from the disorder at a single additional site at each time, we keep track of how the
polymer explores the disorder and this allows (after a certain ``resampling'' procedure) to keep track of the correlations. The martingale difference approach we
	{introduce} has in some sense some similarity to the Clark-Ocone formula,
which was used in the work of \cite{GRZ18,DGRZ18}.
However, our approach of exploring a single new site disorder at a time seems to be necessary for the precise estimates that we need, in order to reach the whole $L^2$ regime. Along the way,  {a fine use of} concentration and negative tail estimates of the log-partition function (e.g.
Proposition \ref{lefttail}) is {made}.

Once all the necessary approximations to the log-partition function are completed, the task is then reduced to a central limit theorem for
a partition function of certain sorts, thus bringing us back to the context of Theorem \ref{gaussianity}.
The previous work of \cite{DGRZ18} seems to be necessarily restricted to a small sub-region of $(0,\beta_{L^2})$, as a consequence of
both the linearisation approach employed but also more importantly (as far as we can tell)
due to the use of the so-called ``second order Poincar\'{e} inequality''  for the central limit theorem,
which requires higher moment estimates that lead outside the $L^2$ regime, if $\beta$ is not restricted to be small enough.

The parallel work of Cosco-Nakajima-Nakashima \cite{CNN20} achieves the Edwards-Wilkinson fluctuations for the SHE and the KPZ by quite different methods than ours, by making use
of clever applications of stochastic calculus and the local limit theorem for polymers \cite{S95, V06, CN19}.

\section{The Central Limit theorem for \texorpdfstring{$ Z_{N,\beta}(\varphi)$}{}    }    \label{section3}

This section is devoted to the proof of Theorem \ref{gaussianity}.
Throughout the paper we rely on polynomial chaos expansions of the
partition function. Specifically, consider the partition function of a polymer chain of length $N$ starting from $x$ at time zero. We can write
\begin{align} \label{chaos1}
	Z_{N,\beta}(x) & = \E_x\bigg[\prod_{1 \le n \le N, \, z \in \Z^d} e^{ (\beta \omega_{n,z} - \lambda(\beta)) \ind_{S_n=z}} \bigg]                 \notag \\
	               & = \E_x\bigg[\prod_{1 \le n \le N, \, z \in \Z^d} \big(1+ (e^{\beta \omega_{n,z} - \lambda(\beta)}-1) \ind_{S_n=z} \big)\bigg]  \notag  \\
	               & = 1+\sum_{k=1}^N \sigma^k \sumtwo{1 \leq n_1 < \ldots < n_k \le N}
	{z_1, \ldots, z_k \in \Z^d} q_{n_1}(z_1-x)\prod_{i=2}^k q_{n_i-n_{i-1}}(z_i-z_{i-1}) \prod_{i=1}^k \eta_{n_i,z_i} \, .
\end{align}
For $(n,z) \in \N \times \Z^d$ we have denoted by $\eta_{n,z}$ the centered random variables
\begin{align} \label{etas}
	\eta_{n,z}:=\frac{e^{\beta \omega_{n,z} - \lambda(\beta)}-1}{\sigma} \, .
\end{align}
The number $\sigma=\sigma(\beta)$ is chosen so that for $(n,z) \in \N \times \Z^d$ the
centered random variables $\eta_{n,z}$ have unit variance. A simple calculation shows that $
	\sigma=\sqrt{e^{\lambda(2\beta)-2\lambda(\beta)}-1}$.
Also, the last equality in \eqref{chaos1} comes from expanding the product in the second line of~\eqref{chaos1} and interchanging the expectation with the summation.
By using the expansion~\eqref{chaos1} we can derive an expression for the averaged
partition function. Let us fix a test function $\varphi \in C_c(\R^d)$.
	{For the sake of the presentation, we will adopt the following notation:}
\begin{align} \label{phinotation}
	\varphi_N(x_1,...,x_k):=\prod_{u \in \{x_1,...,x_k\} }\frac{\varphi\big(\frac{u}{\sqrt{N}}\big)}{N^\frac{d}{2}} \, , \qquad k \geq 1 \, .
\end{align}
We have
\begin{align}\label{chaos2}
	Z_{N,\beta}(\varphi) & :=\sum_{x \in \Z^d} (Z_{N,\beta}(x)-1)\, \varphi_N(x) \notag                                                                                      \\
	                     & = \sum_{k=1}^N \sigma^k \sumtwo{1 \leq n_1 < \ldots < n_k \le N}
	{z_1, \ldots, z_k \in \Z^d} \Bigg ( \sum_{x \in \Z^d} \varphi_N(x)\, q_{n_1}(z_1-x)\Bigg) \prod_{i=2}^k q_{n_i-n_{i-1}}(z_i-z_{i-1}) \prod_{i=1}^k \eta_{n_i,z_i} \notag \\
	                     & = \sum_{k=1}^N Z^{(k)}_{N,\beta}(\varphi)  \, ,
\end{align}
where
\begin{align} \label{fixedchaos}
	Z^{(k)}_{N,\beta}(\varphi):= { \sigma^k} \sumtwo{1 \leq n_1 < \ldots < n_k \le N}
	{z_1, \ldots, z_k \in \Z^d} \Bigg ( \sum_{x \in \Z^d} \varphi_N(x) q_{n_1}(z_1-x)\Bigg) \prod_{i=2}^k q_{n_i-n_{i-1}}(z_i-z_{i-1}) \prod_{i=1}^k \eta_{n_i,z_i} \, .
\end{align}

{The first step towards the proof of Theorem \ref{gaussianity} is the following proposition which identifies the limiting variance of the sequence $(N^{\frac{d-2}{4}}Z_{N,\beta}(\varphi))_{N\geq1}$.}

\begin{proposition} \label{variancep}
	Let $d \geq 3$, $\beta \in (0,\beta_{L^2})$ and fix $\varphi \in C_c(\R^d)$ {to be a test function.} Consider the sequence $(N^{\frac{d-2}{4}}Z_{N,\beta}(\varphi))_{N\geq1}$, where $Z_{N,\beta}(\varphi)$ is defined in \eqref{ZNBF}. Then, one has that

	\begin{align*}
		\bbvar\big[N^{\frac{d-2}{4}}Z_{N,\beta}(\varphi)\big] \xrightarrow{N\to \infty}{} \cC_{\beta} \int_{0}^{1} \dd t \int_{\R^{d}\times \R^{d}}  \dd x \dd y \: \varphi(x)g_{\frac{2t}{d}}(x-y)\varphi(y) \, ,
	\end{align*}

	\noindent	where $\cC_{\beta}=\: \sigma^2(\beta) \: \E[e^{\lambda_2 (\beta) \mathcal{L}_\infty}]$, $\sigma^2(\beta)=e^{\lambda_2 (\beta)}-1$ and $g$ denotes the $d$-dimensional heat kernel.
\end{proposition}

For the proof of Proposition \ref{variancep}, we will need the following standard consequence of the local limit theorem, which we prove for completeness.
\begin{lemma} \label{conver}
	For any test function $\varphi \in C_c(\R^d)$ we have that
	\begin{align*}
		N^{\frac{d}{2}-1}  \sum_{n=1}^N \sum_{x,y \in \Z^d} \varphi_N(x,y) q_{2n}(x-y)  \xrightarrow{N \to \infty}{}
		\int_{0}^{1} \dd t \int_{\R^{d}\times \R^{d}}  \dd x \dd y \: \varphi(x)g_{\frac{2t}{d}}(x-y)\varphi(y) \, .
	\end{align*}
\end{lemma}

\begin{proof}
	Recall that by the local limit theorem for the $d$-dimensional simple random walk, see \cite{LL10}, one has that
	{$q_{2n}(x)=2\big(g_{\frac{2n}{d}}(x)+o(n^{-{\frac{d}{2}}})\big)\ind_{x \in \Z^{d}_{\text{even}}}$}, uniformly in $x \in \Z^d$, as $n\to \infty $, where $\Z^{d}_{\text{even}}:=\{x=(x_1,...,x_d) \in \Z^d: \: x_1+...+x_d \in 2\Z \}$. {The factor $2$ comes from the periodicity of the random walk. The kernel $g_{\frac{2n}{d}}(x)$ appears instead of $g_{2n}(x)$, because after $n$ steps the $d$-dimensional simple random walk $S_n$ has covariance matrix $\frac{n}{d}I$.}  Let us fix ${\theta} \in (0,1)$. {Let us also use the notation}
	\begin{align*}
		 & T_{{\theta},N}:=N^{\frac{d}{2}-1}  \sum_{n=1}^{{\theta} N} \sum_{x,y \in \Z^d} \varphi_N(x,y)\, q_{2n}(x-y) \notag      \\
		 & S_{{\theta},N}:=N^{\frac{d}{2}-1}  \sum_{n>{\theta} N}^{ N} \sum_{x,y \in \Z^d} \varphi_N(x,y)\, q_{2n}(x-y)       \, .
	\end{align*}
	Observe that if we bound $\varphi(\frac{y}{\sqrt{N}})$ {in $\varphi_N(x,y)$} by its supremum norm and use that $\sum_{z\in\Z^d} q_{2n}(z)=1$ we obtain that
	\begin{align*}
		{T_{{\theta},N} \leq  \frac{\norm{\varphi}_{\infty}}{N}  \sum_{n=1}^{{\theta} N} \sum_{x \in \Z^d} \varphi_N(x) \sum_{y \in \Z^d}q_{2n}(x-y) \leq  \frac{\norm{\varphi}_{\infty}}{N}  \sum_{n=1}^{{\theta} N} \sum_{x \in \Z^d} \varphi_N(x) \leq \norm{\varphi}_{\infty} \norm{\varphi}_{1} {\theta} \, .}
	\end{align*}
	On the other hand, by using the local limit theorem {and Riemann approximation one  obtains that}
	\begin{align*}
		S_{{\theta},N} \xrightarrow{N \to \infty }{}
		\int_{{\theta}}^{1} \dd t \int_{\R^{d}\times \R^{d}}  \dd x \dd y \: \varphi(x)g_{\frac{2t}{d}}(x-y)\varphi(y) \, .
	\end{align*}
	By combining those two facts and letting ${\theta} \to 0$, one  obtains the desired result.
\end{proof}

\noindent We are now ready to present the proof of Proposition \ref{variancep}.

\begin{proof}[Proof of Proposition \ref{variancep}]
	{Recalling} \eqref{chaos2}, one arrives into the following expression for the variance of $Z_{N,\beta}(\varphi)$, by using also the fact that terms of different degree in the chaos expansion are orthogonal in { $L^{2}(\bbP)$:}
	\begin{align*}
		{\bbvar\big[Z_{N,\beta}(\varphi)\big]} & =\sum_{k=1}^N \sigma^{2k} \sum_{1 \leq n_1 < \ldots < n_k \le N} \sum_{x,y \in \Z^d} \varphi_N(x,y)\,
		q_{2n_1}(x-y)\prod_{i=2}^k q_{2(n_i-n_{i-1})}(0) \, .
	\end{align*}
	We can factor out the $k=1$ term and change variables to obtain the expression:
	\begin{align} \label{factor}
		\sum_{n=1}^N \sigma^2\sum_{x,y \in \Z^d} \varphi_N(x,y)\,q_{2n}(x-y)  \Bigg ( 1+ \sum_{k=1}^{N-n} \sigma^{2k} \sum_{1\leq \ell_1 < ...<\ell_{k} \leq N-n} \prod_{i=1}^k q_{2(\ell_i-\ell_{i-1})}(0) \Bigg) \, ,
	\end{align}
	where by convention if  $n=N$ the sum on the rightmost parenthesis is equal to $1$.
	Furthermore, one can observe that the right parenthesis is exactly equal to
	$\E\big[e^{\lambda_2(\beta)\mathcal{L}_{N-n}}\big]$, where we recall that
	$\mathcal{L}_N:=\sum_{k=1}^{N}\ind_{S_{2k}=0}$ denotes the number of times a
	random walk returns to $0$ up to time $N$. Thus,
	\begin{align} \label{convul}
		\bbvar\big[N^{\frac{d-2}{4}}Z_{N,\beta}(\varphi)\big] =N^{\frac{d}{2}-1}  \sum_{n=1}^N \sigma^2\sum_{x,y \in \Z^d} \varphi_N(x,y)\,q_{2n}(x-y)  \E[e^{\lambda_2(\beta)\mathcal{L}_{N-n}}]\, .
	\end{align}

	The heuristic idea here is that{,} if in the expression~\eqref{convul} {we ignore} $n$ in the expectation,
	then the sum would factorise. Then, by noticing that
	$\E[e^{\lambda_2(\beta)\mathcal{L}_{N}}]$ converges and by using also
	Lemma \ref{conver}, we obtain the conclusion of Proposition
	\ref{variancep}. Let us justify this heuristic idea rigorously. We have that

	\begin{align} \label{err}
		\E\big[e^{\lambda_2(\beta) \mathcal{L}_{N-n}}\big]=\E\big[e^{\lambda_2(\beta) \mathcal{L}_{N}}\big]+\E\big[(e^{\lambda_2(\beta) \mathcal{L}_{N-n}}-e^{\lambda_2(\beta) \mathcal{L}_{N}})\ind_{\mathcal{L}_N>\mathcal{L}_{N-n}}\big] \, .
	\end{align}
	Also,
	\begin{align} \label{err2}
		\Big |\E\big[(e^{\lambda_2(\beta) \mathcal{L}_{N-n}}-e^{\lambda_2(\beta) \mathcal{L}_{N}})\ind_{\mathcal{L}_N>\mathcal{L}_{N-n}}\big] \Big| \le 2 \E\big[e^{\lambda_2(\beta) \mathcal{L}_{N}} \ind_{\mathcal{L}_N>\mathcal{L}_{N-n}}\big] \, ,
	\end{align}
	by triangle inequality and because $\mathcal{L}_N$ is non-decreasing. Using
		{Hölder} inequality we can further bound the error in~\eqref{err} as follows: We
	choose $p>1$ very close to 1, such that $p
		\lambda_2(\beta)<\log(\frac{1}{\pi_d})$, thus $\E[e^{p\lambda_2(\beta)
					\mathcal{L}_{N}}]<\infty$, for every $N \in \N$. This is only possible when
	$\beta$ is in the $L^2$-regime. Then, by {H\"{o}lder:}

	\begin{align*}
		\E\big[e^{\lambda_2(\beta) \mathcal{L}_{N}} \ind_{\mathcal{L}_N>\mathcal{L}_{N-n}}\big] \le \E\big[e^{p\lambda_2(\beta) \mathcal{L}_{N}}\big]^{\frac{1}{p}} \P\big({\mathcal{L}_N>\mathcal{L}_{N-n}}\big)^{\frac{1}{q}} \, .
	\end{align*}
	Hence,
	\begin{align*}
		\Big |\E\big[(e^{\lambda_2(\beta) \mathcal{L}_{N-n_1}}-e^{\lambda_2(\beta) \mathcal{L}_{N}})\ind_{\mathcal{L}_N>\mathcal{L}_{N-n}}\big] \Big| \le {c_{p,\beta}} \P\big({\mathcal{L}_N>\mathcal{L}_{N-n}}\big)^{\frac{1}{q}} \, ,
	\end{align*}
	where ${c_{p,\beta}}:= {2}\E\big[e^{p\lambda_2(\beta) \mathcal{L}_{\infty}}\big]^{\frac{1}{p}}<\infty$.\\

	\justify Now, we split the sum in~\eqref{convul} into two parts. Let $\theta \in
		(0,1)$. We distinguish two cases:
	\begin{itemize}
		\item  If $ n \le \theta N $, then $N-n \geq (1-\theta)N$. Thus,
		      \begin{align*}
			      \Big |\E\big[(e^{\lambda_2(\beta) \mathcal{L}_{N-n}}-e^{\lambda_2(\beta) \mathcal{L}_{N}})\ind_{\mathcal{L}_N>\mathcal{L}_{N-n}}\big] \Big| \le {c_{p,\beta}} {\P({\mathcal{L}_N>\mathcal{L}_{(1-\theta)N}})^{\frac{1}{q}}} \, ,
		      \end{align*}
		      since $\mathcal{L}_{N}$ is non-decreasing in $N$. We also have that
		      \begin{align*}
			      \P({\mathcal{L}_N>\mathcal{L}_{(1-\theta)N}})\leq \P\big(\exists\, n>(1-\theta)N: \, S_{2n}=0\big) \leq \sum^{\infty}_{n>(1-\theta)N} q_{2n}(0)\xrightarrow[N \to \infty]{} 0 \, ,
		      \end{align*}
		      since $\sum^{\infty}_{n=1} q_{2n}(0)< \infty$, because $d \geq 3$.
		      Therefore, in this case we obtain that, \\
		      \begin{align*}
			          & N^{\frac{d}{2}-1}  \sum_{n=1}^{\theta N} \sigma^2\sum_{x,y \in \Z^d} \varphi_N(x,y)\,q_{2n}(x-y)  \E[e^{\lambda_2(\beta)\mathcal{L}_{N-n}}] \notag            \\
			      =\: & N^{\frac{d}{2}-1}  \sum_{n=1}^{\theta N} \sigma^2\sum_{x,y \in \Z^d} \varphi_N(x,y)\,q_{2n}(x-y)  \Big(\E[e^{\lambda_2(\beta)\mathcal{L}_{N}}]+o(1) \Big)\, .
		      \end{align*}
		\item If $ n > \theta N $, we have that:
		      \begin{align*}
			          & N^{\frac{d}{2}-1}  \sum_{n>\theta N} \sigma^2\sum_{x,y \in \Z^d} \varphi_N(x,y)\,q_{2n}(x-y)  \E[e^{\lambda_2(\beta)\mathcal{L}_{N-n}}] \notag     \\
			      \le & \:   N^{\frac{d}{2}-1}  \sum_{n>\theta N} \sigma^2\sum_{x,y \in \Z^d} \varphi_N(x,y)\,q_{2n}(x-y) \E[e^{\lambda_2(\beta)\mathcal{L}_{\infty}}]\, .
		      \end{align*}

	\end{itemize}
	By combining the two cases above we get that, for every $\theta \in
		(0,1)$
	\begin{align*}
		\limsup_{N\rightarrow \infty} \bbvar[N^{\frac{d-2}{4}}Z_{N,\beta}(\varphi)] \leq \sigma^2 \int_{0}^{\theta} \dd t \int_{\R^{d}\times \R^{d}}  \dd x \dd y \: \varphi(x)g_{\frac{2t}{d}}(x-y)\varphi(y) \E[e^{\lambda_2 (\beta) \mathcal{L}_\infty}] + k(\theta) \, ,
	\end{align*}
	where
	\begin{align*}
		k(\theta)\le {\E[e^{\lambda_2 (\beta) \mathcal{L}_\infty}]} \: \sigma^2 \int_{\theta}^{1} \dd t \int_{\R^{d}\times \R^{d}} \dd x \dd y \: \varphi(x)g_{\frac{2t}{d}}(x-y)\varphi(y) \, ,
	\end{align*}
	\noindent and
	\begin{align*}
		\liminf_{N\rightarrow \infty} \bbvar[N^{\frac{d-2}{4}}Z_{N,\beta}(\varphi)] \geq  \sigma^2 \int_{0}^{\theta} \dd t \int_{\R^{d}\times \R^{d}}  \dd x \dd y \: \varphi(x)g_{\frac{2t}{d}}(x-y)\varphi(y) \E[e^{\lambda_2 (\beta) \mathcal{L}_\infty}] \, .
	\end{align*}
	It is clear that $k(\theta) \longrightarrow 0$ as $\theta \longrightarrow
		1$, hence we obtain the desired result.
\end{proof}

We proceed towards the proof of the Central Limit Theorem {for} the sequence $
	\big(Z_{N,\beta}(\varphi) \big)_{N \geq 1}$ of the averaged partition functions. In order to determine the limiting distribution of the sequence
$\big(N^{\frac{d-2}{4}} Z_{N,\beta}(\varphi)\big)_{N\geq 1}$, we use the Fourth Moment Theorem, {see \cite{dJ87, NP05, NPR10, CSZ17b}}.\vspace{0.2mm} The strategy we deploy is the following: First, we show that it suffices
to consider a large $M \in \N$ and work with {a truncated version of the partition function,}
namely
\begin{align}
	Z^{\ms{\leq}  M}_{N,\beta}(\varphi):= \sum_{k=1}^M \sigma^k \sumtwo{1 \le n_1 < \ldots < n_k \le N}{z_1, \ldots, z_k \in \Z^d} \Bigg ( \sum_{x \in \Z^d} \varphi_N(x)\, q_{n_1}(z_1-x)\Bigg) \prod_{i=2}^k q_{n_i-n_{i-1}}(z_i-z_{i-1}) \prod_{i=1}^k \eta_{n_i,z_i} \label{trchaos} \, .
\end{align}

\noindent To do this it is enough to show that for any $\epsilon>0$ we can
choose a large $M=M(\epsilon)$ such that
$N^{\frac{d-2}{4}}Z^{\ms{\leq}  M}_{N,\beta}(\varphi)$ and
$N^{\frac{d-2}{4}}Z_{N,\beta}(\varphi)$ are $\epsilon$-close in $L^2(\bbP)$,
uniformly {for} $N \in \N$ {large}. Then, by using the Fourth Moment Theorem and the Crámer-Wold device, we show that the random vector $ N^{\frac{d-2}{4}}\big(Z^{(1)}_{N,\beta}(\varphi),..., Z^{(M)}_{N,\beta}(\varphi)\big)$ converges in distribution to a centered Gaussian random vector. This allows us to conclude that
the limiting distribution of $N^{\frac{d-2}{4}}Z^{\ms{\leq}  M}_{N,\beta}(\varphi)$ is a
centered Gaussian. After removing the truncation in $M$, we obtain the desired
result for $N^{\frac{d-2}{4}}Z_{N,\beta}(\varphi)$, namely Theorem
\ref{gaussianity}. \\

We begin by proving that we can approximate $Z_{N,\beta}(\varphi)$ in $L^2(\bbP)$, uniformly for large enough $N$,
by $Z^{\ms{\leq}  M}_{N,\beta}(\varphi)$ for some large $M \in \N$.

\begin{lemma} \label{truncation}
	{For every $\epsilon >0$,} there exists {$M_0 \in \N$, such that for all $M>M_0$}
	\begin{align*}
		\limsup_{N \to \infty} \norm{N^{\frac{d-2}{4}}Z_{N,\beta}(\varphi)- N^{\frac{d-2}{4}}Z^{\ms{\leq}  M}_{N,\beta}(\varphi)}_{L^2(\bbP)} \le \epsilon \, .
	\end{align*}
\end{lemma}

\begin{proof}
	{Consider $\epsilon >0$.} One has that
	\begin{align*}
		  & Z_{N,\beta}(\varphi)- Z^{\ms{\leq}  M}_{N,\beta}(\varphi)              \notag \\
		= & \sum_{k>M}^N \sigma^k \sumtwo{1 \le n_1 < \ldots < n_k \le N}
		{z_1, \ldots, z_k \in \Z^d} \Bigg ( \sum_{x \in \Z^d} \varphi_N(x)\, q_{n_1}(z_1-x)\Bigg) \prod_{i=2}^k q_{n_i-n_{i-1}}(z_i-z_{i-1}) \prod_{i=1}^k \eta_{n_i,z_i} \, .
	\end{align*}

	By an analogous computation as in Proposition \ref{variancep} we have that
	\begin{align*}
		     & \norm{N^{\frac{d-2}{4}}Z_{N,\beta}(\varphi)- N^{\frac{d-2}{4}}Z^{\ms{\leq}  M}_{N,\beta}(\varphi)}^2_{L^2(\bbP)} \notag                                                                                                               \\
		\leq & \,  N^{\frac{d}{2}-1} \sum_{n=1}^N \sigma^2\sum_{x,y \in \Z^d} \varphi_N(x,y)\,q_{2n}(x-y)  \Big ( \sum_{k \geq M }^{N-n} \sigma^{2k} \sum_{1 \le \ell_1 <...<\ell_{k} \leq N-n}    \prod_{i=1}^k q_{2(\ell_i-\ell_{i-1})}(0) \Big)   \\
		\leq & \,  {N^{\frac{d}{2}-1} \sum_{n=1}^N \sigma^2\sum_{x,y \in \Z^d} \varphi_N(x,y)\,q_{2n}(x-y)  \Big ( \sum_{k \geq M }^{N} \sigma^{2k} \sum_{1 \le \ell_1 <...<\ell_{k} \leq N}    \prod_{i=1}^k q_{2(\ell_i-\ell_{i-1})}(0) \Big)}\, .
	\end{align*}
	By {Lemma} \ref{conver}  we have that
	\begin{align*}
		{N^{\frac{d}{2}-1} \sum_{n=1}^N \sigma^2\sum_{x,y \in \Z^d} \varphi_N(x,y)\,q_{2n}(x-y) \xrightarrow[N \to \infty]{} \int_{0}^{1} \dd t \int_{\R^{d}\times \R^{d}}  \dd x \dd y \: \varphi(x)g_{\frac{2t}{d}}(x-y)\varphi(y)} \, .
	\end{align*}
	The sum in the rightmost parenthesis can be bounded by
	\begin{align*}
		\Big ( \sum_{k \geq M }^{N} \sigma^{2k} \sum_{1 \le \ell_1 <...<\ell_{k} \leq N}   \prod_{i=1}^k q_{2(\ell_i-\ell_{i-1})}(0) \Big) \leq \sum_{k\geq M}^{N} \sigma^{2k} R^k_{N}\leq \sum_{k\geq M}^{N} \sigma^{2k} {R_{\infty}^k} \leq \sum_{k\geq M}^{\infty} \sigma^{2k} R_{\infty}^k \, ,
	\end{align*}
	where $R_N=\sum_{k=1}^{N} q_{2n}(0)$ is the expected number of visits to zero
	before time $N$ of the simple random walk and $R_{\infty}=\lim_{N\to \infty} R_N=\sum_{k=1}^{\infty} q_{2n}(0)$. {Since $\beta$ is in the $L^2$-regime, the
	series $\sum_{k\geq 1}^{\infty} {\sigma(\beta)^{2k}} R_{\infty}^k$ is convergent. Therefore, we have that  }
	\begin{align*}
		\sum_{k\geq M}^{\infty} \sigma^{2k} R_{\infty}^k \xrightarrow[M \to \infty]{} 0 \, .
	\end{align*}
	{Therefore, we conclude that if we take $M$ to be sufficiently large we have that}
	\begin{align*}
		\norm{N^{\frac{d-2}{4}}Z_{N,\beta}(\varphi)-
		N^{\frac{d-2}{4}}Z^{\ms{\leq}  M}_{N,\beta}(\varphi)}_{L^2(\bbP)} \le \epsilon \, ,
	\end{align*}
	uniformly for all large enough $ N \in \N$, hence there exists $M_0 \in \N$, so that for $M>M_0$:
	\begin{align*}
		\limsup_{N \to \infty} \norm{N^{\frac{d-2}{4}}Z_{N,\beta}(\varphi)- N^{\frac{d-2}{4}}Z^{\ms{\leq}  M}_{N,\beta}(\varphi)}_{L^2(\bbP)} \le \epsilon \, .
	\end{align*}

\end{proof}

\noindent We proceed by showing that for any $M \in \N$, the random vector
$ N^{\frac{d-2}{4}}\big(Z^{(1)}_{N,\beta}(\varphi),..., Z^{(M)}_{N,\beta}(\varphi)\big)$
converges in distribution to a Gaussian vector. To do this we employ the Cram\'{e}r-Wold device. Namely, we  prove that for any $M$-tuple of real numbers $(t_1,...,t_{M})$ the linear combination $N^{\frac{d-2}{4}} \sum_{k=1}^M t_k Z^{(k)}_{N,\beta}(\varphi)$ converges in distribution to a Gaussian random variable.

\begin{proposition} \label{lincomb}
	For all $M \in \N$ and $(t_1,...,t_M)\in \R^M$, $N^{\frac{d-2}{4}} \sum_{k=1}^M t_k
		Z^{(k)}_{N,\beta}(\varphi)$ converges in distribution to a Gaussian random variable with mean zero and variance equal to
	\begin{align*}
		\sum_{k=1}^M t^2_k \,\mathcal{C}^{(k)}_{\beta} \int_0^1 \dd t \,\int_{\R^{2d} } \dd x \, \dd y\, \varphi(x) g_{\frac{2t}{d}}(x-y) \varphi(y) \, ,
	\end{align*}
	where $ \displaystyle {C^{(k)}_{\beta}= \sigma(\beta)^{2k} \sumtwo{1\leq \ell_1 <...<\ell_{k-1}}{\ell_0:=0} \prod_{i=1}^{k-1} q_{2(\ell_i-\ell_{i-1})}(0)}$ for $k>1$ and $C^{(1)}_{\beta}=\sigma(\beta)^{2}$.
\end{proposition}
\begin{proof}
	{We start by introducing some shorthand notation} that is going to be useful for a concise presentation of the rest of the proof.
	For any $u \in \Z^d$, $\tau^{\ms{(k)}}_{u}$ will denote a time-increasing sequence of $(k+1)$ space-time points
	$(n_i,z_i)_{0 \leq i \leq k}\subset \N \times \Z^d$  with a starting point $(n_0,z_0):=(0,u)$.
	We will use the convention that for two sequences $\tau_x^{(k)}=(n_i,z_i)_{0\leq i\leq k}$ and
	$\tau_y^{(\ell)}=(m_i,w_i)_{0\leq i\leq \ell}$, the equality $\tau_x^{(k)}=\tau_y^{(\ell)}$ means that $k=\ell$
	and $(n_i,z_i)=(m_i,w_i)$ for $i=1,...,k$, that is for all points in the sequences $\tau_x^{(k)}$ and $\tau_y^{(\ell)}$ except the starting ones.

	Given a sequence $\tau^{\ms{(k)}}_{u}= (n_i,z_i)_{1\leq i \leq k}$, we will use the following notation
	\begin{align*}
		q(\tau^{\ms{(k)}}_u)\,{:=}\,q_{n_1}(z_1-u)\prod_{i=2}^k q_{n_i-n_{i-1}} (z_i-z_{i-1}) \hspace{0,5cm}\text{   and    }\hspace{0,5cm}\eta(\tau^{\ms{(k)}}_u)\,{:=}\,\prod_{i=1}^k \eta_{n_i,z_i} \, .
	\end{align*}
	Furthermore, we recall from \eqref{phinotation}, that for a finite set $\{x_1,... ,x_k\} \subset \Z^d$ we use the notation
	\begin{align}
		\varphi_N(x_1,...,x_k):=\prod_{u \in \{x_1,...,x_k\} }\frac{\varphi\big(\frac{u}{\sqrt{N}}\big)}{N^\frac{d}{2}} \, .
	\end{align}
	{We start by deriving} the limiting variance of $N^{\frac{d-2}{4}}\sum_{k=1}^M t_k \,Z^{(k)}_{N,\beta}(\varphi)$. We have that
	\begin{align*}
		{\bbvar}\bigg(N^{\frac{d-2}{4}}\sum_{k=1}^M t_k Z^{(k)}_{N,\beta}(\varphi)\bigg)={\sum_{k=1}^M t_k^2 \,  N^{\frac{d}{2}-1}\,  \bbE\Big[ \big(Z^{(k)}_{N,\beta}(\varphi)\big)^2\Big]} \, ,
	\end{align*}
	because for every $k \geq 1$, $\bbE \Big[Z^{(k)}_{N,\beta}(\varphi)\Big]=0$ and if $1\leq  k < \ell $, we have that $\bbE \Big[Z^{(k)}_{N,\beta}(\varphi)\, Z^{(\ell)}_{N,\beta}(\varphi) \Big]=0$, see \eqref{fixedchaos}. One can follow the steps of the proof of Proposition \ref{variancep}, to see that
	\begin{align*}
		\lim_{N \to \infty} N^{\frac{d}{2}-1}\, \bbE\Big[ \big(Z^{(k)}_{N,\beta}(\varphi)\big)^2\Big]=\mathcal{C}^{(k)}_{\beta} \int_0^1 \dd t \,\int_{\R^{2d} } \dd x \, \dd y\, \varphi(x) g_{\frac{2t}{d}}(x-y) \varphi(y) \, ,
	\end{align*}
	where $ \displaystyle C^{(k)}_{\beta} \,{: =}\, \sigma(\beta)^{2k} \sumtwo{1\leq \ell_1 <...<\ell_{k-1}}{\ell_0:=0} \prod_{i=1}^{k-1} q_{2(\ell_i-\ell_{i-1})}(0)$ for $k>1$ and $C^{(1)}_{\beta}\,{: =}\,\sigma(\beta)^2$. \\

	In order to show that $N^{\frac{d-2}{4}}\sum_{k=1}^M t_k\, Z^{(k)}_{N,\beta}(\varphi)$ converges in distribution to a Gaussian limit we will employ the Fourth Moment Theorem,  { which states that a sequence of random variables in a fixed Wiener chaos or multilinear polynomials of finite degree converge to a Gaussian random variable if the 4th moment converges to three times the square of the variance}, see \cite{dJ87, NP05, NPR10, CSZ17b} for more details. Namely, we will show that as $N \to \infty$,
	\begin{align*}
		\vspace{0.2mm}
		\bbE \Bigg[\bigg( N^{\frac{d-2}{4}} \sum_{k=1}^M t_k Z^{(k)}_{N,\beta}(\varphi)\bigg)^4 \Bigg] = 3 \bbvar \Bigg[ N^{\frac{d-2}{4}} \sum_{k=1}^M t_k Z^{(k)}_{N,\beta}(\varphi) \Bigg]^2 +o(1) \, .
	\end{align*}
	that is, the fourth moment of $ N^{\frac{d-2}{4}} \sum_{k=1}^M t_k Z^{(k)}_{N,\beta}(\varphi)$ converges to 3 times its variance, squared.
	\noindent In view of the chaos expansion \eqref{fixedchaos} we have that
	\begin{align}
		\bbE \Bigg[\bigg( N^{\frac{d-2}{4}} \sum_{k=1}^M t_k Z^{(k)}_{N,\beta}(\varphi)\bigg)^4 \Bigg]= &
		N^{d-2} \sum_{1\leq \sfa, \sfb, \sfc, \sfd \leq M } t_\sfa t_\sfb t_\sfc t_\sfd \: \bbE \Big[Z^{(\sfa)}_{N,\beta}(\varphi) Z^{(\sfb)}_{N,\beta}(\varphi) Z^{(\sfc)}_{N,\beta}(\varphi) Z^{(\sfd)}_{N,\beta}(\varphi)\Big] \notag                                                                                                                                                                                                                                       \\
		=                                                                                               & N^{d-2} \sum_{1\leq \sfa, \sfb, \sfc, \sfd \leq M} t_\sfa t_\sfb t_\sfc t_\sfd \, \,
		{ \sigma^{ \sfa+ \sfb+ \sfc+ \sfd}}\sum_{x,y,z,w \,\in \Z^d} \varphi_N(x,y,z,w) \notag                                                                                                                                                                                                                                                                                                                                                                                 \\
		\times                                                                                          & \sum_{\tau^{\ms{(\sfa)}}_x, \tau^{\ms{(\sfb)}}_y, \tau^{\ms{(\sfc)}}_z, \tau^{\ms{(\sfd)}}_w} \, \, \prodtwo{(u,\sfs)\, \in  \{(x,\sfa), (y,\sfb),}{(z,\sfc), (w,\sfd)\}} q(\tau^{\ms(\sfs)}_u)  \,  \bbE\Big[\prodtwo{(u,\sfs) \, \in \{(x,\sfa),(y,\sfb),}{(z,\sfc),(w,\sfd)\}}  \eta(\tau^{\ms(\sfs)}_u)\Big ]\, .                                 \label{4thmom}
	\end{align}

	Since $M$ is finite, we can fix a quadruple $(\sfa,\sfb,\sfc,\sfd)$ and deal with the rest of the sum which varies as $N \to \infty$. Thus, we will focus on the sum
	\begin{align} \label{secondform}
		N^{d-2}  \sum_{x,y,z,w \,\in \Z^d} \varphi_N(x,y,z,w) \, \,	{ \sigma^{ \sfa+ \sfb+ \sfc+ \sfd}} \hspace{-0.3cm}
		\sum_{\tau^{\ms{(\sfa)}}_x, \tau^{\ms{(\sfb)}}_y, \tau^{\ms{(\sfc)}}_z, \tau^{\ms{(\sfd)}}_w} \, \, \prodtwo{(u,\sfs)\, \in  \{(x,\sfa), (y,\sfb),}{(z,\sfc), (w,\sfd)\}} q(\tau^{\ms(\sfs)}_u)  \,  \bbE\Big[\prodtwo{(u,\sfs) \, \in \{(x,\sfa),(y,\sfb),}{(z,\sfc),(w,\sfd)\}}  \eta(\tau^{\ms(\sfs)}_u)\Big ] \, ,
	\end{align}
	instead of \eqref{4thmom}.
	We note that the expectation
	\begin{align} \label{matchingconstr}
		\bbE\Big[\prodtwo{(u,\sfs) \, \in \{(x,\sfa),(y,\sfb),}{(z,\sfc),(w,\sfd)\}}  \eta(\tau^{\ms(\sfs)}_u)\Big ] \, ,
	\end{align} is non-zero only if the random variables $\eta$ appearing in the product,
	are matched to each other. This is because, if a random variable $\eta$ stands alone in the expectation \eqref{matchingconstr}, then due to independence and the fact that every $\eta$ has mean zero, the expectation is trivially zero.
	The possible matchings among the $\eta$ variables can be double, triple or quadruple.  We {cannot} have more than quadruple matchings, because points in a sequence $\tau^{\ms(\sfs)}_u$ are strictly increasing in time, thus they {cannot} match with each other.

	We will show that when $N \to \infty$, only one type of matchings  contributes to \eqref{secondform} and hence also to \eqref{4thmom}. Specifically, the only configuration that contributes{, asymptotically,} is the one where four random walk paths meet in pairs without switching their pair. In terms of the sequences  $\tau^{\ms{(\sfa)}}_x, \tau^{\ms{(\sfb)}}_y, \tau^{\ms{(\sfc)}}_z, \tau^{\ms{(\sfd)}}_w$, this condition translates to that $\tau^{\ms{(\sfa)}}_x, \tau^{\ms{(\sfb)}}_y, \tau^{\ms{(\sfc)}}_z, \tau^{\ms{(\sfd)}}_w$ must be pairwise equal {to two sequences which do not share any common points. For the rest of the proof, when we say pairwise equal we will always mean pairwise equal to two distinct sequences which do not share any common points.}
	We will first focus on sequences $\tau^{\ms{(\sfa)}}_x, \tau^{\ms{(\sfb)}}_y, \tau^{\ms{(\sfc)}}_z, \tau^{\ms{(\sfd)}}_w$, which do not satisfy this condition { and show that their contribution is negligible.}

		{Consider sequences $ \tau^{\ms{(\sfa)}}_x , \tau^{\ms{(\sfb)}}_y , \tau^{\ms{(\sfc)}}_z , \tau^{\ms{(\sfd)}}_w$ and let $\tau:=\tau^{\ms{(\sfa)}}_x \cup \tau^{\ms{(\sfb)}}_y \cup \tau^{\ms{(\sfc)}}_z \cup \tau^{\ms{(\sfd)}}_w=(f_i,h_i)_{1 \leq i \leq |\tau|}$ with $f_1\leq f_2 \leq\dots\leq f_{|\tau|}.$ Let $1\leq i_\star \leq |\tau|$ be the {first} index, so that for all $(u,\sfs)\, \in  \{(x,\sfa), (y,\sfb),(z,\sfc), (w,\sfd)\}$, the sequences $\tau_u^{\ms{(\sfs)}}\cap \big([1,f_{i_\star})\times \Z^d \big)$ are pairwise equal, {but this fails to hold for $\tau_u^{\ms{(\sfs)}}\cap \big([1,f_{i_\star}]\times \Z^d \big)$,} see figures \ref{fig:Type1}, \ref{fig:Type2}}.

	If there does not exist such index $1 \leq i_\star \leq |\tau|$, then the four random walks meet pairwise without switching their pair.
	For this kind of sequences $\tau^{\ms{(\sfa)}}_x, \tau^{\ms{(\sfb)}}_y, \tau^{\ms{(\sfc)}}_z,\tau^{\ms{(\sfd)}}_w$, {for which $i_\star$ does not exist}, we have that $\tau^{\ms{(\sfa)}}_x, \tau^{\ms{(\sfb)}}_y, \tau^{\ms{(\sfc)}}_z,\tau^{\ms{(\sfd)}}_w$ have to be pairwise equal. Their contribution to \eqref{4thmom} is
	\begin{align} \label{4thmom_simpl}
		3\, N^{d-2} \sum_{1\leq \sfa,\sfb\leq M} t^2_\sfa t^2_\sfb 	\, \,{ \sigma^{2( \sfa+ \sfb)}} \sum_{x,y,z,w \in \Z^d}   \varphi_N(x,y,z,w)  \sumtwo{\tau^{\ms{(\sfa)}}_x=\tau^{\ms{(\sfa)}}_y,\tau^{\ms{(\sfb)}}_w=\tau^{\ms{(\sfb)}}_z}{\tau^{\ms{(\sfa)}}_x\cap \tau^{\ms{(\sfb)}}_z=\eset}   q(\tau^{\ms{(\sfa)}}_x)\,q(\tau^{\ms{(\sfa)}}_y)\, q(\tau^{\ms{(\sfb)}}_w)\,q(\tau^{\ms{(\sfb)}}_z) \, .
	\end{align}
	The factor $3$ accounts for the number of ways we can pair the sequences $\tau^{\ms{(\sfa)}}_x, \tau^{\ms{(\sfb)}}_y, \tau^{\ms{(\sfc)}}_z,\tau^{\ms{(\sfd)}}_w$.
	The sum in \eqref{4thmom_simpl} equals $ 3 N^{d-2} \: \bbE\Big[\big(\sum_{k=1}^M t_k Z^{(k)}_{N,\beta}(\varphi)\big)^2\Big]^2+o(1)$ as $N\to \infty$. The $o(1)$ factor is a consequence of the restriction  $\tau^{\ms{(\sfa)}}_x\cap \tau^{\ms{(\sfb)}}_z\neq \eset$ in \eqref{4thmom_simpl}, which  excludes configurations of the four random walk paths such that four walks meet simultaneously at a single point. It is part of the proof below to show that the contribution of these configurations is negligible in the large $N$ limit.

	Hence, for now we can focus on the {cases} for which such a point $(f_{i_\star},h_{i_\star})$ exists {and show that their contribution is negligible for \eqref{4thmom}}.

	\begin{figure}  

		\centering
		\tikzset{every picture/.style={line width=0.50pt}} 

		\begin{tikzpicture}[x=0.35pt,y=0.4pt,yscale=-1,xscale=1]

			\draw    (100,62.25) -- (100,333.5) ;
			\draw  [fill={rgb, 255:red, 0; green, 0; blue, 0 }  ,fill opacity=1 ] (96,98.38) .. controls (96,96.23) and (97.73,94.5) .. (99.88,94.5) .. controls (102.02,94.5) and (103.75,96.23) .. (103.75,98.38) .. controls (103.75,100.52) and (102.02,102.25) .. (99.88,102.25) .. controls (97.73,102.25) and (96,100.52) .. (96,98.38) -- cycle ;
			\draw  [fill={rgb, 255:red, 0; green, 0; blue, 0 }  ,fill opacity=1 ] (96,148.38) .. controls (96,146.23) and (97.73,144.5) .. (99.88,144.5) .. controls (102.02,144.5) and (103.75,146.23) .. (103.75,148.38) .. controls (103.75,150.52) and (102.02,152.25) .. (99.88,152.25) .. controls (97.73,152.25) and (96,150.52) .. (96,148.38) -- cycle ;
			\draw    (99.88,98.38) .. controls (155.75,98.5) and (158.75,105.5) .. (185.75,125.5) ;
			\draw    (99.88,148.38) .. controls (144.75,146.5) and (152.75,149.5) .. (185.75,125.5) ;
			\draw [color={rgb, 255:red, 0; green, 0; blue, 0}  ,draw opacity=1 ]   (99.88,218.38) .. controls (139.75,218.5) and (178.63,245.63) .. (190.75,257.5) ;
			\draw [color={rgb, 255:red, 0; green, 0; blue, 0 }  ,draw opacity=1 ]   (99.88,298.38) .. controls (151.75,293.5) and (174.63,269.63) .. (190.75,257.5) ;
			\draw [color={rgb, 255:red, 0; green, 0; blue, 0 }  ,draw opacity=1 ]   (314.6,122.08) .. controls (332.36,138.33) and (383.29,149.69) .. (392.17,184.63) ;
			\draw [color={rgb, 255:red, 0; green, 0; blue, 0 }  ,draw opacity=1 ]   (319.55,254.25) .. controls (353.8,218.48) and (379.72,230.66) .. (392.17,184.63) ;
			\draw    (314.6,122.08) .. controls (354.81,92.36) and (459.13,148.88) .. (484.91,180.4) ;
			\draw    (319.55,254.25) .. controls (369.56,278.26) and (444.7,210.12) .. (484.91,180.4) ;
			\draw  [fill={rgb, 255:red, 0; green, 0; blue, 0 }  ,fill opacity=1 ] (478.32,180.23) .. controls (478.34,178.09) and (480.09,176.36) .. (482.23,176.38) .. controls (484.37,176.39) and (486.09,178.14) .. (486.07,180.28) .. controls (486.06,182.42) and (484.31,184.14) .. (482.17,184.13) .. controls (480.03,184.11) and (478.31,182.37) .. (478.32,180.23) -- cycle ;
			\draw    (391.17,184.62) .. controls (407.6,176.54) and (418.28,174.61) .. (437.69,164.07) ;
			\draw  [dash pattern={on 0.84pt off 2.51pt}]  (437.69,164.07) .. controls (449.76,154.15) and (426.3,172.66) .. (466.51,142.94) ;
			\draw    (466.51,142.94) .. controls (506.72,113.22) and (563.63,104.26) .. (604.23,113.87) ;
			\draw    (392.17,184.63) .. controls (404.79,197.84) and (404.03,210.51) .. (419.97,219.28) ;
			\draw  [dash pattern={on 0.84pt off 2.51pt}]  (419.97,219.28) .. controls (453.88,232.19) and (418.63,220.6) .. (455.88,232.2) ;
			\draw    (453.89,231.53) .. controls (470.95,232.51) and (577.73,256.7) .. (606.05,210.22) ;
			\draw    (482.2,180.25) .. controls (493.75,174.25) and (511.75,163.75) .. (541.25,179.25) ;
			\draw    (482.2,180.25) .. controls (488.25,188.75) and (531.25,192.75) .. (541.25,179.25) ;
			\draw  [fill={rgb, 255:red, 0; green, 0; blue, 0 }  ,fill opacity=1 ] (537.38,179.22) .. controls (537.39,177.08) and (539.14,175.36) .. (541.28,175.38) .. controls (543.42,175.39) and (545.14,177.14) .. (545.12,179.28) .. controls (545.11,181.42) and (543.36,183.14) .. (541.22,183.12) .. controls (539.08,183.11) and (537.36,181.36) .. (537.38,179.22) -- cycle ;
			\draw    (541.25,179.25) .. controls (580.03,162.98) and (592.83,144.4) .. (604.39,111.81) ;
			\draw    (541.25,179.25) .. controls (581.24,180.99) and (587.81,194.37) .. (605.71,209.16) ;
			\draw  [fill={rgb, 255:red, 0; green, 0; blue, 0 }  ,fill opacity=1 ] (600.36,113.85) .. controls (600.37,111.71) and (602.12,109.98) .. (604.26,110) .. controls (606.4,110.01) and (608.12,111.76) .. (608.11,113.9) .. controls (608.09,116.04) and (606.34,117.76) .. (604.2,117.75) .. controls (602.06,117.73) and (600.34,115.99) .. (600.36,113.85) -- cycle ;
			\draw  [fill={rgb, 255:red, 0; green, 0; blue, 0 }  ,fill opacity=1 ] (602.18,210.19) .. controls (602.19,208.05) and (603.94,206.33) .. (606.08,206.34) .. controls (608.22,206.36) and (609.94,208.1) .. (609.93,210.24) .. controls (609.91,212.38) and (608.16,214.11) .. (606.02,214.09) .. controls (603.88,214.08) and (602.16,212.33) .. (602.18,210.19) -- cycle ;
			\draw    (604.23,113.87) .. controls (608.54,117.15) and (593.2,108) .. (637.62,105.51) ;
			\draw    (604.23,113.87) .. controls (609.91,120.92) and (617.5,122.57) .. (637.9,123.11) ;
			\draw    (606.05,210.22) .. controls (610.36,213.5) and (595.02,204.35) .. (639.44,201.86) ;
			\draw    (606.05,210.22) .. controls (611.73,217.26) and (619.32,218.92) .. (639.72,219.46) ;
			\draw  [fill={rgb, 255:red, 0; green, 0; blue, 0 }  ,fill opacity=1 ] (96,218.38) .. controls (96,216.23) and (97.73,214.5) .. (99.88,214.5) .. controls (102.02,214.5) and (103.75,216.23) .. (103.75,218.38) .. controls (103.75,220.52) and (102.02,222.25) .. (99.88,222.25) .. controls (97.73,222.25) and (96,220.52) .. (96,218.38) -- cycle ;
			\draw  [fill={rgb, 255:red, 0; green, 0; blue, 0 }  ,fill opacity=1 ] (96,298.38) .. controls (96,296.23) and (97.73,294.5) .. (99.88,294.5) .. controls (102.02,294.5) and (103.75,296.23) .. (103.75,298.38) .. controls (103.75,300.52) and (102.02,302.25) .. (99.88,302.25) .. controls (97.73,302.25) and (96,300.52) .. (96,298.38) -- cycle ;
			\draw  [fill={rgb, 255:red, 0; green, 0; blue, 0 }  ,fill opacity=1 ] (186.88,257.5) .. controls (186.88,255.36) and (188.61,253.63) .. (190.75,253.63) .. controls (192.89,253.63) and (194.63,255.36) .. (194.63,257.5) .. controls (194.63,259.64) and (192.89,261.38) .. (190.75,261.38) .. controls (188.61,261.38) and (186.88,259.64) .. (186.88,257.5) -- cycle ;
			\draw  [fill={rgb, 255:red, 0; green, 0; blue, 0 }  ,fill opacity=1 ] (181,125.38) .. controls (181,123.23) and (182.73,121.5) .. (184.88,121.5) .. controls (187.02,121.5) and (188.75,123.23) .. (188.75,125.38) .. controls (188.75,127.52) and (187.02,129.25) .. (184.88,129.25) .. controls (182.73,129.25) and (181,127.52) .. (181,125.38) -- cycle ;
			\draw  [fill={rgb, 255:red, 0; green, 0; blue, 0 }  ,fill opacity=1 ] (387.29,184.59) .. controls (387.31,182.45) and (389.05,180.73) .. (391.19,180.74) .. controls (393.33,180.76) and (395.06,182.51) .. (395.04,184.65) .. controls (395.03,186.79) and (393.28,188.51) .. (391.14,188.49) .. controls (389,188.48) and (387.28,186.73) .. (387.29,184.59) -- cycle ;
			\draw  (184.88,125.38) .. controls (189.18,128.66) and (173.85,119.5) .. (218.27,117.01) ;
			\draw   (185.75,125.5) .. controls (191.43,132.55) and (199.02,134.2) .. (219.42,134.74) ;
			\draw   (190.75,257.5) .. controls (195.05,260.78) and (179.72,251.63) .. (224.14,249.14) ;
			\draw    (190.75,257.5) .. controls (196.43,264.55) and (204.02,266.2) .. (224.42,266.74) ;
			\draw    (314.84,122.66) .. controls (310.36,119.63) and (326.19,127.89) .. (281.98,132.91) ;
			\draw   (314.84,122.66) .. controls (308.77,115.95) and (301.09,114.73) .. (280.7,115.35) ;
			\draw   (319.55,254.25) .. controls (315.07,251.22) and (330.9,259.48) .. (286.69,264.5) ;
			\draw    (320.04,255.46) .. controls (313.97,248.75) and (306.29,247.53) .. (285.9,248.15) ;
			\draw  [fill={rgb, 255:red, 0; green, 0; blue, 0 }  ,fill opacity=1 ] (316.17,255.46) .. controls (316.17,253.32) and (317.9,251.58) .. (320.04,251.58) .. controls (322.18,251.58) and (323.92,253.32) .. (323.92,255.46) .. controls (323.92,257.6) and (322.18,259.33) .. (320.04,259.33) .. controls (317.9,259.33) and (316.17,257.6) .. (316.17,255.46) -- cycle ;
			\draw  [fill={rgb, 255:red, 0; green, 0; blue, 0 }  ,fill opacity=1 ] (310.73,122.08) .. controls (310.73,119.94) and (312.46,118.21) .. (314.6,118.21) .. controls (316.74,118.21) and (318.48,119.94) .. (318.48,122.08) .. controls (318.48,124.22) and (316.74,125.96) .. (314.6,125.96) .. controls (312.46,125.96) and (310.73,124.22) .. (310.73,122.08) -- cycle ;

			\draw (100,344.25) node  [font=\tiny]  {$\mathbb{Z}^{d}$};
			\draw (649.33,113) node  [rotate=-0.4]  {\scalebox{0.75}{$...$}};
			\draw (649.15,210) node  [rotate=-0.4]  {\scalebox{0.75}{$...$}};
			\draw (97.88,101.78) node [anchor=north east] [inner sep=0.75pt]  [font=\tiny]  {\scalebox{0.8}{$( 0,x)$}};
			\draw (97.88,151.78) node [anchor=north east] [inner sep=0.75pt]  [font=\tiny]  {\scalebox{0.8}{$( 0,y)$}};
			\draw (97.88,221.78) node [anchor=north east] [inner sep=0.75pt]  [font=\tiny]  {\scalebox{0.8}{$( 0,z)$}};
			\draw (97.88,301.77) node [anchor=north east] [inner sep=0.75pt]  [font=\tiny]  {\scalebox{0.8}{$( 0,w)$}};
			\draw (385.29,184.58) node [anchor=east] [inner sep=0.75pt]  [font=\tiny]  {\scalebox{0.75}{$( {f_{i_{\star }}} ,{h_{i_{\star }}})$}};
			\draw (315.8,112.81) node [anchor=south] [inner sep=0.75pt]  [font=\tiny] {\scalebox{0.75}{$(\bar{f}_{a} ,\bar{h}_{a})$}};
			\draw (320.04,267.53) node [anchor=north] [inner sep=0.75pt]  [font=\tiny]  {\scalebox{0.75}{$(\ubar{f}_{b} ,\ubar{h}_{b})$}};
			\draw (229.16,125) node  [rotate=-0.4]  {\scalebox{0.75}{$...$}};
			\draw (235.83,257) node  [rotate=-0.4]  {\scalebox{0.75}{$...$}};
			\draw (270.36,125) node  [rotate=-178.56]  {\scalebox{0.75}{$...$}};
			\draw (275.96,257) node  [rotate=-178.56]  {\scalebox{0.75}{$...$}};
			\draw (187.68,110) node [anchor=south] [inner sep=0.75pt]  [font=\tiny]  {\scalebox{0.75}{$(\bar{f}_{1} ,\bar{h}_{{1}})$}};
			\draw (193.75,273) node [anchor=north] [inner sep=0.75pt]  [font=\tiny]  {\scalebox{0.75}{$(\ubar{f}_{1} ,\ubar{h}_{{1}})$}};

			\draw (320.04,380) node [font=\tiny] {(a)};
		\end{tikzpicture}
		\begin{tikzpicture}[x=0.35pt,y=0.4pt,yscale=-1,xscale=1]

			\draw    (100,62.25) -- (100,333.5) ;
			\draw  [fill={rgb, 255:red, 0; green, 0; blue, 0 }  ,fill opacity=1 ] (96,98.38) .. controls (96,96.23) and (97.73,94.5) .. (99.88,94.5) .. controls (102.02,94.5) and (103.75,96.23) .. (103.75,98.38) .. controls (103.75,100.52) and (102.02,102.25) .. (99.88,102.25) .. controls (97.73,102.25) and (96,100.52) .. (96,98.38) -- cycle ;
			\draw  [fill={rgb, 255:red, 0; green, 0; blue, 0 }  ,fill opacity=1 ] (96,148.38) .. controls (96,146.23) and (97.73,144.5) .. (99.88,144.5) .. controls (102.02,144.5) and (103.75,146.23) .. (103.75,148.38) .. controls (103.75,150.52) and (102.02,152.25) .. (99.88,152.25) .. controls (97.73,152.25) and (96,150.52) .. (96,148.38) -- cycle ;
			\draw    (99.88,98.38) .. controls (155.75,98.5) and (158.75,105.5) .. (185.75,125.5) ;
			\draw    (99.88,148.38) .. controls (144.75,146.5) and (152.75,149.5) .. (185.75,125.5) ;
			\draw [dash pattern={on 0.84pt off 2.51pt}] [color={rgb, 255:red, 0; green, 0; blue, 0}  ,draw opacity=1 ]   (99.88,218.38) .. controls (139.75,218.5) and (178.63,245.63) .. (190.75,257.5) ;
			\draw [dash pattern={on 0.84pt off 2.51pt}][color={rgb, 255:red, 0; green, 0; blue, 0 }  ,draw opacity=1 ]   (99.88,298.38) .. controls (151.75,293.5) and (174.63,269.63) .. (190.75,257.5) ;
			\draw [color={rgb, 255:red, 0; green, 0; blue, 0}  ,draw opacity=1 ]   (314.6,122.08) .. controls (332.36,138.33) and (383.29,149.69) .. (392.17,184.63) ;
			\draw [color={rgb, 255:red, 0; green, 0; blue, 0}  ,draw opacity=1 ]   (319.55,254.25) .. controls (353.8,218.48) and (379.72,230.66) .. (392.17,184.63) ;
			\draw  [dash pattern={on 0.84pt off 2.51pt}]    (314.6,122.08) .. controls (354.81,92.36) and (459.13,148.88) .. (484.91,180.4) ;
			\draw  [dash pattern={on 0.84pt off 2.51pt}]    (319.55,254.25) .. controls (369.56,278.26) and (444.7,210.12) .. (484.91,180.4) ;
			\draw  [fill={rgb, 255:red, 0; green, 0; blue, 0 }  ,fill opacity=1 ] (478.32,180.23) .. controls (478.34,178.09) and (480.09,176.36) .. (482.23,176.38) .. controls (484.37,176.39) and (486.09,178.14) .. (486.07,180.28) .. controls (486.06,182.42) and (484.31,184.14) .. (482.17,184.13) .. controls (480.03,184.11) and (478.31,182.37) .. (478.32,180.23) -- cycle ;
			\draw   [dash pattern={on 0.84pt off 2.51pt}]   (391.17,184.62) .. controls (407.6,176.54) and (418.28,174.61) .. (437.69,164.07) ;
			\draw  [dash pattern={on 0.84pt off 2.51pt}]  (437.69,164.07) .. controls (449.76,154.15) and (426.3,172.66) .. (466.51,142.94) ;
			\draw   [dash pattern={on 0.84pt off 2.51pt}]   (466.51,142.94) .. controls (506.72,113.22) and (563.63,104.26) .. (604.23,113.87) ;
			\draw   [dash pattern={on 0.84pt off 2.51pt}]   (392.17,184.63) .. controls (404.79,197.84) and (404.03,210.51) .. (419.97,219.28) ;
			\draw  [dash pattern={on 0.84pt off 2.51pt}]  (419.97,219.28) .. controls (453.88,232.19) and (418.63,220.6) .. (455.88,232.2) ;
			\draw   [dash pattern={on 0.84pt off 2.51pt}]   (453.89,231.53) .. controls (470.95,232.51) and (577.73,256.7) .. (606.05,210.22) ;
			\draw   [dash pattern={on 0.84pt off 2.51pt}]   (482.2,180.25) .. controls (493.75,174.25) and (511.75,163.75) .. (541.25,179.25) ;
			\draw   [dash pattern={on 0.84pt off 2.51pt}]   (482.2,180.25) .. controls (488.25,188.75) and (531.25,192.75) .. (541.25,179.25) ;
			\draw  [fill={rgb, 255:red, 0; green, 0; blue, 0 }  ,fill opacity=1 ] (537.38,179.22) .. controls (537.39,177.08) and (539.14,175.36) .. (541.28,175.38) .. controls (543.42,175.39) and (545.14,177.14) .. (545.12,179.28) .. controls (545.11,181.42) and (543.36,183.14) .. (541.22,183.12) .. controls (539.08,183.11) and (537.36,181.36) .. (537.38,179.22) -- cycle ;
			\draw   [dash pattern={on 0.84pt off 2.51pt}]   (541.25,179.25) .. controls (580.03,162.98) and (592.83,144.4) .. (604.39,111.81) ;
			\draw  [dash pattern={on 0.84pt off 2.51pt}]    (541.25,179.25) .. controls (581.24,180.99) and (587.81,194.37) .. (605.71,209.16) ;
			\draw  [fill={rgb, 255:red, 0; green, 0; blue, 0 }  ,fill opacity=1 ] (600.36,113.85) .. controls (600.37,111.71) and (602.12,109.98) .. (604.26,110) .. controls (606.4,110.01) and (608.12,111.76) .. (608.11,113.9) .. controls (608.09,116.04) and (606.34,117.76) .. (604.2,117.75) .. controls (602.06,117.73) and (600.34,115.99) .. (600.36,113.85) -- cycle ;
			\draw  [fill={rgb, 255:red, 0; green, 0; blue, 0 }  ,fill opacity=1 ] (602.18,210.19) .. controls (602.19,208.05) and (603.94,206.33) .. (606.08,206.34) .. controls (608.22,206.36) and (609.94,208.1) .. (609.93,210.24) .. controls (609.91,212.38) and (608.16,214.11) .. (606.02,214.09) .. controls (603.88,214.08) and (602.16,212.33) .. (602.18,210.19) -- cycle ;
			\draw  [dash pattern={on 0.84pt off 2.51pt}]  (604.23,113.87) .. controls (608.54,117.15) and (593.2,108) .. (637.62,105.51) ;
			\draw  [dash pattern={on 0.84pt off 2.51pt}]  (604.23,113.87) .. controls (609.91,120.92) and (617.5,122.57) .. (637.9,123.11) ;
			\draw  [dash pattern={on 0.84pt off 2.51pt}]  (606.05,210.22) .. controls (610.36,213.5) and (595.02,204.35) .. (639.44,201.86) ;
			\draw  [dash pattern={on 0.84pt off 2.51pt}]  (606.05,210.22) .. controls (611.73,217.26) and (619.32,218.92) .. (639.72,219.46) ;
			\draw  [fill={rgb, 255:red, 0; green, 0; blue, 0 }  ,fill opacity=1 ] (96,218.38) .. controls (96,216.23) and (97.73,214.5) .. (99.88,214.5) .. controls (102.02,214.5) and (103.75,216.23) .. (103.75,218.38) .. controls (103.75,220.52) and (102.02,222.25) .. (99.88,222.25) .. controls (97.73,222.25) and (96,220.52) .. (96,218.38) -- cycle ;
			\draw  [fill={rgb, 255:red, 0; green, 0; blue, 0 }  ,fill opacity=1 ] (96,298.38) .. controls (96,296.23) and (97.73,294.5) .. (99.88,294.5) .. controls (102.02,294.5) and (103.75,296.23) .. (103.75,298.38) .. controls (103.75,300.52) and (102.02,302.25) .. (99.88,302.25) .. controls (97.73,302.25) and (96,300.52) .. (96,298.38) -- cycle ;
			\draw  [fill={rgb, 255:red, 0; green, 0; blue, 0 }  ,fill opacity=1 ] (186.88,257.5) .. controls (186.88,255.36) and (188.61,253.63) .. (190.75,253.63) .. controls (192.89,253.63) and (194.63,255.36) .. (194.63,257.5) .. controls (194.63,259.64) and (192.89,261.38) .. (190.75,261.38) .. controls (188.61,261.38) and (186.88,259.64) .. (186.88,257.5) -- cycle ;
			\draw  [fill={rgb, 255:red, 0; green, 0; blue, 0 }  ,fill opacity=1 ] (181,125.38) .. controls (181,123.23) and (182.73,121.5) .. (184.88,121.5) .. controls (187.02,121.5) and (188.75,123.23) .. (188.75,125.38) .. controls (188.75,127.52) and (187.02,129.25) .. (184.88,129.25) .. controls (182.73,129.25) and (181,127.52) .. (181,125.38) -- cycle ;
			\draw  [fill={rgb, 255:red, 0; green, 0; blue, 0 }  ,fill opacity=1 ] (387.29,184.59) .. controls (387.31,182.45) and (389.05,180.73) .. (391.19,180.74) .. controls (393.33,180.76) and (395.06,182.51) .. (395.04,184.65) .. controls (395.03,186.79) and (393.28,188.51) .. (391.14,188.49) .. controls (389,188.48) and (387.28,186.73) .. (387.29,184.59) -- cycle ;
			\draw   (184.88,125.38) .. controls (189.18,128.66) and (173.85,119.5) .. (218.27,117.01) ;
			\draw    (185.75,125.5) .. controls (191.43,132.55) and (199.02,134.2) .. (219.42,134.74) ;
			\draw  [dash pattern={on 0.84pt off 2.51pt}]  (190.75,257.5) .. controls (195.05,260.78) and (179.72,251.63) .. (224.14,249.14) ;
			\draw  [dash pattern={on 0.84pt off 2.51pt}]  (190.75,257.5) .. controls (196.43,264.55) and (204.02,266.2) .. (224.42,266.74) ;
			\draw    (314.84,122.66) .. controls (310.36,119.63) and (326.19,127.89) .. (281.98,132.91) ;
			\draw  (314.84,122.66) .. controls (308.77,115.95) and (301.09,114.73) .. (280.7,115.35) ;
			\draw  [dash pattern={on 0.84pt off 2.51pt}]  (319.55,254.25) .. controls (315.07,251.22) and (330.9,259.48) .. (286.69,264.5) ;
			\draw  [dash pattern={on 0.84pt off 2.51pt}]  (320.04,255.46) .. controls (313.97,248.75) and (306.29,247.53) .. (285.9,248.15) ;
			\draw  [fill={rgb, 255:red, 0; green, 0; blue, 0 }  ,fill opacity=1 ] (316.17,255.46) .. controls (316.17,253.32) and (317.9,251.58) .. (320.04,251.58) .. controls (322.18,251.58) and (323.92,253.32) .. (323.92,255.46) .. controls (323.92,257.6) and (322.18,259.33) .. (320.04,259.33) .. controls (317.9,259.33) and (316.17,257.6) .. (316.17,255.46) -- cycle ;
			\draw  [fill={rgb, 255:red, 0; green, 0; blue, 0 }  ,fill opacity=1 ] (310.73,122.08) .. controls (310.73,119.94) and (312.46,118.21) .. (314.6,118.21) .. controls (316.74,118.21) and (318.48,119.94) .. (318.48,122.08) .. controls (318.48,124.22) and (316.74,125.96) .. (314.6,125.96) .. controls (312.46,125.96) and (310.73,124.22) .. (310.73,122.08) -- cycle ;

			\draw (100,344.25) node  [font=\tiny]  {$\mathbb{Z}^{d}$};
			\draw (649.33,113) node  [rotate=-0.4]  {\scalebox{0.75}{$...$}};
			\draw (649.15,210) node  [rotate=-0.4]  {\scalebox{0.75}{$...$}};
			\draw (97.88,101.78) node [anchor=north east] [inner sep=0.75pt]  [font=\tiny]  {\scalebox{0.8}{$( 0,x)$}};
			\draw (97.88,151.78) node [anchor=north east] [inner sep=0.75pt]  [font=\tiny]  {\scalebox{0.8}{$( 0,y)$}};
			\draw (97.88,221.78) node [anchor=north east] [inner sep=0.75pt]  [font=\tiny]  {\scalebox{0.8}{$( 0,z)$}};
			\draw (97.88,301.77) node [anchor=north east] [inner sep=0.75pt]  [font=\tiny]  {\scalebox{0.8}{$( 0,w)$}};
			\draw (385.29,184.58) node [anchor=east] [inner sep=0.75pt]  [font=\tiny]  {\scalebox{0.8}{$( {f_{i_{\star }}} ,{h_{i_{\star }}})$}};
			\draw (315.8,112.81) node [anchor=south] [inner sep=0.75pt]  [font=\tiny] {\scalebox{0.8}{$(\bar{f}_{a} ,\bar{h}_{a})$}};
			\draw (320.04,267.53) node [anchor=north] [inner sep=0.75pt]  [font=\tiny]  {\scalebox{0.8}{$(\ubar{f}_{b} ,\ubar{h}_{b})$}};
			\draw (229.16,125) node  [rotate=-0.4]  {\scalebox{0.75}{$...$}};
			\draw (235.83,257) node  [rotate=-0.4]  {\scalebox{0.75}{$...$}};
			\draw (270.36,125) node  [rotate=-178.56]  {\scalebox{0.75}{$...$}};;
			\draw (275.96,257) node  [rotate=-178.56]  {\scalebox{0.75}{$...$}};;
			\draw (187.68,110) node [anchor=south] [inner sep=0.75pt]  [font=\tiny]  {\scalebox{0.8}{$(\bar{f}_{1} ,\bar{h}_{{1}})$}};

			\draw (320.04,380) node [font=\tiny] {(b)};
		\end{tikzpicture} \hspace{0.1cm}

		\caption{ \scriptsize (a) A sample \textbf{$\sfT_1$} configuration. The walks start matching in pairs $(x\leftrightarrow y, z \leftrightarrow w)$, but then switch pair at $(f_{i_\star},h_{i_\star})$. (b) The same configuration after summation of all the possible values of the points $(f_{i},h_{i})_{i>i_\star}$, of the initial positions $(0,z), (0,w)$ and of all the points $(\ubar{f}_{i},\ubar{h}_{i})_{1\leq i<b}$.  }
		\label{fig:Type1}
	\end{figure}
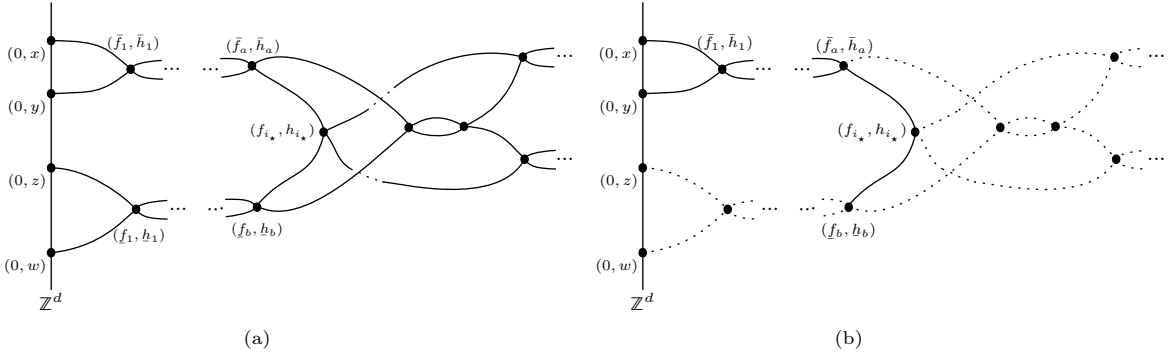

	\noindent We distinguish the following cases {for} such sequences $\tau^{\ms{(\sfa)}}_x, \tau^{\ms{(\sfb)}}_y, \tau^{\ms{(\sfc)}}_z,\tau^{\ms{(\sfd)}}_w$:\\
	\begin{itemize}
		\item  \textbf{Type 1 $(\sfT_1)$}. For all  $(u,\sfs)\, \in  \{(x,\sfa), (y,\sfb),(z,\sfc), (w,\sfd)\}$, we have  $\tau_u^{\ms{(\sfs)}}\cap \big([1,f_{i_\star})\times \Z^d \big)\neq \eset$. \\

		\item \textbf{Type 2 $(\sfT_2)$}. \vspace{0.1cm} For exactly two of the points $(u,\sfs) \in  \{(x,\sfa), (y,\sfb),(z,\sfc), (w,\sfd)\}$, we have that $\tau_u^{\ms{(\sfs)}}\cap \big([1,f_{i_\star})\times \Z^d \big)\neq \eset$. \\

		\item \textbf{Type 3 $(\sfT_3)$}. For all  $(u,\sfs) \in  \{(x,\sfa), (y,\sfb),(z,\sfc), (w,\sfd)\}$ we have that $\tau_u^{\ms{(\sfs)}}\cap \big([1,f_{i_\star})\times \Z^d \big)= \eset$.
	\end{itemize}
	Note that {we have not included the case that} three of the sets $\tau_u^{\ms{(\sfs)}}\cap \big([1,f_{i_\star})\times \Z^d \big)$ are non-empty. This is because, in this case, by the definition of $i_\star$, we have that $\tau_u^{\ms{(\sfs)}}\cap \big([1,f_{i_\star})\times \Z^d \big)$ have to be pairwise equal, therefore all four of them are non-empty. Thus, this is the case of $\sfT_1$ sequences. \\

	\noindent \textbf{($\sfT_1$ sequences)}.
	We begin with the case of $\sfT_1$ sequences $\tau^{\ms{(\sfa)}}_x, \tau^{\ms{(\sfb)}}_y, \tau^{\ms{(\sfc)}}_z,\tau^{\ms{(\sfd)}}_w$. In this case, the four random walks meet pairwise without switching their pair before time $f_{i_\star}$. Let us suppose at first that the walk starting from $(0,x)$ is paired to the walk starting from $(0,y)$ and the walk starting from $(0,z)$ is paired to the walk starting from $(0,w)$, that is
	\begin{align*}
		\tau_x^{\ms{(\sfa)}}\cap \big([1,f_{i_\star})\times \Z^d \big) & =\tau_y^{\ms{(\sfb)}}\cap \big([1,f_{i_\star})\times \Z^d \big) \notag \\
		                                                               & \text{and} \notag                                                      \\
		\tau_z^{\ms{(\sfc)}}\cap \big([1,f_{i_\star})\times \Z^d \big) & =\tau_w^{\ms{(\sfd)}}\cap \big([1,f_{i_\star})\times \Z^d \big) \, .
	\end{align*}
	We shall refer to this type of sequences as $\sfT^{\ms{x \leftrightarrow y}}_1$. Analogously, we define $\sfT^{\ms{x \leftrightarrow z}}_1$ and $\sfT^{\ms{x \leftrightarrow w}}_1$. {By symmetry it only suffices to consider $\sfT^{\ms{x \leftrightarrow y}}_1$}. We will first show how we can perform the summation
	\begin{align} \label{T1xy}
		N^{d-2} \hspace{-0.3cm} \sum_{x,y,z,w \,\in \Z^d} \varphi_N(x,y,z,w) \,\,	{ \sigma^{ \sfa+ \sfb+ \sfc+ \sfd}} \hspace{-0.5cm}
		\sum_{\tau^{\ms{(\sfa)}}_x, \tau^{\ms{(\sfb)}}_y, \tau^{\ms{(\sfc)}}_z, \tau^{\ms{(\sfd)}}_w \,\in \sfT^{\ms{x \leftrightarrow y}}_1} \,  \prodtwo{(u,\sfs)\, \in  \{(x,\sfa), (y,\sfb),}{(z,\sfc), (w,\sfd)\}} q(\tau^{\ms(\sfs)}_u)  \,  \bbE\Big[\prodtwo{(u,\sfs) \, \in \{(x,\sfa),(y,\sfb),}{(z,\sfc),(w,\sfd)\}}  \eta(\tau^{\ms(\sfs)}_u)\Big ] \, .
	\end{align}
	Since the $\eta$ variables have to be paired to each other, we can bound the expectation in \eqref{T1xy} as
	\begin{align} \label{momentsC}
		\bbE\Big[\prodtwo{(u,\sfs) \, \in \{(x,\sfa),(y,\sfb),}{(z,\sfc),(w,\sfd)\}}  \eta(\tau^{\ms(\sfs)}_u)\Big ] \leq C^{2M} \, , \qquad \qquad C=\max\big\{1,\bbE[\eta^3],\bbE[\eta^4]\big\} \, .
	\end{align}
	Moreover, since $M$ is fixed and $1 \leq \sfa,\sfb,\sfc,\sfd \leq M$ we have that ${ \sigma^{ \sfa+ \sfb+ \sfc+ \sfd} \leq (\sigma\vee 1)^{4M}}$. Therefore,
	\begin{align} \label{expboundbyC}
		     & N^{d-2} \hspace{-0.25cm} \sum_{x,y,z,w \,\in \Z^d} \varphi_N(x,y,z,w)\, \, { \sigma^{ \sfa+ \sfb+ \sfc+ \sfd}} \hspace{-0.7cm}
		\sum_{\tau^{\ms{(\sfa)}}_x, \tau^{\ms{(\sfb)}}_y, \tau^{\ms{(\sfc)}}_z, \tau^{\ms{(\sfd)}}_w \,\in \sfT^{\ms{x \leftrightarrow y}}_1} \,  \prodtwo{(u,\sfs)\, \in  \{(x,\sfa), (y,\sfb),}{(z,\sfc), (w,\sfd)\}} q(\tau^{\ms(\sfs)}_u)  \,  \bbE\Big[\prodtwo{(u,\sfs) \, \in \{(x,\sfa),(y,\sfb),}{(z,\sfc),(w,\sfd)\}}  \eta(\tau^{\ms(\sfs)}_u)\Big ]  \notag \\
		\leq & \, C^{2M} {(\sigma\vee 1)^{4M}} N^{d-2} \hspace{-0.25cm} \sum_{x,y,z,w \,\in \Z^d} \varphi_N(x,y,z,w) \hspace{-0.20cm}
		\sum_{\tau^{\ms{(\sfa)}}_x, \tau^{\ms{(\sfb)}}_y, \tau^{\ms{(\sfc)}}_z, \tau^{\ms{(\sfd)}}_w \,\in \sfT^{\ms{x \leftrightarrow y}}_1} \,  \prodtwo{(u,\sfs)\, \in  \{(x,\sfa), (y,\sfb),}{(z,\sfc), (w,\sfd)\}} q(\tau^{\ms(\sfs)}_u)  \, .
	\end{align}
	By the definition of $\sfT_1$ sequences, we have that for {a given $\sfT^{\ms{x \leftrightarrow y}}_1$ sequence} $\tau^{\ms{(\sfa)}}_x, \tau^{\ms{(\sfb)}}_y, \tau^{\ms{(\sfc)}}_z,\tau^{\ms{(\sfd)}}_w$, with $\tau=\tau^{\ms{(\sfa)}}_x \cup \tau^{\ms{(\sfb)}}_y \cup \tau^{\ms{(\sfc)}}_z \cup \tau^{\ms{(\sfd)}}_w=(f_i,h_i)_{1 \leq i \leq p}$ {and} $p=|\tau|$, we can decompose the sequence $(f_i,h_i)_{1 \leq i < i_{\star}}$ into two disjoint subsequences $
		(\bar{f}_{1},\bar{h}_{1}), ..., (\bar{f}_{a},\bar{h}_{a})$ and $(\ubar{f}_{1},\ubar{h}_{1}), ..., (\ubar{f}_{b},\ubar{h}_{b})$, see Figure \ref{fig:Type1}, so that
	\begin{align} \label{T1xyqexpansion}
		\prodtwo{(u,\sfs)\, \in  \{(x,\sfa), (y,\sfb),}{(z,\sfc), (w,\sfd)\}} q(\tau^{\ms(\sfs)}_u)= & \, q_{\bar{f}_1}(\bar{h}_1-x)q_{\bar{f}_1}(\bar{h}_1-y) \prod_{i=2}^{a} q^2_{(\bar{f}_{i}-\bar{f}_{i-1})} (\bar{h}_{i}-\bar{h}_{i-1})          \notag                                                \\
		\times                                                                                       & \, q_{\ubar{f}_1}(\ubar{h}_1-z) q_{\ubar{f}_1}(\ubar{h}_1-w) \prod_{i=2}^{b} q^2_{(\ubar{f}_{i}-\ubar{f}_{i-1})} (\ubar{h}_{i}-\ubar{h}_{i-1})  \notag                                               \\
		\times                                                                                       & \,  q^{ \nu_a }_{(f_{i_{\star}}-\bar{f}_{a})}(h_{i_{\star}}-\bar{h}_a) \, q^{ \nu_b }_{(f_{i_{\star}}-\ubar{f}_{b})}(h_{i_{\star}}-\ubar{h}_b)  \notag                                               \\
		\times                                                                                       & \prod_{m=1}^{\sfm_{i_\star+1}} q_{f_{i_\star+1}-f^{({i_\star+1})}_{r_m}}(h_{i_\star+1}-h^{({i_\star+1})}_{r_m}) \:  ... \prod_{m=1}^{\sfm_{p}} q_{f_{p}-f^{({p})}_{r_m}}(h_{p}-h^{({p})}_{r_m}) \, .
	\end{align}
	For every $i_\star+1 \leq j \leq p$, the number $\sfm_j$ ranges from $2$ to $4$ and indicates whether $(f_j,h_j)$ {is} a double, triple or quadruple matching. Furthermore,
	for every $i_\star+1 \leq j \leq p$ and $1\leq m \leq \sfm_j$, $(f^{(j)}_{r_m},h^{(j)}_{r_m})$ is some space-time point which belongs to the sequence $(f_i,h_i)_{i_\star \leq i \leq p}\cup \{(\bar{f}_{a},\bar{h}_{a}), (\ubar{f}_{b},\ubar{h}_{b}) \}$, such that $f^{(j)}_{r_m}<f_j$. Also, the exponents $\nu_a,\nu_b$ in \eqref{T1xyqexpansion} can take values in $\{1,2\}$ and indicate whether the matching in $(f_{i_\star},h_{i_\star})$
	was double, triple or quadruple. In any case the product above is bounded by the corresponding expression for $\nu_a,\nu_b=1$, since we have {$q_n(x)\leq 1$}.

	In order to perform the summation in \eqref{T1xy} for $\sfT^{\ms{x \leftrightarrow y}}_1$ sequences we make the following observation. We can {start by summing} the last point $(f_p,h_p)$ as follows:
	We use the fact that $q_n(x)\leq 1$ and Cauchy-Schwarz to obtain that
	\begin{align} \label{overlapbound}
		     & \sum_{(f_{p},h_{p})} \prod_{m=1}^{\sfm_{p}} q_{f_{p}-f^{\ms{(p)} }_{r_m}}(h_{p}-h^{\ms{(p)} }_{r_m})  \leq \sum_{(f_{p},h_{p})}  q_{f_{p}-f^{\ms{(p)} }_{r_1}}(h_{p}-h^{\ms{(p)} }_{r_1}) q_{f_{p}-f^{\ms{(p)} }_{r_2}}(h_{p}-h^{\ms{(p)} }_{r_2})  \notag \\
		\leq & \Big ( \sum_{(f_{p},h_{p})}  q^2_{f_{p}-f^{\ms{(p)} }_{r_1}}(h_{p}-h^{\ms{(p)} }_{r_1})  \Big ) ^{\frac{1}{2}} \Big ( \sum_{(f_{p},h_{p})}  q^2_{f_{p}-f^{\ms{(p)} }_{r_2}}(h_{p}-h^{\ms{(p)} }_{r_2})  \Big ) ^{\frac{1}{2}} \notag                       \\
		=    & \Big ( \sum_{f_{p}}  q_{2(f_{p}-f^{\ms{(p)} }_{r_1})}(0) \Big ) ^{\frac{1}{2}} \Big ( \sum_{f_{p}}  q_{2(f_{p}-f^{\ms{(p)} }_{r_2})}(0) \Big) ^{\frac{1}{2}}  \notag                                                                                       \\
		\leq & (\sqrt{R_N})^2=R_N\leq R_{\infty}{=\frac{\pi_d}{1-\pi_d}} <1 \, .
	\end{align}
	For the last inequality, we used that the range of $f_{p}-f^{\ms{(p)} }_{r_i}$ is contained in $\{1,2,...,N\}$ {and the fact that, $\pi_d<\frac{1}{2}$ for $d\geq 3$, since $\pi_3 \approx 0.34$, see \cite{Sp76}, and $\pi_{d+1}<\pi_d$ for $d \geq 3$, see \cite{OS96}.}
	We can successively iterate this estimate for all values of $(f_i,h_i)$ as long as $i>i_{\star}$. Therefore,
	by recalling \eqref{T1xy}, \eqref{expboundbyC} and \eqref{T1xyqexpansion} we deduce that
	\begin{align} \label{T1xylastbound}
		       & {(\sigma\vee 1)^{4M}} C^{2M}  N^{d-2}  \sum_{x,y,z,w \,\in \Z^d} \varphi_N(x,y,z,w)
		\sum_{\tau^{\ms{(\sfa)}}_x, \tau^{\ms{(\sfb)}}_y, \tau^{\ms{(\sfc)}}_z, \tau^{\ms{(\sfd)}}_w\, \in\, \sfT^{\ms{x \leftrightarrow y}}_1}  \, \prodtwo{(u,\sfs)\, \in  \{(x,\sfa), (y,\sfb),}{(z,\sfc), (w,\sfd)\}} q(\tau^{\ms(\sfs)}_u) \notag \\
		\leq   & \, c_M \, {(\sigma\vee 1)^{4M}} \, C^{2M}  N^{d-2}  \sum_{x,y,z,w \,\in \Z^d} \varphi_N(x,y,z,w)  \notag                                                                                                                              \\
		\times & \sum_{a,b=1}^{2M} \Big( \sum_{(\bar{f}_{i},\bar{h}_{i})_{1 \leq i \leq a}} q_{\bar{f}_1}(\bar{h}_1-x)q_{\bar{f}_1}(\bar{h}_1-y) \prod_{i=2}^{a} q^2_{(\bar{f}_{i}-\bar{f}_{i-1})} (\bar{h}_{i}-\bar{h}_{i-1}) \Big )\notag            \\
		\times & \Big(  \sum_{(\ubar{f}_{i},\ubar{h}_{i})_{1 \leq i \leq b}} q_{\ubar{f}_1}(\ubar{h}_1-z) q_{\ubar{f}_1}(\ubar{h}_1-w) \prod_{i=2}^{b} q^2_{(\ubar{f}_{i}-\ubar{f}_{i-1})} (\ubar{h}_{i}-\ubar{h}_{i-1}) \Big ) \notag                 \\
		\times & \Big(  \sum_{(f_{i_\star},h_{i_\star})} q_{(f_{i_{\star}}-\bar{f}_{a})}(h_{i_{\star}}-\bar{h}_a) q_{(f_{i_{\star}}-\ubar{f}_{b})}(h_{i_{\star}}-\ubar{h}_b) \Big )                                 \, ,
	\end{align}
	where $c_M$ is {a} constant combinatorial factor which bounds the number of different ways that the points of $\sfT^{\ms{x \leftrightarrow y}}_1$ can be mapped to a fixed sequence
	$(f_i,h_i)_{1 \leq i \leq p}$, for all $p\leq \frac{\sfa+\sfb+\sfc+\sfd}{2}\leq 2M$.
	Therefore, the last step for showing that the sum \eqref{T1xy} has negligible contribution in \eqref{4thmom} is to show that for all fixed $a,b$ the following sum vanishes when $N$ goes to infinity:

	\begin{align} \label{T1full}
		\tilde{C}_M \,  N^{d-2} \hspace{-0.2cm}\sum_{x,y,z,w \in \Z^d}  \varphi_N(x,y,z,w) & \Big( \sum_{(\bar{f}_{i},\bar{h}_{i})_{1 \leq i \leq a}} q_{\bar{f}_1}(\bar{h}_1-x)q_{\bar{f}_1}(\bar{h}_1-y) \prod_{i=2}^{a} q^2_{(\bar{f}_{i}-\bar{f}_{i-1})} (\bar{h}_{i}-\bar{h}_{i-1}) \Big )\notag              \\
		\times                                                                             & \Big(  \sum_{(\ubar{f}_{i},\ubar{h}_{i})_{1 \leq i \leq b}} q_{\ubar{f}_1}(\ubar{h}_1-z) q_{\ubar{f}_1}(\ubar{h}_1-w) \prod_{i=2}^{b} q^2_{(\ubar{f}_{i}-\ubar{f}_{i-1})} (\ubar{h}_{i}-\ubar{h}_{i-1}) \Big ) \notag \\
		\times                                                                             & \Big(  \sum_{(f_{i_\star},h_{i_\star})} q_{(f_{i_{\star}}-\bar{f}_{a})}(h_{i_{\star}}-\bar{h}_a) q_{(f_{i_{\star}}-\ubar{f}_{b})}(h_{i_{\star}}-\ubar{h}_b) \Big )                                 \, ,
	\end{align}
	where $\tilde{C}_M=c_M \,  {(\sigma\vee 1)^{4M}}\, C^{2M}\,$.
	Let us describe how this can be done. Recall that
	\begin{align*}
		\varphi_N(x,y,z,w)=\prod_{u \in \{x,y,z,w\} }\frac{\varphi\big(\frac{u}{\sqrt{N}}\big)}{N^\frac{d}{2}} \, .
	\end{align*}
	In \eqref{T1full}, we can bound $\varphi(\frac{z}{\sqrt{N}})\varphi(\frac{w}{\sqrt{N}})$ by $\norm{\varphi}_{\infty}^2$ and sum out $z,w$ using that $ \sum_{u \in \Z^d} q_n(u)=1$ so that we bound \eqref{T1full} by
	\begin{align} \label{T1sumoutzw}
		\frac{\tilde{C}_M \norm{\varphi}_{\infty}^2 }{N^2} \sum_{x,y \, \in \Z^d} \varphi_N(x,y) & \Big( \sum_{(\bar{f}_{i},\bar{h}_{i})_{1 \leq i \leq a}} q_{\bar{f}_1}(\bar{h}_1-x)q_{\bar{f}_1}(\bar{h}_1-y) \prod_{i=2}^{a} q^2_{(\bar{f}_{i}-\bar{f}_{i-1})} (\bar{h}_{i}-\bar{h}_{i-1}) \Big )\notag \\
		\times                                                                                   & \Big(  \sum_{(\ubar{f}_{i},\ubar{h}_{i})_{1 \leq i \leq b}}  \prod_{{i=2}}^{b} q^2_{(\ubar{f}_{i}-\ubar{f}_{i-1})} (\ubar{h}_{i}-\ubar{h}_{i-1}) \Big ) \notag                                           \\
		\times                                                                                   & \Big(  \sum_{(f_{i_\star},h_{i_\star})} q_{(f_{i_{\star}}-\bar{f}_{a})}(h_{i_{\star}}-\bar{h}_a) q_{(f_{i_{\star}}-\ubar{f}_{b})}(h_{i_{\star}}-\ubar{h}_b) \Big ) \, .
	\end{align}
	{We} sum out all points $(\ubar{f}_{i-1},\ubar{h}_{i-1})_{2 \leq i <b}$ successively, starting from $(\ubar{f}_{1},\ubar{h}_{1})$ and moving forward.  The contribution of each of these summations is bounded by $R_N<1$, since for each $2 \leq i<b$,
	\begin{align} \label{T1sumbubble}
		\sum_{(\ubar{f}_{i-1},\ubar{h}_{i-1})} q^2_{(\ubar{f}_{i}-\ubar{f}_{i-1})}(\ubar{h}_{i}-\ubar{h}_{i-1})=\sum_{\ubar{f}_{i-1}}q_{2(\ubar{f}_{i}-\ubar{f}_{i-1})}(0)\leq R_N<1 \, .
	\end{align}
	because the range of $\ubar{f}_{i}-\ubar{f}_{i-1}$ is contained in $\{1,...,N\}$.
	Therefore, we are left {with estimating}
	\begin{align*}
		\frac{\tilde{C}_M \norm{\varphi}_{\infty}^2 }{N^2} \sum_{x,y \, \in \Z^d} \varphi_N(x,y) & \Big( \sum_{(\bar{f}_{i},\bar{h}_{i})_{1 \leq i \leq a}} q_{\bar{f}_1}(\bar{h}_1-x)q_{\bar{f}_1}(\bar{h}_1-y) \prod_{i=2}^{a} q^2_{(\bar{f}_{i}-\bar{f}_{i-1})} (\bar{h}_{i}-\bar{h}_{i-1}) \Big )\notag      \\
		\times                                                                                   & \Big(  \sum_{(f_{i_\star},h_{i_\star})}  \sum_{(\ubar{f}_{b},\ubar{h}_{b})}   q_{(f_{i_{\star}}-\bar{f}_{a})}(h_{i_{\star}}-\bar{h}_a) q_{(f_{i_{\star}}-\ubar{f}_{b})}(h_{i_{\star}}-\ubar{h}_b) \Big ) \, .
	\end{align*}
	The contribution of the sums over $(\ubar{f}_b,\ubar{h}_b)$ and $(f_{i_\star},h_{i_\star})$ is
	\begin{align} \label{T1N2}
		\sum_{(f_{i_\star},h_{i_\star})}   q_{(f_{i_{\star}}-\bar{f}_{a})}(h_{i_{\star}}-\bar{h}_a)\, \sum_{(\ubar{f}_{b},\ubar{h}_{b})}   q_{(f_{i_{\star}}-\ubar{f}_{b})}(h_{i_{\star}}-\ubar{h}_b) \leq N^2 \, .
	\end{align}
	{by summing first over space, using that $\sum_{u \in \Z^d} q_n(u)=1$ and then summing over time using that
	the range of $f_{i_{\star}}-\bar{f}_{a}$ and $f_{i_{\star}}-\ubar{f}_{b}$ is contained in $\{1,...,N\}$.}
	Therefore, it remains to show that the following sum vanishes as $N\to \infty$:
	\begin{align*}
		\tilde{C}_M \norm{\varphi}_{\infty}^2  \sum_{x,y \, \in \Z^d} \varphi_N(x,y) & \Big( \sum_{(\bar{f}_{i},\bar{h}_{i})_{1 \leq i \leq a}} q_{\bar{f}_1}(\bar{h}_1-x)q_{\bar{f}_1}(\bar{h}_1-y) \prod_{i=2}^{a} q^2_{(\bar{f}_{i}-\bar{f}_{i-1})} (\bar{h}_{i}-\bar{h}_{i-1}) \Big )\notag \, .
	\end{align*}
	We perform the summation over $(\bar{f}_i,\bar{h}_i)$ for $2 \leq i \leq a $ starting from $(\bar{f}_a,\bar{h}_a)$ and moving backward. The contribution of each of these summations is bounded by $R_N<1$. Consequently, we need to show that
	\begin{align*}
		\tilde{C}_M \norm{\varphi}_{\infty}^2  \sum_{x,y \, \in \Z^d} \varphi_N(x,y) & \sum_{(\bar{f}_{1},\bar{h}_{1})} q_{\bar{f}_1}(\bar{h}_1-x)q_{\bar{f}_1}(\bar{h}_1-y) \xrightarrow[N \to \infty]{} 0\, .
	\end{align*}
	By summing out the points $\bar{h}_1 \in \Z^d$ it suffices to show that
	\begin{align*}
		\tilde{C}_M \norm{\varphi}_{\infty}^2  \sum_{x,y \, \in \Z^d} \varphi_N(x,y) & \sum_{\bar{f}_{1}} q_{2\bar{f}_1}(x-y) \xrightarrow[N \to \infty]{} 0\, .
	\end{align*}
	But it follows from Lemma \ref{conver} that the last sum is $O(N^{1-\frac{d}{2}})$ hence vanishes as $N \to \infty$, since $d\geq 3$. Therefore, we have proved that the sum \eqref{T1xy} vanishes as $N\to \infty$. It is exactly the same to prove the analoguous sums for $\sfT_1^{\ms{x \leftrightarrow z}}$ and $\sfT_1^{\ms{x \leftrightarrow w}}$ sequences vanish as $N \to \infty$. \vspace{0.2cm}

	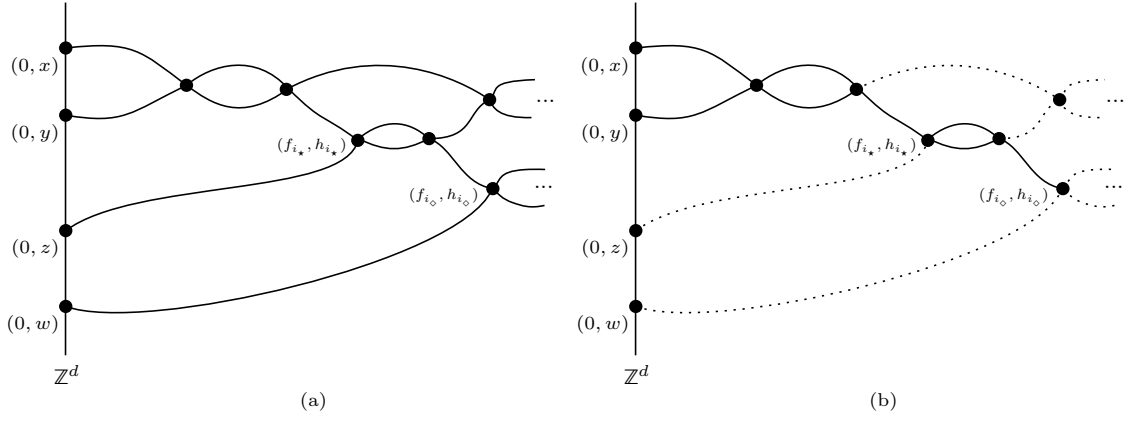
\begin{figure} 
		\centering
		\tikzset{every picture/.style={line width=0.60pt}}

		\begin{tikzpicture}[x=0.65pt,y=0.65pt,yscale=-1,xscale=1]

			\draw    (37.25,38.25) -- (37.25,242.99) ;
			\draw  [fill={rgb, 255:red, 0; green, 0; blue, 0 }  ,fill opacity=1 ] (34,64.63) .. controls (34,62.76) and (35.51,61.25) .. (37.38,61.25) .. controls (39.24,61.25) and (40.75,62.76) .. (40.75,64.63) .. controls (40.75,66.49) and (39.24,68) .. (37.38,68) .. controls (35.51,68) and (34,66.49) .. (34,64.63) -- cycle ;
			\draw  [fill={rgb, 255:red, 0; green, 0; blue, 0 }  ,fill opacity=1 ] (34,103.63) .. controls (34,101.76) and (35.51,100.25) .. (37.38,100.25) .. controls (39.24,100.25) and (40.75,101.76) .. (40.75,103.63) .. controls (40.75,105.49) and (39.24,107) .. (37.38,107) .. controls (35.51,107) and (34,105.49) .. (34,103.63) -- cycle ;
			\draw  [fill={rgb, 255:red, 0; green, 0; blue, 0 }  ,fill opacity=1 ] (34,170.63) .. controls (34,168.76) and (35.51,167.25) .. (37.38,167.25) .. controls (39.24,167.25) and (40.75,168.76) .. (40.75,170.63) .. controls (40.75,172.49) and (39.24,174) .. (37.38,174) .. controls (35.51,174) and (34,172.49) .. (34,170.63) -- cycle ;
			\draw  [fill={rgb, 255:red, 0; green, 0; blue, 0 }  ,fill opacity=1 ] (34,214.63) .. controls (34,212.76) and (35.51,211.25) .. (37.38,211.25) .. controls (39.24,211.25) and (40.75,212.76) .. (40.75,214.63) .. controls (40.75,216.49) and (39.24,218) .. (37.38,218) .. controls (35.51,218) and (34,216.49) .. (34,214.63) -- cycle ;
			\draw    (37.38,64.63) .. controls (76.51,60.22) and (80.51,66.22) .. (106.88,86.25) ;
			\draw    (37.38,103.63) .. controls (71.25,109.25) and (74.51,102.22) .. (106.88,86.25) ;
			\draw  [fill={rgb, 255:red, 0; green, 0; blue, 0 }  ,fill opacity=1 ] (103.5,86.25) .. controls (103.5,84.39) and (105.01,82.88) .. (106.88,82.88) .. controls (108.74,82.88) and (110.25,84.39) .. (110.25,86.25) .. controls (110.25,88.11) and (108.74,89.63) .. (106.88,89.63) .. controls (105.01,89.63) and (103.5,88.11) .. (103.5,86.25) -- cycle ;
			\draw [color={rgb, 255:red, 0; green, 0; blue, 0 }  ,draw opacity=1 ]   (165.51,88.22) .. controls (186.51,109.22) and (184.51,102.22) .. (205.51,118.22) ;
			\draw  [fill={rgb, 255:red, 0; green, 0; blue, 0 }  ,fill opacity=1 ] (202.14,118.22) .. controls (202.14,116.36) and (203.65,114.85) .. (205.51,114.85) .. controls (207.38,114.85) and (208.89,116.36) .. (208.89,118.22) .. controls (208.89,120.09) and (207.38,121.6) .. (205.51,121.6) .. controls (203.65,121.6) and (202.14,120.09) .. (202.14,118.22) -- cycle ;
			\draw    (37.38,170.63) .. controls (77.38,140.63) and (198.51,151.22) .. (205.51,118.22) ;
			\draw    (37.38,214.63) .. controls (82.51,230.99) and (263.51,188.99) .. (282.51,147.22) ;
			\draw    (281.38,94.63) .. controls (287.25,89.25) and (279,84) .. (307.5,83) ;
			\draw    (281.38,94.63) .. controls (283.25,99.25) and (286.5,105.5) .. (305.5,105) ;
			\draw   (283.25,146.25) .. controls (288.18,141.74) and (285.28,136.26) .. (301.64,135.21) .. controls (304.78,135) and (308.63,134.96) .. (313.38,135.13) ;
			\draw    (284.25,146.25) .. controls (285.27,148.77) and (288.7,152.29) .. (294.52,154.33) .. controls (300.34,156.36) and (304.8,157.82) .. (313,156) ;
			\draw    (106.88,86.25) .. controls (121.51,73.99) and (146.51,66.99) .. (165.51,88.22) ;
			\draw    (106.88,86.25) .. controls (121.51,97.22) and (138.51,108.22) .. (165.51,88.22) ;
			\draw  [fill={rgb, 255:red, 0; green, 0; blue, 0 }  ,fill opacity=1 ] (161,88.63) .. controls (161,86.76) and (162.51,85.25) .. (164.38,85.25) .. controls (166.24,85.25) and (167.75,86.76) .. (167.75,88.63) .. controls (167.75,90.49) and (166.24,92) .. (164.38,92) .. controls (162.51,92) and (161,90.49) .. (161,88.63) -- cycle ;
			\draw    (205.88,117.25) .. controls (220.51,105.22) and (234.51,106.22) .. (246.51,117.22) ;
			\draw    (205.88,117.25) .. controls (221.51,124.99) and (231.51,124.99) .. (246.51,117.22) ;
			\draw [color={rgb, 255:red, 0; green, 0; blue, 0 }  ,draw opacity=1 ]   (246.51,117.22) .. controls (260.51,122.22) and (264.51,143.22) .. (283.25,146.25) ;
			\draw    (164.38,88.63) .. controls (217.51,60.22) and (263.51,82.22) .. (281.38,94.63) ;
			\draw    (246.51,117.22) .. controls (273.51,117.22) and (267.51,101.22) .. (281.38,94.63) ;
			\draw  [fill={rgb, 255:red, 0; green, 0; blue, 0 }  ,fill opacity=1 ] (278,94.63) .. controls (278,92.76) and (279.51,91.25) .. (281.38,91.25) .. controls (283.24,91.25) and (284.75,92.76) .. (284.75,94.63) .. controls (284.75,96.49) and (283.24,98) .. (281.38,98) .. controls (279.51,98) and (278,96.49) .. (278,94.63) -- cycle ;
			\draw  [fill={rgb, 255:red, 0; green, 0; blue, 0 }  ,fill opacity=1 ] (243.14,117.22) .. controls (243.14,115.36) and (244.65,113.85) .. (246.51,113.85) .. controls (248.38,113.85) and (249.89,115.36) .. (249.89,117.22) .. controls (249.89,119.09) and (248.38,120.6) .. (246.51,120.6) .. controls (244.65,120.6) and (243.14,119.09) .. (243.14,117.22) -- cycle ;
			\draw  [fill={rgb, 255:red, 0; green, 0; blue, 0 }  ,fill opacity=1 ] (279.88,146.25) .. controls (279.88,144.39) and (281.39,142.88) .. (283.25,142.88) .. controls (285.11,142.88) and (286.63,144.39) .. (286.63,146.25) .. controls (286.63,148.11) and (285.11,149.63) .. (283.25,149.63) .. controls (281.39,149.63) and (279.88,148.11) .. (279.88,146.25) -- cycle ;

			\draw (38,254) node  [font=\scriptsize]  {$\mathbb{Z}^{d}$};
			\draw (35.38,68.03) node [anchor=north east] [inner sep=0.75pt]  [font=\tiny]  {$( 0,x)$};
			\draw (35.38,107.03) node [anchor=north east] [inner sep=0.75pt]  [font=\tiny]  {$( 0,y)$};
			\draw (35.38,174.03) node [anchor=north east] [inner sep=0.75pt]  [font=\tiny]  {$( 0,z)$};
			\draw (35.38,218.03) node [anchor=north east] [inner sep=0.75pt]  [font=\tiny]  {$( 0,w)$};
			\draw (313.5,95) node  [font=\scriptsize]  {$...$};
			\draw (312.5,145) node  [font=\scriptsize]  {$...$};
			\draw (179,122) node  [font=\tiny]  {\scalebox{0.8}{$( {f_{i_{\star }}} ,{h_{i_{\star }}})$}};
			\draw (255,151) node  [font=\tiny]  {\scalebox{0.8}{$( {f_{i_{\diamond }}} ,{h_{i_{\diamond }}})$}};
			\draw (180,270) node [font=\tiny] {(a)};

		\end{tikzpicture}
		\begin{tikzpicture}[x=0.65pt,y=0.65pt,yscale=-1,xscale=1]
			\draw    (37.25,38.25) -- (37.25,242.99) ;
			\draw  [fill={rgb, 255:red, 0; green, 0; blue, 0 }  ,fill opacity=1 ] (34,64.63) .. controls (34,62.76) and (35.51,61.25) .. (37.38,61.25) .. controls (39.24,61.25) and (40.75,62.76) .. (40.75,64.63) .. controls (40.75,66.49) and (39.24,68) .. (37.38,68) .. controls (35.51,68) and (34,66.49) .. (34,64.63) -- cycle ;
			\draw  [fill={rgb, 255:red, 0; green, 0; blue, 0 }  ,fill opacity=1 ] (34,103.63) .. controls (34,101.76) and (35.51,100.25) .. (37.38,100.25) .. controls (39.24,100.25) and (40.75,101.76) .. (40.75,103.63) .. controls (40.75,105.49) and (39.24,107) .. (37.38,107) .. controls (35.51,107) and (34,105.49) .. (34,103.63) -- cycle ;
			\draw  [fill={rgb, 255:red, 0; green, 0; blue, 0 }  ,fill opacity=1 ] (34,170.63) .. controls (34,168.76) and (35.51,167.25) .. (37.38,167.25) .. controls (39.24,167.25) and (40.75,168.76) .. (40.75,170.63) .. controls (40.75,172.49) and (39.24,174) .. (37.38,174) .. controls (35.51,174) and (34,172.49) .. (34,170.63) -- cycle ;
			\draw  [fill={rgb, 255:red, 0; green, 0; blue, 0 }  ,fill opacity=1 ] (34,214.63) .. controls (34,212.76) and (35.51,211.25) .. (37.38,211.25) .. controls (39.24,211.25) and (40.75,212.76) .. (40.75,214.63) .. controls (40.75,216.49) and (39.24,218) .. (37.38,218) .. controls (35.51,218) and (34,216.49) .. (34,214.63) -- cycle ;
			\draw    (37.38,64.63) .. controls (76.51,60.22) and (80.51,66.22) .. (106.88,86.25) ;
			\draw    (37.38,103.63) .. controls (71.25,109.25) and (74.51,102.22) .. (106.88,86.25) ;
			\draw  [fill={rgb, 255:red, 0; green, 0; blue, 0 }  ,fill opacity=1 ] (103.5,86.25) .. controls (103.5,84.39) and (105.01,82.88) .. (106.88,82.88) .. controls (108.74,82.88) and (110.25,84.39) .. (110.25,86.25) .. controls (110.25,88.11) and (108.74,89.63) .. (106.88,89.63) .. controls (105.01,89.63) and (103.5,88.11) .. (103.5,86.25) -- cycle ;
			\draw [color={rgb, 255:red, 0; green, 0; blue, 0}  ,draw opacity=1 ]   (165.51,88.22) .. controls (186.51,109.22) and (184.51,102.22) .. (205.51,118.22) ;
			\draw  [fill={rgb, 255:red, 0; green, 0; blue, 0 }  ,fill opacity=1 ] (202.14,118.22) .. controls (202.14,116.36) and (203.65,114.85) .. (205.51,114.85) .. controls (207.38,114.85) and (208.89,116.36) .. (208.89,118.22) .. controls (208.89,120.09) and (207.38,121.6) .. (205.51,121.6) .. controls (203.65,121.6) and (202.14,120.09) .. (202.14,118.22) -- cycle ;
			\draw  [dash pattern={on 0.84pt off 2.51pt}]  (37.38,170.63) .. controls (77.38,140.63) and (198.51,151.22) .. (205.51,118.22) ;
			\draw  [dash pattern={on 0.84pt off 2.51pt}]  (37.38,214.63) .. controls (82.51,230.99) and (263.51,188.99) .. (282.51,147.22) ;
			\draw  [dash pattern={on 0.84pt off 2.51pt}]  (281.38,94.63) .. controls (287.25,89.25) and (279,84) .. (307.5,83) ;
			\draw  [dash pattern={on 0.84pt off 2.51pt}]  (281.38,94.63) .. controls (283.25,99.25) and (286.5,105.5) .. (305.5,105) ;
			\draw  [dash pattern={on 0.84pt off 2.51pt}]  (283.25,146.25) .. controls (288.18,141.74) and (285.28,136.26) .. (301.64,135.21) .. controls (304.78,135) and (308.63,134.96) .. (313.38,135.13) ;
			\draw  [dash pattern={on 0.84pt off 2.51pt}]  (284.25,146.25) .. controls (285.27,148.77) and (288.7,152.29) .. (294.52,154.33) .. controls (300.34,156.36) and (304.8,157.82) .. (313,156) ;
			\draw    (106.88,86.25) .. controls (121.51,73.99) and (146.51,66.99) .. (165.51,88.22) ;
			\draw    (106.88,86.25) .. controls (121.51,97.22) and (138.51,108.22) .. (165.51,88.22) ;
			\draw  [fill={rgb, 255:red, 0; green, 0; blue, 0 }  ,fill opacity=1 ] (161,88.63) .. controls (161,86.76) and (162.51,85.25) .. (164.38,85.25) .. controls (166.24,85.25) and (167.75,86.76) .. (167.75,88.63) .. controls (167.75,90.49) and (166.24,92) .. (164.38,92) .. controls (162.51,92) and (161,90.49) .. (161,88.63) -- cycle ;
			\draw    (205.88,117.25) .. controls (220.51,105.22) and (234.51,106.22) .. (246.51,117.22) ;
			\draw   (205.88,117.25) .. controls (221.51,124.99) and (231.51,124.99) .. (246.51,117.22) ;
			\draw [color={rgb, 255:red, 0; green, 0; blue, 0}  ,draw opacity=1 ]   (246.51,117.22) .. controls (260.51,122.22) and (264.51,143.22) .. (283.25,146.25) ;
			\draw  [dash pattern={on 0.84pt off 2.51pt}]  (164.38,88.63) .. controls (217.51,60.22) and (263.51,82.22) .. (281.38,94.63) ;
			\draw  [dash pattern={on 0.84pt off 2.51pt}]  (246.51,117.22) .. controls (273.51,117.22) and (267.51,101.22) .. (281.38,94.63) ;
			\draw  [fill={rgb, 255:red, 0; green, 0; blue, 0 }  ,fill opacity=1 ] (278,94.63) .. controls (278,92.76) and (279.51,91.25) .. (281.38,91.25) .. controls (283.24,91.25) and (284.75,92.76) .. (284.75,94.63) .. controls (284.75,96.49) and (283.24,98) .. (281.38,98) .. controls (279.51,98) and (278,96.49) .. (278,94.63) -- cycle ;
			\draw  [fill={rgb, 255:red, 0; green, 0; blue, 0 }  ,fill opacity=1 ] (243.14,117.22) .. controls (243.14,115.36) and (244.65,113.85) .. (246.51,113.85) .. controls (248.38,113.85) and (249.89,115.36) .. (249.89,117.22) .. controls (249.89,119.09) and (248.38,120.6) .. (246.51,120.6) .. controls (244.65,120.6) and (243.14,119.09) .. (243.14,117.22) -- cycle ;
			\draw  [fill={rgb, 255:red, 0; green, 0; blue, 0 }  ,fill opacity=1 ] (279.88,146.25) .. controls (279.88,144.39) and (281.39,142.88) .. (283.25,142.88) .. controls (285.11,142.88) and (286.63,144.39) .. (286.63,146.25) .. controls (286.63,148.11) and (285.11,149.63) .. (283.25,149.63) .. controls (281.39,149.63) and (279.88,148.11) .. (279.88,146.25) -- cycle ;

			\draw (38,254) node  [font=\scriptsize]  {$\mathbb{Z}^{d}$};
			\draw (35.38,68.03) node [anchor=north east] [inner sep=0.75pt]  [font=\tiny]  {$( 0,x)$};
			\draw (35.38,107.03) node [anchor=north east] [inner sep=0.75pt]  [font=\tiny]  {$( 0,y)$};
			\draw (35.38,174.03) node [anchor=north east] [inner sep=0.75pt]  [font=\tiny]  {$( 0,z)$};
			\draw (35.38,218.03) node [anchor=north east] [inner sep=0.75pt]  [font=\tiny]  {$( 0,w)$};
			\draw (313.5,95) node  [font=\scriptsize]  {$...$};
			\draw (312.5,145) node  [font=\scriptsize]  {$...$};
			\draw (179,122) node  [font=\tiny]  {\scalebox{0.8}{$( {f_{i_{\star }}} ,{h_{i_{\star }}})$}};
			\draw (255,151) node  [font=\tiny]  {\scalebox{0.8}{$( {f_{i_{\diamond }}} ,{h_{i_{\diamond }}})$}};
			\draw (180,270) node [font=\tiny] {(b)};
		\end{tikzpicture}
		\caption{ \scriptsize (a) A sample \textbf{$\sfT_2$} configuration. (b) The same configuration after summation of all possible values of the points $(f_{i},h_{i})_{i>i_\diamond}$ and  of the initial positions $(0,z), (0,w)$.}
		\label{fig:Type2}
	\end{figure}

	\noindent \textbf{($\sfT_2$ sequences)} Recall that by the definition of $\sfT_2$ sequences we have that for exactly two of the points $(u,\sfs) \in  \{(x,\sfa), (y,\sfb),(z,\sfc), (w,\sfd)\}$, it holds for the corresponding sets $\tau_u^{\ms{(\sfs)}}\cap \big([1,f_{i_\star})\times \Z^d \big)\neq \eset$ that
	\begin{align*}
		\tau_x^{\ms{(\sfa)}}\cap \big([1,f_{i_\star})\times \Z^d \big) & = \tau_y^{\ms{(\sfb)}}\cap \big([1,f_{i_\star})\times \Z^d \big)\neq \eset \notag           \\ \vspace{0.2cm}
		                                                               & \text{ and }	\notag                                                                          \\
		\tau_z^{\ms{(\sfc)}}\cap \big([1,f_{i_\star})\times \Z^d \big) & = \tau_w^{\ms{(\sfd)}}\cap \big([1,f_{i_\star})\times \Z^d \big)= \eset \, . \vspace{0.2cm}
	\end{align*}
	We will refer to this type of $\sfT_2$ sequences as $\sfT^{\ms{x \leftrightarrow y}}_2$. Analogously, we {can} define $\sfT^{\ms{x \leftrightarrow z}}_2$ and $\sfT^{\ms{x \leftrightarrow w}}_2$. {We will show} that the sum
	\begin{align} \label{T2xy}
		N^{d-2} \hspace{-0.3cm} \sum_{x,y,z,w \,\in \Z^d} \varphi_N(x,y,z,w)\,\,	{ \sigma^{ \sfa+ \sfb+ \sfc+ \sfd}} \hspace{-0.6cm}
		\sum_{\tau^{\ms{(\sfa)}}_x, \tau^{\ms{(\sfb)}}_y, \tau^{\ms{(\sfc)}}_z, \tau^{\ms{(\sfd)}}_w \,\in \sfT^{\ms{x \leftrightarrow y}}_2} \,  \prodtwo{(u,\sfs)\, \in  \{(x,\sfa), (y,\sfb),}{(z,\sfc), (w,\sfd)\}} q(\tau^{\ms(\sfs)}_u)  \,  \bbE\Big[\prodtwo{(u,\sfs) \, \in \{(x,\sfa),(y,\sfb),}{(z,\sfc),(w,\sfd)\}}  \eta(\tau^{\ms(\sfs)}_u)\Big ] \, ,
	\end{align}
	vanishes as $N \to \infty$. {By using \eqref{momentsC} { and the bound $\sigma^{\sfa+\sfb+\sfc+\sfd} \leq (\sigma\vee 1)^{4M}$} we obtain that}
	\begin{align} \label{T2xywithout}
		     & N^{d-2} \hspace{-0.3cm} \sum_{x,y,z,w \,\in \Z^d} \varphi_N(x,y,z,w)\,\,	{ \sigma^{ \sfa+ \sfb+ \sfc+ \sfd}} \hspace{-0.6cm}
		\sum_{\tau^{\ms{(\sfa)}}_x, \tau^{\ms{(\sfb)}}_y, \tau^{\ms{(\sfc)}}_z, \tau^{\ms{(\sfd)}}_w \,\in \sfT^{\ms{x \leftrightarrow y}}_2} \,  \prodtwo{(u,\sfs)\, \in  \{(x,\sfa), (y,\sfb),}{(z,\sfc), (w,\sfd)\}} q(\tau^{\ms(\sfs)}_u)  \,  \bbE\Big[\prodtwo{(u,\sfs) \, \in \{(x,\sfa),(y,\sfb),}{(z,\sfc),(w,\sfd)\}}  \eta(\tau^{\ms(\sfs)}_u)\Big ] \notag \\
		\leq & \, {(\sigma\vee 1)^{4M}} C^{2M}\, N^{d-2} \hspace{-0.25cm} \sum_{x,y,z,w \,\in \Z^d} \varphi_N(x,y,z,w) \hspace{-0.20cm}
		\sum_{\tau^{\ms{(\sfa)}}_x, \tau^{\ms{(\sfb)}}_y, \tau^{\ms{(\sfc)}}_z, \tau^{\ms{(\sfd)}}_w \,\in \sfT^{\ms{x \leftrightarrow y}}_2} \,  \prodtwo{(u,\sfs)\, \in  \{(x,\sfa), (y,\sfb),}{(z,\sfc), (w,\sfd)\}} q(\tau^{\ms(\sfs)}_u)  \, .
	\end{align}

	{By the definition of }$(f_{i_\star},h_{i_\star})$ we have that $(f_{i_\star},h_{i_\star})$ is the first point of at least one of the sequences $\tau_z^{\ms{(\sfc)}}, \tau_w^{\ms{(\sfd)}}$. Let us assume {that} it is the first point of exactly one of them.  We will refer to this type of sequences, $\tau^{\ms{(\sfa)}}_x, \tau^{\ms{(\sfb)}}_y, \tau^{\ms{(\sfc)}}_z, \tau^{\ms{(\sfd)}}_w$, as $\sfT^{\ms{x \leftrightarrow y}}_{2,\diamond}$ sequences, see figure \ref{fig:Type2}.
	Without loss of generality, we may assume that $(f_{i_\star},h_{i_\star})$ is the first point of $\tau_z^{\ms{(\sfc)}}$. In that case, $(f_{i_{\star}},h_{i_{\star}})$ can be a double or triple matching. Let $(f_{i_{\diamond}},h_{i_{\diamond}})$ be the first point of $\tau_w^{\ms{(\sfd)}}$. We have that $f_{i_{\star}} \leq f_{i_{\diamond}} $. Therefore, we first show that
	\begin{align} \label{T2diamond}
		{(\sigma\vee 1)^{4M}}	C^{2M}\, N^{d-2} \hspace{-0.25cm} \sum_{x,y,z,w \,\in \Z^d} \varphi_N(x,y,z,w) \hspace{-0.20cm}
		\sum_{\tau^{\ms{(\sfa)}}_x, \tau^{\ms{(\sfb)}}_y, \tau^{\ms{(\sfc)}}_z, \tau^{\ms{(\sfd)}}_w \,\in \sfT^{\ms{x \leftrightarrow y}}_{2,\diamond}} \,  \prodtwo{(u,\sfs)\, \in  \{(x,\sfa), (y,\sfb),}{(z,\sfc), (w,\sfd)\}} q(\tau^{\ms(\sfs)}_u) \xrightarrow[N\to \infty]{} 0 \, .
	\end{align}
	Similarly to the case of $\sfT_1$ sequences, for given $\sfT^{\ms{x \leftrightarrow y}}_{2,\diamond}$ sequences $\tau^{\ms{(\sfa)}}_x , \tau^{\ms{(\sfb)}}_y , \tau^{\ms{(\sfc)}}_z , \tau^{\ms{(\sfd)}}_w$ with $\tau=\tau^{\ms{(\sfa)}}_x \cup \tau^{\ms{(\sfb)}}_y \cup \tau^{\ms{(\sfc)}}_z \cup \tau^{\ms{(\sfd)}}_w=(f_i,h_i)_{1 \leq i \leq p}$ and $p=|\tau|$, the cardinality of $\tau$, we have that (see Figure \ref{fig:Type2})

	\begin{align} \label{T2xyqexpansion}
		\prodtwo{(u,\sfs)\, \in  \{(x,\sfa), (y,\sfb),}{(z,\sfc), (w,\sfd)\}} q(\tau^{\ms(\sfs)}_u)= & \, q_{\bar{f}_1}(\bar{h}_1-x)q_{\bar{f}_1}(\bar{h}_1-y) \prod_{i=2}^{a} q^2_{(\bar{f}_{i}-\bar{f}_{i-1})} (\bar{h}_{i}-\bar{h}_{i-1})          \notag                                                                                                 \\
		\times                                                                                       & q^{ \nu_a }_{(f_{i_{\star}}-\bar{f}_{a})}(h_{i_{\star}}-\bar{h}_a) \, q_{f_{i_{\star}}}(h_{i_{\star}}-z)  \,q_{f_{i_{\diamond}}}(h_{i_{\diamond}}-w)\notag                                                                                            \\
		\times                                                                                       & \prod_{m=1}^{\sfm_{i_\star+1}} q_{f_{i_\star+1}-f^{({i_\star+1})}_{r_m}}(h_{i_\star+1}-h^{({i_\star+1})}_{r_m}) \:  ... \prod_{m=1}^{\sfm_{i_\diamond}-1} q_{f_{i_\diamond}-f^{({i_\diamond})}_{r_m}}(h_{i_\diamond}-h^{({i_\diamond})}_{r_m}) \notag \\
		\times                                                                                       & \prod_{m=1}^{\sfm_{i_\diamond+1}} q_{f_{i_\diamond+1}-f^{({i_\diamond+1})}_{r_m}}(h_{i_\diamond+1}-h^{({i_\diamond+1})}_{r_m}) \:  ... \prod_{m=1}^{\sfm_{p}} q_{f_{p}-f^{({p})}_{r_m}}(h_{p}-h^{({p})}_{r_m})  \, ,
	\end{align}
	{where,} for every $i_\star+1 \leq j \leq p$, the number $\sfm_j$ ranges from $2$ to $4$ and indicates whether $(f_j,h_j)$ was a double, triple or quadruple matching. Also,
	for every $i_\star+1 \leq j \leq p$ and $1\leq m \leq \sfm_j$, $(f^{(j)}_{r_m},h^{(j)}_{r_m})$ is some space-time point which belongs to the sequence $(f_i,h_i)_{i_\star \leq i \leq p}\cup \{(\bar{f}_{a},\bar{h}_{a}) \}$, such that $f^{(j)}_{r_m}<f_j$.
	However, note that in the third line of \eqref{T2xyqexpansion}, the product for $(f_{i_\diamond},h_{i_\diamond})$ runs from $m=1$ to $\sfm_{i_\diamond}-1$, since $q_{f_{i_{\diamond}}}(h_{i_{\diamond}}-w)$ appears in the second line. The exponent $\nu_a$ in the second line of \eqref{T2xyqexpansion} can take values $1$ or $2$ and indicates whether $(f_{i_\star},h_{i_\star})$ {is} a double or triple matching; it cannot be a quadruple matching since we assumed that it is contained only in $\tau_z^{(c)}$ {and not in $\tau_w^{(d)}$.} In any case, we can bound $q^{ \nu_a }_{(f_{i_{\star}}-\bar{f}_{a})}(h_{i_{\star}}-\bar{h}_a)$ by $q_{(f_{i_{\star}}-\bar{f}_{a})}(h_{i_{\star}}-\bar{h}_a)$.

	We first make some observations so that the presentation is more concise.
	By iterating  \eqref{overlapbound} we obtain that
	\begin{align} \label{dobs1}
		\sum_{ (f_{i_{\diamond}+1},h_{i_{\diamond}+1})} \prod_{m=1}^{\sfm_{i_\diamond+1}} q_{f_{i_\diamond+1}-f^{({i_\diamond+1})}_{r_m}}(h_{i_\diamond+1}-h^{({i_\diamond+1})}_{r_m}) \:  ... \sum_{ (f_{p},h_p) }\prod_{m=1}^{\sfm_{p}} q_{f_{p}-f^{({p})}_{r_m}}(h_{p}-h^{({p})}_{r_m}) \leq 1 \, .
	\end{align}
	We also have that
	\begin{align} \label{dobs2}
		{\sum_{w \in \Z^d} \varphi_N(w) \, q_{f_{i_{\diamond}}}(h_{i_{\diamond}}-w) = \frac{1}{N^{\frac{d}{2}}} \sum_{w \in \Z^d} \varphi\big(\tfrac{w}{\sqrt{N}}\big) \, q_{f_{i_{\diamond}}}(h_{i_{\diamond}}-w) \leq \frac{\norm{\varphi}_{\infty}}{N^{\frac{d}{2}}} \sum_{w \in \Z^d} q_{f_{i_{\diamond}}}(h_{i_{\diamond}}-w) =\frac{\norm{\varphi}_{\infty}}{N^{\frac{d}{2}}} \, ,}
	\end{align}
	and {then we can sum}
	\begin{align} \label{dobs3}
		{\sum_{(f_{i_{\diamond}},h_{i_{\diamond}}) }\prod_{m=1}^{\sfm_{i_\diamond}-1} q_{f_{i_\diamond}-f^{({i_\diamond})}_{r_m}}(h_{i_\diamond}-h^{({i_\diamond})}_{r_m}) \leq \sum_{(f_{i_{\diamond}},h_{i_{\diamond}}) } q_{f_{i_\diamond}-f^{({i_\diamond})}_{r_1}}(h_{i_\diamond}-h^{({i_\diamond})}_{r_1}) \leq N \, ,}
	\end{align}
	{Having summed out the points} $(f_i,h_i)_{i\geq i_\diamond}$, we can iterate estimate \eqref{overlapbound} again to obtain that
	\begin{align} \label{dobs4}
		\sum_{ (f_{i_{\star}+1},h_{i_{\star}+1})} \prod_{m=1}^{\sfm_{i_\star+1}} q_{f_{i_\star+1}-f^{({i_\star+1})}_{r_m}}(h_{i_\star+1}-h^{({i_\star+1})}_{r_m}) \:  ... \hspace{-0.4cm} \sum_{ (f_{i_{\diamond}-1},h_{i_{\diamond}-1})} \prod_{m=1}^{\sfm_{i_\diamond-1}} q_{f_{i_\diamond-1}-f^{({i_\diamond-1})}_{r_m}}(h_{i_\diamond-1}-h^{({i_\diamond-1})}_{r_m}) \leq 1 \, .
	\end{align}
	\noindent {Therefore, in view of \eqref{T2xyqexpansion}, \eqref{T2diamond} and by using \eqref{dobs1}, \eqref{dobs2}, \eqref{dobs3} and \eqref{dobs4} in their respective order, we get that}
	\begin{align*}
		       & {(\sigma\vee 1)^{4M}} C^{2M}  N^{d-2}  \sum_{x,y,z,w \,\in \Z^d} \varphi_N(x,y,z,w)
		\sum_{\tau^{\ms{(\sfa)}}_x, \tau^{\ms{(\sfb)}}_y, \tau^{\ms{(\sfc)}}_z, \tau^{\ms{(\sfd)}}_w\, \in\, \sfT^{\ms{x \leftrightarrow y}}_{2,\diamond}}  \, \prodtwo{(u,\sfs)\, \in  \{(x,\sfa), (y,\sfb),}{(z,\sfc), (w,\sfd)\}} q(\tau^{\ms(\sfs)}_u) \notag \\
		\leq   & \, \norm{\varphi}_{\infty} c_{M,\diamond} \,  {(\sigma\vee 1)^{4M}} \, C^{2M}  N^{\frac{d}{2}-1}  \sum_{x,y,z \,\in \Z^d} \varphi_N(x,y,z)  \notag                                                                                               \\
		\times & \sum_{a=1}^{2M} \Big(\sum_{(\bar{f}_{i},\bar{h}_{i})_{1 \leq i \leq a}}q_{\bar{f}_1}(\bar{h}_1-x)q_{\bar{f}_1}(\bar{h}_1-y) \prod_{i=2}^{a} q^2_{(\bar{f}_{i}-\bar{f}_{i-1})} (\bar{h}_{i}-\bar{h}_{i-1})      \Big)    \notag                   \\
		\times & \Big(\sum_{(f_{i_\star},h_{i_\star})}   q_{(f_{i_{\star}}-\bar{f}_{a})}(h_{i_{\star}}-\bar{h}_a) \, q_{f_{i_{\star}}}(h_{i_{\star}}-z) \Big)  \, ,
	\end{align*}
	where {$c_{M,\diamond}$ is a constant combinatorial factor  which bounds the number of possible assignments of $\sfT^{\ms{x \leftrightarrow y}}_{2,\diamond}$ sequences, $\tau^{\ms{(\sfa)}}_x, \tau^{\ms{(\sfb)}}_y, \tau^{\ms{(\sfc)}}_z, \tau^{\ms{(\sfd)}}_w$ to $(f_i,h_i)_{1\leq i \leq p}$.} {  We set $\tilde{C}_{M,\diamond}:=  c_{M,\diamond} \, (\sigma\vee 1)^{4M} \, C^{2M}$}. In order to establish \eqref{T2diamond}, we need to show that for all fixed {$a\leq 2M$ }
	\begin{align*}
		\norm{\varphi}_{\infty} {\tilde{C}_{M,\diamond} \, }  N^{\frac{d}{2}-1}
		\times & \Big(\sum_{(\bar{f}_{i},\bar{h}_{i})_{1 \leq i \leq a}}q_{\bar{f}_1}(\bar{h}_1-x)q_{\bar{f}_1}(\bar{h}_1-y) \prod_{i=2}^{a} q^2_{(\bar{f}_{i}-\bar{f}_{i-1})} (\bar{h}_{i}-\bar{h}_{i-1})      \Big)    \notag \\
		\times & \Big(\sum_{(f_{i_\star},h_{i_\star})}   q_{(f_{i_{\star}}-\bar{f}_{a})}(h_{i_{\star}}-\bar{h}_a) \, q_{f_{i_{\star}}}(h_{i_{\star}}-z) \Big)  \,\notag  \xrightarrow[N \to \infty]{} 0 \, .
	\end{align*}
	{In analogy to \eqref{dobs2}, we have  that}
	\begin{align*}
		\sum_{z \in \Z^d} \varphi_N(z) \, q_{f_{i_{\star}}}(h_{i_{\star}}-z) \leq \frac{\norm{\varphi}_{\infty}}{N^{\frac{d}{2}}} \, .
	\end{align*}
	{Furthermore, by summing {over} $(f_{i_\star},h_{i_\star})$ we deduce that}
	\begin{align*}
		\sum_{(f_{i_\star},h_{i_\star})}  q_{(f_{i_{\star}}-\bar{f}_{a})}(h_{i_{\star}}-\bar{h}_a) \leq N \, ,
	\end{align*}
	{since the spatial sum is equal to $1$} and $f_{i_{\star}}-\bar{f}_{a} \in \{1,...,N\}$.  Therefore, the last step {in order} to establish \eqref{T2diamond} is to show that
	\begin{align*}
		{\tilde{C}_{M,\diamond} \, } \norm{\varphi}^2_{\infty}  \sum_{x,y \,\in \Z^d} \varphi_N(x,y)
		 & \sum_{(\bar{f}_{i},\bar{h}_{i})_{1 \leq i \leq a}}q_{\bar{f}_1}(\bar{h}_1-x)q_{\bar{f}_1}(\bar{h}_1-y) \prod_{i=2}^{a} q^2_{(\bar{f}_{i}-\bar{f}_{i-1})} (\bar{h}_{i}-\bar{h}_{i-1})  \xrightarrow[N \to \infty]{} 0 \, .
	\end{align*}
	By summing over the points $(\bar{f}_{i},\bar{h}_{i})_{2 \leq i \leq a}$, this amounts to proving that
	\begin{align*}
		{\tilde{C}_{M,\diamond} \, } \norm{\varphi}^2_{\infty} \sum_{x,y \,\in \Z^d} \varphi_N(x,y) \sum_{\bar{f}_{1}} q_{2\bar{f}_1}(x-y)\xrightarrow[N \to \infty]{} 0 \, ,
	\end{align*}
	which is true by Lemma \ref{conver}.
	The same procedure can be followed for sequences of type $\sfT^{\ms{x \leftrightarrow z}}_{2,\diamond}$ and $\sfT^{\ms{x \leftrightarrow w}}_{2,\diamond}$. {So, this concludes the estimate for $\sfT^{\ms{x \leftrightarrow y}}_2$ sequences in the case that $(f_{i_\star},h_{i_\star})$ is the first point of {only} one of the sequences $\tau_z^{\ms{(\sfc)}}, \tau_w^{\ms{(\sfd)}}$ and by symmetry also for the analogous cases for $\sfT^{\ms{x \leftrightarrow z}}_2$ and $\sfT^{\ms{x \leftrightarrow w}}_2$. }

	{Let us treat the case where $(f_{i_\star},h_{i_\star})$ is the first point of both sequences $\tau_z^{\ms{(\sfc)}}, \tau_w^{\ms{(\sfd)}}$}.
	Then, $(f_{i_{\star}},h_{i_{\star}})$ is a triple or quadruple matching, { i.e. either $(f_{i_{\star}},h_{i_{\star}}) \in \tau_x^{\ms{(\sfa)}},\tau_z^{\ms{(\sfc)}}, \tau_w^{\ms{(\sfd)}}$, or $(f_{i_{\star}},h_{i_{\star}}) \in \tau_y^{\ms{(\sfb)}},\tau_z^{\ms{(\sfc)}}, \tau_w^{\ms{(\sfd)}}$, or $(f_{i_{\star}},h_{i_{\star}}) \in \tau_x^{\ms{(\sfa)}},\tau_y^{\ms{(\sfb)}},\tau_z^{\ms{(\sfc)}}, \tau_w^{\ms{(\sfd)}}$.}
	Both cases can be treated as we did for $T_1$ sequences. Namely, we can first restrict ourselves to the sequence $(f_i,h_i)_{1 \leq i \leq i_\star}$ by using the bound we used in \eqref{overlapbound}. After following the procedure we described for $\sfT_1$ sequences we get that the sum in this case is either $O(N^{-\frac{d}{2}})$ if $(f_{i_{\star}},h_{i_{\star}})$ is a triple matching and $O(N^{-1-\frac{d}{2}})$ when  $(f_{i_{\star}},h_{i_{\star}})$ is a quadruple matching. Thus, in total the contribution of $\sfT_2$ sequences to \eqref{4thmom}, is $O(N^{1-\frac{d}{2}})$.\\

	\noindent\textbf{($\sfT_3$ sequences)}. For all  $(u,\sfs) \in  \{(x,\sfa), (y,\sfb),(z,\sfc), (w,\sfd)\}$ we have that $\tau_u^{\ms{(\sfs)}}\cap \big([1,f_{i_\star})\times \Z^d \big)= \eset$. This implies that $i_{\star}=1$ and $(f_{i_{\star}},h_{i_{\star}})$ is a triple or quadruple matching. It is easy {to see}, using the technique for $\sfT_1$ and $\sfT_2$ sequences, that the contribution of $\sfT_3$ sequences to \eqref{4thmom} is $O(N^{-\frac{d}{2}})$.\\

	Therefore, we have showed that the part of the sum \eqref{4thmom} which is over sequences of Type 1 ($\sfT_1$), Type 2 ($\sfT_2$) or Type 3 ($\sfT_3$) is negligible in the $N \to \infty$ limit. Thus, the proof is complete.

\end{proof}

\begin{proof}[Proof of Theorem \ref{gaussianity}]
	By Proposition \ref{lincomb} we obtain that $Z^{\ms{\leq}  M}_{N,\beta}(\varphi)$ converges in distribution to a centered Gaussian random variable $\mathcal{G}_M$ as $N \to \infty$, with variance equal to
	\begin{align*}
		\bbvar\big[\cG_M\big]=\sum_{k=1}^M  \,\mathcal{C}^{(k)}_{\beta} \int_0^1 \dd t \,\int_{\R^{2d} } \dd x \, \dd y\, \varphi(x) g_{\frac{2t}{d}}(x-y) \varphi(y) .
	\end{align*}
	We also have that
	\begin{align*}
		\lim_{M \to \infty} \bbvar\big[\cG_M\big]=\sum_{k=1}^\infty  \,\mathcal{C}^{(k)}_{\beta} \int_0^1 \dd t \,\int_{\R^{2d} } \dd x \, \dd y\, \varphi(x) g_{\frac{2t}{d}}(x-y) \varphi(y)=\bbvar{\mathcal{Z}_{\beta}(\varphi)} \, ,
	\end{align*}
	where $\mathcal{Z}_{\beta}(\varphi)$ is the random variable defined by Theorem \ref{gaussianity}, since
	\begin{align*}
		\sum_{k=1}^\infty  \,\mathcal{C}^{(k)}_{\beta}=\sigma^2(\beta) \sum_{k=1}^{\infty}\sigma(\beta)^{2(k-1)} \sumtwo{1\leq \ell_1 <...<\ell_{k-1}}{\ell_0:=0} \prod_{i=1}^{k-1} q_{2(\ell_i-\ell_{i-1})}(0)=\sigma^2(\beta) \E[e^{\lambda_2(\beta) \cL_{\infty}}] \, .
	\end{align*}
	Combining this with Lemma \ref{truncation}, we obtain the conclusion of Theorem \ref{gaussianity}, that is $Z_{N,\beta}(\varphi)\xrightarrow[N \to \infty]{(d)} \mathcal{Z}_{\beta}(\varphi)$.
\end{proof}

\section{Edwards-Wilkinson fluctuations for the log-partition function}
\justify

In this section we prove Theorem \ref{loggaussianity}, namely, the Edwards-Wilkinson fluctuations for the log-partition function.

We will need to impose one more condition to the random environment for technical reasons. Specifically, we require that the law of the random environment satisfies a concentration inequality. In particular, we assume that there exists an exponent $\gamma>1$ and constants $C_1,C_2>0$, such that for every $n \in \N$,  $1$-Lipschitz function $f:\R^n \to \R$ and i.i.d. random variables $\omega_1,...,\omega_n$ having law $\bbP$, we have that
\begin{align} \label{concentration}
	\bbP \Big( \big|f(\omega_1,...,\omega_n)-M_f\big| \geq t \Big) \leq C_1 \exp\Big( -\frac{t^\gamma}{C_2} \Big)\, ,
\end{align}
where $M_f$ denotes a median of $f(\omega_1,...,\omega_n)$. One can replace the median by $\bbE[f(\omega_1,...,\omega_n)]$, by changing the constants $C_1,C_2$ appropriately. {Condition} \eqref{concentration} is satisfied if $\omega$ has a density of the form $\exp(-V(\cdot)+U(\cdot))$, where $V$ is uniformly strictly convex and $U$ is bounded, see \cite{Led01}. {It also} enables us to formulate the following left-tail estimate. For $\Lambda \subseteq \N \times \Z^d$, let $Z^{\Lambda}_{N,\beta}(x)$ denote the partition function which contains disorder only from $\Lambda$, {that is
		\begin{align*}
			Z^{\Lambda}_{N,\beta}(x)= \E_x\Big[\exp\big(\sum_{(n,z) \in \Lambda } \big(\beta \omega_{n,z}-\lambda(\beta)\big)  \ind_{S_n=z} \big) \Big] \, .
		\end{align*} Then, we have the following Proposition: }
\begin{proposition}[Left-tail estimate] \label{lefttail}
	For every $\beta \in (0,\beta_{L^2})$ there exists a constant $c_{\beta}>0$, such that: for every $N \in \N$, $\Lambda \subseteq \N \times \Z^d$, one has that $\forall t \geq0$
	\begin{align*}
		{\bbP\Big(\log Z^{\Lambda}_{N,\beta}(x) \leq -t\Big) \leq c_\beta \, \exp\Big(-\frac{t^\gamma}{c_\beta}\Big) } \, ,
	\end{align*}
	where $\gamma$, is the exponent in \eqref{concentration}.
\end{proposition}

Proposition \ref{lefttail} provides an additional advantage to our analysis and that is the existence of all negative moments for the partition function and all positive moments for the log-partition function. In particular, the following is in our disposal,
\begin{proposition} \label{negmom}
	For every $\beta \in (0,\beta_{L^2})$, $\Lambda \subseteq \N \times \Z^d$ and $p>0$ one has that there exist constants $C^{\ms{\sfn \sfe\sfg}}_{p,\beta},C^{\ms{\sfl \sfo \sfg}}_{p,\beta}  $ such that
	\begin{align*}
		 & \sup_{N \in \N} \bbE \Big[ \big(Z^{\Lambda}_{N,\beta}(x)\big)^{-p}\Big] \leq C^{\ms{\sfn \sfe\sfg}}_{p,\beta} \, , \notag  \\
		 & {\sup_{N \in \N} \bbE \Big[\big| \log Z^{\Lambda}_{N,\beta}(x)\big|^{p}\Big] \leq C^{\ms{\sfl \sfo \sfg}}_{p,\beta} \, . }
	\end{align*}
\end{proposition}
We refer to \cite{CSZ18b} for the proofs of Propositions \ref{lefttail}, \ref{negmom}, as the method presented there can be followed exactly to give those results in our case. For Proposition \ref{lefttail} see also \cite{CTT17}, where this method appeared in the context of {pinning models.}

We will also need the existence of $2+\delta$ moments for the partition function. This can be established with the use of hypercontractivity, for which we refer to {Section 3 of} \cite{CSZ18b} for a detailed exposition. In particular, we have the following proposition:
\begin{proposition} \label{hyperc}
	For every $\beta \in (0,\beta_{L^2})$, there exists $p=p_{\beta} \in (2,\infty)$, such that
	\begin{align*}
		\sup_{N \in \N} \bbE\Big[\big(Z_{N,\beta}(x)\big)^p\Big] <\infty \, .
	\end{align*}
\end{proposition}

Let us proceed to the sketch of the proof for the {Edwards-Wilkinson} fluctuations for the log-partition function.
For every $ x \in \Z^d $ we define a space-time window around $x$ as follows
\begin{align} \label{Asets}
	A_N^x=\Big\{(n,z):\:1 \leq n \leq N^{\epsilon}, |x-z|< N^{\frac{\epsilon}{ 2}+{\alpha} }  \Big \} \, ,
\end{align}
{for $\epsilon \in (0,1)$ and $\alpha \in (0,\frac{\epsilon}{2})$, much smaller than $\frac{\epsilon}{2}$. These scale parameters are going to be determined later in the proofs.}
{We decompose} the partition function as:
\begin{align*}
	Z_{N,\beta}(x)=Z^{A}_{N,\beta}(x)+\hat{Z}^{A}_{N,\beta}(x)  \, ,
\end{align*}
where
\begin{align*}
	{Z^{A}_{N,\beta}(x)=\E_x\Big[\exp\big(\sum_{(n,z) \in A_N^x } \big(\beta \omega_{n,z}-\lambda(\beta)\big)  \ind_{S_n=z} \big) \Big] \, ,}
\end{align*}
is the partition function which contains disorder only from the set $A_N^x$, while the remainder, {$\hat{Z}^{A}_{N,\beta}(x)=Z_{N,\beta}(x)-Z^{A}_{N,\beta}(x)$}, necessarily contains disorder from points outside of $A_N^x$ in its chaos decomposition, see also \cite{CSZ18b}, Section 2, for analoguous definitions. The chaos expansions of $Z^{A}_{N,\beta}(x),\hat{Z}^{A}_{N,\beta}(x)$ are
\begin{align} \label{Z_A_chaos}
	Z^{A}_{N,\beta}(x)=1+\sum_{k\geq 1} \sigma^k \sum_{(n_i,z_i)_{1\leq i\leq k}\subset A_N^x} q_{n_1}(z_1-x)\prod_{i=2}^k q_{n_i-n_{i-1}}(z_i-z_{i-1}) \prod_{i=1}^k \eta_{n_i,z_i} \, ,
\end{align}
and
\begin{align} \label{Z_hat_chaos}
	\hat{Z}^{A}_{N,\beta}(x)=\sum_{k\geq1}\sigma^k \sum_{(n_i,z_i)_{1\leq i\leq k}\cap (A_N^x)^{\ms c}\neq \eset} q_{n_1}(z_1-x)\prod_{i=2}^k q_{n_i-n_{i-1}}(z_i-z_{i-1}) \prod_{i=1}^k \eta_{n_i,z_i} \, .
\end{align}
We can then write, for every $x \in \Z^d$,

\begin{align} \label{decomposition}
	{	\log Z_{N,\beta}(x)=\log  Z^{A}_{N,\beta}(x)+\log\bigg(1+ \frac{\hat{Z}^{A}_{N,\beta}(x)}{Z^{A}_{N,\beta}(x)}\bigg) } \, .
\end{align}

The first step we take is to show that the contribution of the term {$\log  Z^{A}_{N,\beta}(x)$} to the fluctuations of {$\log Z_{N,\beta}(x)$} is negligible, when averaged over $x$, in the following sense
\begin{proposition} \label{first}
	Let  $\varphi \in C_c(\R^d)$ be a test function. Then, we have that
	\begin{align} \label{logA}
		N^{\frac{d-2}{4}}  \sum_{x \in \Z^d} \varphi_N(x)\, {\Big(\log   Z^{A}_{N,\beta}(x)-\bbE\big[\log  Z^{A}_{N,\beta}(x) \big] \Big)} \xrightarrow[N \to \infty]{L^2(\bbP)} 0 \, .
	\end{align}
\end{proposition}
The second step is to prove {that we} can replace
$\displaystyle \log\Big(1+ \tfrac{\hat{Z}^{A}_{N,\beta}(x)}{Z^{A}_{N,\beta}(x)}\Big) $ by $\tfrac{\hat{Z}^{A}_{N,\beta}(x)}{Z^{A}_{N,\beta}(x)}
$.
In particular, if we define $O_N(x):=  \log\Big(1+ \tfrac{\hat{Z}^{A}_{N,\beta}(x)}{Z^{A}_{N,\beta}(x)}\Big) - \tfrac{\hat{Z}^{A}_{N,\beta}(x)}{Z^{A}_{N,\beta}(x)}  $, then we will show {that}
\begin{proposition} \label{remainder}
	Let  $\varphi \in C_c(\R^d)$ be a test function. Then, we have that
	\begin{align*}
		N^{\frac{d-2}{4}}  \sum_{x \in \Z^d} \varphi_N(x)\, \Big(O_N(x)-\bbE[O_N(x)] \Big) \xrightarrow[N \to \infty]{L^1(\bbP)} 0 \, .
	\end{align*}
\end{proposition}

\noindent Therefore, we need to identify the fluctuations of the quotient
$\displaystyle \frac{\hat{Z}^{A}_{N,\beta}(x)}{Z^{A}_{N,\beta}(x)} $. Note that this quantity has mean zero since each term in the chaos expansion of $\hat{Z}^{A}_{N,\beta}(x)$ contains disorder outside $A_N^x$, see \eqref{Z_hat_chaos}. To study the fluctuations of $\displaystyle \frac{\hat{Z}^{A}_{N,\beta}(x)}{Z^{A}_{N,\beta}(x)}$ we define, for a suitable $\rho \in (\epsilon,1)$, the set
\begin{align} \label{rhodef}
	B^{\geq}_N=\big( (N^\rho,N] \cap \N \big) \times \Z^d\, ,
\end{align}
and show, employing the local limit theorem for random walks,  that the asymptotic factorisation $\hat{Z}^{A}_{N,\beta}(x) \approx Z^{A}_N(x)\big(Z^{B^{\geq}}_{N,\beta}(x)-1 \big)$ takes place when we average over $x$, namely
\begin{proposition} \label{factorisation}
	Let  $\varphi \in C_c(\R^d)$ be a test function. Then, we have that
	\begin{align*}
		N^{\frac{d-2}{4}}  \sum_{x \in \Z^d} \varphi_N(x)\, \bigg(\frac{\hat{Z}^{A}_{N,\beta}(x)}{Z^{A}_{N,\beta}(x)}- \big(Z_{N,\beta}^{B^{\geq}}(x)-1\big) \bigg) \xrightarrow[N \to \infty]{L^1(\bbP)} 0 \, .
	\end{align*}
\end{proposition}

The last step is to show that the fluctuations of $Z^{B^{\geq}}_{N,\beta}(x)-1$ when averaged over $x$, are Gaussian with variance equal to that of Theorem \ref{gaussianity}, namely

\begin{proposition} \label{fourth2}
	Let  $\varphi \in C_c(\R^d)$ be a test function. Then, we have the following convergence in distribution,
	\begin{align*}
		N^{\frac{d-2}{4}}  \sum_{x \in \Z^d} \varphi_N(x)\,
		\big(Z^{B^{\geq}}_{N,\beta}(x)-1 \big) \xrightarrow[N \to \infty]{(d)} \cZ_{\beta}(\varphi) \, ,
	\end{align*}
	where $\cZ_{\beta}(\varphi)$ is the centered normal random variable appearing in Theorem \ref{gaussianity}.
\end{proposition}

We begin with the proof of Proposition \ref{first}.

\begin{proof}[Proof of Proposition \ref{first}]

	It suffices to restrict the summation and show that
	\begin{align} \label{firstsum}
		N^{\frac{d}{2}-1} \sum_{|x-y|\leq 2N^{\frac{\epsilon}{ 2}+\alpha}} \varphi_N(x,y)\,{ \cov\big( \log  Z^{A}_{N,\beta}(x), \log  Z^{A}_{N,\beta}(y) \big) }\xrightarrow[N \to \infty]{} \ 0 \, ,
	\end{align}
	because, by the definition of the sets $A_N^x$, if $|x-y|> 2N^{\frac{\epsilon}{ 2}+\alpha}$, then {$\log  Z^{A}_{N,\beta}(x)$ and $\log  Z^{A}_{N,\beta}(y)$} are independent, { so the covariance vanishes.} The proof will be divided in four steps.\\

	\noindent  \textbf{(Step 1) - Martingale decomposition.}
	We will  expand the covariance appearing in \eqref{firstsum} by using a martingale difference decomposition.
	Let $\{\omega_{a_1},\omega_{a_2},...\}$ be an arbitrary enumeration of the disorder indexed by ${\N} \times \Z^d$. We can then define a filtration $(\mathcal{F}_j)_{j\geq1}$, such that $\mathcal{F}_j=\sigma(\omega_{a_1},...,\omega_{a_j})$. We define also $\mathcal{F}_0=\{\eset,\Omega\}$, where $\Omega$ is the underlying sample space where the random variables $(\omega_{n,z})_{(n,z) \in \N \times \Z^d}$, are defined. Using this filtration we can write the difference {$\log  Z^{A}_{N,\beta}(x)-\bbE[\log  Z^{A}_{N,\beta}(x)]$} as
	a telescoping sum, namely

	\begin{align} \label{Z-decomp}
		{\log  Z^{A}_{N,\beta}(x)-\bbE\big[\log  Z^{A}_{N,\beta}(x)\big]= \sum_{j \geq 1} \Big(\bbE[\log  Z^{A}_{N,\beta}(x)| \mathcal{F}_j] - \bbE[\log  Z^{A}_{N,\beta}(x)| \mathcal{F}_{j-1}]\Big) \, . }
	\end{align}
	{Then, using} the shorthand notation {$D_j(x)= \bbE[\log  Z^{A}_{N,\beta}(x)| \mathcal{F}_j] - \bbE[\log  Z^{A}_{N,\beta}(x)| \mathcal{F}_{j-1}]$}  we have that:
	\begin{align*}
		{	\cov\big(\log  Z^{A}_{N,\beta}(x),\log Z^{A}_{N,\beta}(y)\big)= }
		\sum_{k,j \geq 1} \bbE[ D_k(x)D_j(y)]
		=  \sum_{j \geq 1} \bbE[ D_j(x)D_j(y)]  \, ,
	\end{align*}
	where we used the fact that if $j<k$, conditioning on $\mathcal{F}_j$ shows that $D_j(x), D_k(y)$ are orthogonal in $L^2(\bbP)$. Therefore, we are able to rewrite the sum in~\eqref{firstsum} as
	\begin{align} \label{scndsum}
		N^{\frac{d}{2}-1} \sum_{|x-y|\leq 2N^{\frac{\epsilon}{ 2}+\alpha}} \varphi_N(x,y)\, \sum_{j \geq 1} \bbE[ D_j(x)D_j(y)]  \, .
	\end{align}
	One has to make an important observation at this point. If $a_j$ is not contained in $A_N^x$, then $D_j(x)=0$. Hence, the rightmost sum in ~\eqref{scndsum} is non-zero only for $j\geq1$, such that $a_j \in A_N^x \cap A_N^y$. \\

	\noindent  \textbf{(Step 2) - Resampling.}
	Let us now look more closely to {the} martingale differences { $D_j(x)$}. We {will rewrite} them in a closed form using a local resampling scheme. Fix $j$ such that {$a_j \in A_N^x \cap A_N^y$}. We can write
	\begin{align*}
		{\log Z^{A}_{N,\beta}(x)=\log Z_{N,\beta}^{A,\sfT_{a_j}}(x)+\big(\log Z^{A}_{N,\beta}(x)-\log Z_{N,\beta}^{A,\sfT_{a_j}}(x)\big) \, ,}
	\end{align*}
	where we used the notation $\sfT_{a_j}\omega$ to denote the {disorder} environment, where the $\omega_{a_j}$ {disorder} variable has been replaced {by} an independent copy $\tilde \omega_{a_j}$. {We also have}
	\begin{align*}
		\log Z^{A}_{N,\beta}(x)
		=                       \bbE_{\tilde{\omega}} \big[\log Z^{A,\sfT_{a_j}}_{N,\beta}(x)\big]+ \bbE_{\tilde{\omega}} \Bigg[\log\Bigg( \frac{Z^{A}_{N,\beta}(x)}{Z^{A,\sfT_{a_j}}_{N,\beta}(x)}\Bigg)\Bigg] \, ,
	\end{align*}
	where $\bbE_{\tilde \omega}[\cdot]$ denotes the expectation with respect to the resampled noise
	(which in the expression above reduces to just integration of $\tilde\omega_{a_j}$), since the left hand side of the above equation does not depend on $\tilde{\omega}$. We note that the following equality is true:
	\begin{align} \label{reduct}
		\bbE\Big[ \bbE_{\tilde{\omega}} \big[\log Z^{A,\sfT_{a_j}}_{N,\beta}(x)\big]\Big|\mathcal{F}_j\Big]=  \bbE\Big[ \log Z^{A}_{N,\beta}(x)\Big|\mathcal{F}_{j-1}\Big] \, .
	\end{align}
	One can see this by rewriting both sides of the equation, using the fact that, given a random function $f(\omega)$, where $\omega=(\omega_k)_{k\geq 1}$ is a sequence of i.i.d. random variables, then
	$ \bbE[f(\omega)|\mathcal{F}_j]=\int f(\omega) \prod_{k>j} \bbP( \dd \omega_{k}) $.

	\noindent In conclusion, we have managed to rewrite the difference $D_j(x)$ as
	\begin{align} \label{difference}
		D_j(x)= \bbE\bigg[\bbE_{\tilde \omega}\Big[ \log Z_{N,\beta}^{A}(x)-\log Z_{N,\beta}^{A,\sfT_{a_j}}(x) \Big]\Big|\mathcal{F}_j\bigg]\, .
	\end{align}
	The next step shows how we can remove the logarithms.	\\

	\noindent \textbf{(Step 3) - Removing the logarithms.}  { We fix a positive number $\sfh \in (0, \frac{1-\epsilon}{2})$ and for $x \in \Z^d$, we define }
	\begin{align} \label{ejsets}
		E_j(x):=\Big \{Z_{N,\beta}^{A}(x),Z_{N,\beta}^{A,\sfT_{a_j}}(x) \geq {N^{-\sfh}} \Big \}  \, .
	\end{align}
	We then decompose $D_j(x)$ as follows
	\begin{align*}
		D_j(x)=\, & \bbE\bigg[\bbE_{\tilde \omega}\Big[ \Big ( \log Z_{N,\beta}^{A}(x) -\log Z_{N,\beta}^{A,\sfT_{a_j}}(x)  \Big) \ind_{E_j(x)} \Big]\Big|\mathcal{F}_j\bigg] \notag    \\
		+         & \bbE\bigg[\bbE_{\tilde \omega}\Big[ \Big ( \log Z_{N,\beta}^{A}(x) -\log Z_{N,\beta}^{A,\sfT_{a_j}}(x)  \Big) \ind_{E^c_j(x)} \Big]\Big|\mathcal{F}_j\bigg]    \, .
	\end{align*}
	{We hereafter use the notation }
	\begin{align*}
		 & {D^{\ms (\sfb)}_j(x)}:= \bbE\bigg[\bbE_{\tilde \omega}\Big[ \Big ( \log Z_{N,\beta}^{A}(x) -\log Z_{N,\beta}^{A,\sfT_{a_j}}(x)  \Big) \ind_{E_j(x)} \Big]\Big|\mathcal{F}_j\bigg] \, ,\notag \\
		 & {D^{\ms (\sfs)}_j(x)}:=\bbE\bigg[\bbE_{\tilde \omega}\Big[ \Big ( \log Z_{N,\beta}^{A}(x) -\log Z_{N,\beta}^{A,\sfT_{a_j}}(x)  \Big) \ind_{E^c_j(x)} \Big]\Big|\mathcal{F}_j\bigg] \, .
	\end{align*}
	for the two summands of this decomposition.
	We then have that
	\begin{align} \label{martexp}
		\sum_{{j\, :\, a_j \in A_N^x \cap A_N^y}} \bbE \Big [ D_j(x)D_j(y) \Big]=
		  & \sum_{{j\, :\, a_j \in A_N^x \cap A_N^y}} \bbE \Big [ D^{\ms (\sfb)}_j(x)D^{\ms (\sfb)}_j(y) \Big] \notag                                                                                                                \\
		+ & \sum_{{j\, :\, a_j \in A_N^x \cap A_N^y}} \bbE \Big [ D^{\ms (\sfb)}_j(x)D^{\ms (\sfs)}_j(y) \Big]+\bbE \Big [ D^{\ms (\sfs)}_j(x)D^{\ms (\sfb)}_j(y) \Big]+\bbE \Big [ D^{\ms (\sfs)}_j(x)D^{\ms (\sfs)}_j(y) \Big] \,.
	\end{align}
	We will first prove that
	\begin{align} \label{bigbigestimate}
		N^{\frac{d}{2}-1} \sum_{|x-y|\leq 2N^{\frac{\epsilon}{ 2}+\alpha}} \varphi_N(x,y)\, \sum_{{j\, :\, a_j \in A_N^x \cap A_N^y}} \bbE\Big[ D^{\ms (\sfb)}_j(x)D^{\ms (\sfb)}_j(y)\Big] \xrightarrow[N\to \infty]{} 0  \, .
	\end{align}
	Note that
	\begin{align} \label{difflogbound}
		\big|D^{\ms (\sfb)}_j(x)\big| \leq & \, \bbE\bigg[\bbE_{\tilde \omega}\Big[ \, \big | \log Z_{N,\beta}^{A}(x) -\log Z_{N,\beta}^{A,\sfT_{a_j}}(x)  \big| \ind_{E_j(x)} \, \Big]\Big|\mathcal{F}_j\bigg] \notag \\
		\leq                               & \, N^{\sfh}\, \bbE\bigg[\bbE_{\tilde \omega}\Big[\, \big | Z_{N,\beta}^{A}(x)-Z_{N,\beta}^{A,\sfT_{a_j}}(x) \big| \ind_{E_j(x)} \,\Big]\Big|\mathcal{F}_j\bigg] \notag    \\
		\leq                               & \,N^{\sfh} \, \bbE\bigg[\bbE_{\tilde \omega}\Big[ \, \big | Z_{N,\beta}^{A}(x)-Z_{N,\beta}^{A,\sfT_{a_j}}(x) \big| \, \Big]\Big|\mathcal{F}_j\bigg]          \, ,
	\end{align}
	where we used the fact that if $x,y \in [t,\infty)$ for some positive $t>0$, then $|\log x-\log y| \leq \frac{1}{t} |x-y|$, {for the second inequality}. For the sake of the presentation, we shall adopt the notation
	\begin{align} \label{Wjdef}
		{\bbE\bigg[\bbE_{\tilde \omega}\Big[ \, \big | Z_{N,\beta}^{A}(x)-Z_{N,\beta}^{A,\sfT_{a_j}}(x) \big| \, \Big]\Big|\mathcal{F}_j\bigg]:= \sfW_{j}(x)} \, ,
	\end{align}
	by omitting the dependence in $N$. By using the estimate \eqref{difflogbound} and summing {over} $j\, :\, a_j \in A_N^x \cap A_N^y$ we deduce that
	\begin{align} \label{boundnologs}
		\sum_{ j\, :\, a_j \in A_N^x \cap A_N^y }  \bbE\Big[ \big|D^{\ms (\sfb)}_j(x) D^{\ms (\sfb)}_j(y)\big|\Big]
		\leq   N^{2 \sfh} \sum_{ j\, :\, a_j \in A_N^x \cap A_N^y } \bbE \Big[ \sfW_j(x)\sfW_j(y) \Big]  \, .
	\end{align}
	If we denote by $S^x$ the path of a random walk starting at $x$ we have
	\begin{align} \label{partfdiffj}
		Z_{N,\beta}^{A}(x)-Z_{N,\beta}^{A,\sfT_{a_j}}(x)=\sigma(\beta) (\eta_{a_j}-\tilde{\eta}_{a_j}) \E_{x}\big[e^{{\sfH}^{x}_{A \fgebackslash a_j}} \ind_{a_j \in S^x }\big]  \, ,
	\end{align}
	where
	\begin{align} \label{Hnotajdef}
		{\sfH}^{x}_{A \fgebackslash a_j}(\omega):=\sumtwo{a \in A^{x}_{N}}{a \neq a_j} \big[ \beta \omega_{a}-\lambda(\beta)\big ] \ind_{a \in S^x}  \, ,
	\end{align}
	and recall from \eqref{etas} that
	\begin{align*}
		\eta_{a_j}=\frac{e^{\beta \omega_{a_j} - \lambda(\beta)}-1}{\sigma(\beta)} \quad \text{   and  } \quad  \tilde{\eta} _{a_j}=\frac{e^{\beta \tilde{\omega}_{a_j} - \lambda(\beta)}-1}{\sigma(\beta)} \, .
	\end{align*}
	At this point, we will bound $\sfW_j(x)$. By \eqref{Wjdef} and \eqref{partfdiffj} we have that
	\begin{align*}
		\sfW_j(x)= & \int \prod_{k>j}  \bbP(\dd \omega_{a_k})   				\int  \bbP(\dd \tilde{\omega}_{a_j})  \, \big|Z_{N,\beta}^{A}(x)-Z_{N,\beta}^{A,\sfT_{a_j}}(x)\big| \notag                                                                         \\
		=          & \int \prod_{k>j}  \bbP(\dd \omega_{a_k})   \int    \bbP(\dd \tilde{\omega}_{a_j})\, \sigma(\beta)\, |\eta_{a_j}-\tilde{\eta}_{a_j}| \, \E_{x}\big[e^{{\sfH}^{x}_{A \fgebackslash a_j} (\omega)} \ind_{a_j \in S^x }\big] \, .
	\end{align*}
	We will perform this integration in steps.	The expectation, $\E_{x}\big[e^{{\sfH}^{x}_{A \fgebackslash a_j} (\omega)} \ind_{a_j \in S^x }\big]$, does not depend on $\tilde{\omega}_{a_j}$ by \eqref{Hnotajdef}, and we have by triangle inequality
	\begin{align} \label{intWj1}
		\int   \bbP(\dd \tilde{\omega}_{a_j})\, \sigma(\beta)\, |\eta_{a_j}-\tilde{\eta}_{a_j}| \leq & \,\sigma(\beta) \Big(|\eta_{a_j}|+1\Big) \, .
	\end{align}
	Furthermore, by exchanging the integral and the expectation we deduce that
	\begin{align} \label{intWj2}
		\int \prod_{k>j}  \bbP(\dd \omega_{a_k}) \E_{x}\big[e^{{\sfH}^{x}_{A \fgebackslash a_j} (\omega)} \ind_{a_j \in S^x }\big] =  \E_{x}\big[e^{{\sfH}^{x}_{A \cap \{a_1,...,a_{j-1}\} } (\omega)} \ind_{a_j \in S^x }\big] \, ,
	\end{align}
	where
	\begin{align*}
		{\sfH}^{x}_{A \cap \{a_1,...,a_{j-1}\}}(\omega):=\sumtwo{1\leq k \leq j-1 \vspace{0.3mm}}{a_k \in A_N^x} \big[ \beta \omega_{a_k}-\lambda(\beta)\big ] \ind_{a_k \in S^x}  \, .
	\end{align*}
	If $j=1$, we set the corresponding energy to be equal to $0$.
	Hence, combining \eqref{intWj1} and \eqref{intWj2} we obtain that
	\begin{align*}
		\sfW_j(x)\leq \sigma(\beta) \Big(|\eta_{a_j}|+1\Big) \E_{x}\big[e^{{\sfH}^{x}_{A \cap \{a_1,...,a_{j-1}\} } (\omega)} \ind_{a_j \in S^x }\big] \, .
	\end{align*}
	Therefore, by Fubini we get that
	\begin{align*}
		\sfW_j(x)\,\sfW_j(y) \leq \sigma^{2}(\beta) \Big(|\eta_{a_j}|+1\Big)^2 \,  \E_{x,y}\big[e^{{\sfH}^{x}_{A \cap \{a_1,...,a_{j-1}\} } (\omega)+{\sfH}^{y}_{A \cap \{a_1,...,a_{j-1}\}} (\omega)} \ind_{a_j \in S^x \cap S^y }\big] \, ,
	\end{align*}
	which after taking the expectation $\bbE[\,\cdot\,]$ leads to
	\begin{align} \label{appearanceofloctime}
		\bbE \Big[\sfW_j(x)\,\sfW_j(y)\Big]  \leq 4 \sigma^2(\beta) \E_{x,y}\big[e^{\lambda_2(\beta)\mathcal{L}_{N}(x,y)} \ind_{a_j \in S^x \cap S^y } \big] \, .
	\end{align}
	Therefore, by summing over $j\, :\, a_j \in A_N^x \cap A_N^y$ we deduce that
	\begin{align}
		\sum_{j\, :\, a_j \in A_N^x \cap A_N^y} \bbE \Big[\sfW_j(x)\,\sfW_j(y)\Big]  \leq 4 \sigma^2(\beta) \E_{x,y}\big[e^{\lambda_2(\beta)\mathcal{L}_{N}(x,y)} \mathcal{L}_{N^{\epsilon}}(x,y) \big] \, .
	\end{align}
	Note that the rightmost overlap, $\mathcal{L}_{N^{\epsilon}}(x,y)$, goes up to time $N^\epsilon$, since by \eqref{Asets}, for every $j\, :\, a_j \in A_N^x \cap A_N^y$,  $a_j$ has time index $t\leq N^{\epsilon}$ therefore,
	\begin{align*}
		\sum_{j\, :\, a_j \in A_N^x \cap A_N^y} \ind_{a_j \in S^x \cap S^y} \leq \sum_{n=1}^{N^\epsilon} \ind_{S_n^x=S_n^y}:=\cL_{N^{\epsilon}}(x,y) \, .
	\end{align*}
	Recalling \eqref{boundnologs} we get that
	\begin{align*}
		\sum_{ j\, :\, a_j \in A_N^x \cap A_N^y }  \bbE\Big[ \big|D^{\ms (\sfb)}_j(x) D^{\ms (\sfb)}_j(y)\big|\Big] \leq \, {N^{2\sfh}} \,
		4 \sigma^2(\beta) \,  \E_{x,y}\big[e^{\lambda_2(\beta)\mathcal{L}_{N}(x,y)} \mathcal{L}_{N^{\epsilon}}(x,y) \big] \, .
	\end{align*}
	So far, we have shown that
	\begin{align} \label{loctimeestimate}
		     & N^{\frac{d}{2}-1} \sum_{|x-y|\leq 2N^{\frac{\epsilon}{ 2}+\alpha}} \varphi_N(x,y)\, \sum_{j\, :\, a_j \in A_N^x \cap A_N^y} \bbE\Big[ D^{\ms (\sfb)}_j(x)D^{\ms (\sfb)}_j(y)\Big] \notag                                  \\
		\leq & \, 4 \sigma^2(\beta) \,N^{\frac{d}{2}-1+2\sfh} \,\sum_{|x-y|\leq 2N^{\frac{\epsilon}{ 2}+\alpha}} \varphi_N(x,y)\,  \,  \E_{x,y}\big[e^{\lambda_2(\beta)\mathcal{L}_{N}(x,y)} \mathcal{L}_{N^{\epsilon}}(x,y) \big]  \, .
	\end{align}
	Therefore, to establish \eqref{bigbigestimate}, we derive an estimate for $\E_{x,y}\big[e^{\lambda_2(\beta)\mathcal{L}_{N}(x,y)} \mathcal{L}_{N^{\epsilon}}(x,y) \big]$. Let us denote by $\tau_{x,y}$ the first meeting time of two independent random walks starting from $x,y \in \Z^d$, respectively. By conditioning on $\tau_{x,y}$  we obtain
	\begin{align*}
		\E_{x,y}\Big[e^{\lambda_2(\beta)\mathcal{L}_{N}(x,y)} \mathcal{L}_{N^\epsilon}(x,y)\Big] & =\sum_{n=1}^{N^{\epsilon}} \E_{x,y}\Big[e^{\lambda_2(\beta)\mathcal{L}_{N}(x,y)} \mathcal{L}_{N^\epsilon}(x,y)|\tau_{x,y}=n\Big]\P(\tau_{x,y}=n) \, .
	\end{align*}
	Using the Markov property we obtain
	\begin{align*}
		\sum_{n=1}^{N^{\epsilon}} \E_{x,y}\Big[e^{\lambda_2(\beta)\mathcal{L}_{N}(x,y)} \mathcal{L}_{N^\epsilon}(x,y)|\tau_{x,y}=n\Big]\P(\tau_{x,y}=n) \hspace{-0.1cm}
		= \hspace{-0.15cm} \sum_{n=1}^{N^{\epsilon}} \E \Big[e^{\lambda_2(\beta)(1+\mathcal{L}_{N-n})} \big(1+\mathcal{L}_{N^\epsilon-n}\big)\Big]\P(\tau_{x,y}=n)     \, .
	\end{align*}
	For every $1 \leq n \leq N^{\epsilon}$, we can bound the expectation
	\begin{align*}
		\E \Big[e^{\lambda_2(\beta)(1+\mathcal{L}_{N-n})} \big(1+\mathcal{L}_{N^\epsilon-n}\big)\Big]\leq e^{\lambda_2(\beta)}\Big (\E \big [e^{\lambda_2(\beta)\mathcal{L}_{\infty}} \big] + \E \big [e^{\lambda_2(\beta)\mathcal{L}_{\infty}}\mathcal{L}_{\infty} \big] \Big) :=c(\beta) <\infty \, ,
	\end{align*}
	because $\beta \in (0,\beta_{L^2})$, see \eqref{secondmoment}. Moreover, we have that \begin{align*}
		\P(\tau_{x,y}=n) \leq \sum_{z \in \Z^d} q_n(z-x)q_n(z-y) =  q_{2n}(x-y) \, .
	\end{align*}
	Therefore,
	\begin{align} \label{kerestimate}
		\E_{x,y}\Big[e^{\lambda_2(\beta)\mathcal{L}_{N}(x,y)} \mathcal{L}_{N^\epsilon}(x,y)\Big] \leq c(\beta)  \sum_{n=1}^{N^{\epsilon}} q_{2n}(x-y)  \, .
	\end{align}
	Recalling \eqref{bigbigestimate}, \eqref{loctimeestimate} and \eqref{kerestimate}, in order to conclude Step 3, we need to show that

	\begin{align*}
		N^{\frac{d}{2}-1+2\sfh} \sum_{|x-y| \leq 2N^{\frac{\epsilon}{2}+\alpha}} \varphi_N(x,y)\,
		\sum_{n=1}^{N^{\epsilon}} q_{2n}(x-y) \xrightarrow[N \to \infty]{} 0 \, .
	\end{align*}
	We bound $\varphi(\frac{y}{\sqrt{N}})$ by its supremum norm and use the fact that $\sum_{z \in \Z^d} q_{2n}(z)=1$, to obtain that
	\begin{align} \label{erbound}
		N^{\frac{d}{2}-1+2\sfh} \sum_{|x-y| \leq 2N^{\frac{\epsilon}{2}+\alpha}} \varphi_N(x,y)\,
		\sum_{n=1}^{N^{\epsilon}} q_{2n}(x-y)
		\leq \norm{\varphi}_{\infty}N^{2\sfh+ \epsilon-1} \sum_{x \in \Z^d} \varphi_N(x)\,
		\leq  \norm{\varphi}_{\infty} \norm{\varphi}_{1 } N^{2\sfh+ \epsilon-1} \, .
	\end{align}
	Since $\sfh \in (0,\frac{1-\epsilon}{2})$, we have that $2\sfh+\epsilon<1$, hence the last bound vanishes as $N \to \infty$, which concludes the proof of \eqref{bigbigestimate}.

	\noindent \textbf{(Step 4) - Events of small partition functions.}
	Let us see how one can treat the rest of the terms in the expansion \eqref{martexp}, which involve the complementary events $E^c_j(x),E^c_j(y)$, recall their definition from \eqref{ejsets}. We need to show that
	\begin{align*}
		 & N^{\frac{d}{2}-1} \sum_{|x-y|\leq 2N^{\frac{\epsilon}{ 2}+\alpha}} \varphi_N(x,y)\, \sum_{ j\, :\, a_j \in A_N^x \cap A_N^y }  \bbE\Big[ D^{\ms (\sfs)}_j(x) D^{\ms (\sfb)}_j(y)\Big] \xrightarrow[N \to \infty]{} 0  \ , \notag  \vspace{0.2cm} \\
		 & N^{\frac{d}{2}-1} \sum_{|x-y|\leq 2N^{\frac{\epsilon}{ 2}+\alpha}} \varphi_N(x,y)\, \sum_{ j\, :\, a_j \in A_N^x \cap A_N^y }  \bbE\Big[ D^{\ms (\sfb)}_j(x) D^{\ms (\sfs)}_j(y)\Big] \xrightarrow[N \to \infty]{} 0  \ , \vspace{0.2cm} \notag  \\
		 & N^{\frac{d}{2}-1} \sum_{|x-y|\leq 2N^{\frac{\epsilon}{ 2}+\alpha}} \varphi_N(x,y)\, \sum_{ j\, :\, a_j \in A_N^x \cap A_N^y }  \bbE\Big[ D^{\ms (\sfs)}_j(x) D^{\ms (\sfs)}_j(y)\Big] \xrightarrow[N \to \infty]{} 0  \ .
	\end{align*}
	It suffices to show one of the these results, since all of them
	can be treated with similar arguments. Let us present for example the proof that
	\begin{align*}
		N^{\frac{d}{2}-1} \sum_{|x-y|\leq 2N^{\frac{\epsilon}{ 2}+\alpha}} \varphi_N(x,y)\, \sum_{ j\, :\, a_j \in A_N^x \cap A_N^y }  \bbE\Big[ D^{\ms (\sfb)}_j(x) D^{\ms (\sfs)}_j(y)\Big] \xrightarrow[N \to \infty]{} 0 \, .
	\end{align*}
	Recall that
	\begin{align*}
		D^{\ms (\sfb)}_j(x)=\bbE\bigg[\bbE_{\tilde \omega}\Big[ \Big ( \log Z_{N,\beta}^{A}(x) -\log Z_{N,\beta}^{A,\sfT_{a_j}}(x)  \Big) \ind_{E_j(x)} \Big]\Big|\mathcal{F}_j\bigg] \, ,
	\end{align*}
	and
	\begin{align*}
		D^{\ms (\sfs)}_j(y)=\bbE\bigg[\bbE_{\tilde \omega}\Big[ \Big ( \log Z_{N,\beta}^{A}(y) -\log Z_{N,\beta}^{A,\sfT_{a_j}}(y)  \Big) \ind_{E^c_j(y)} \Big]\Big|\mathcal{F}_j\bigg] \, .
	\end{align*}
	By Cauchy-Schwarz one has that
	\begin{align*}
		\bbE\Big[ D^{\ms (\sfb)}_j(x) D^{\ms (\sfs)}_j(y)\Big] \leq \bbE\Big[\big(D^{\ms (\sfb)}_j(x)\big)^2\Big]^{\frac{1}{2}}\, \bbE\Big[\big(D^{\ms (\sfs)}_j(y)\big)^2\Big]^{\frac{1}{2}} \, .
	\end{align*}
	Note that,
	\begin{align*}
		\bbE\Big[\big(D^{\ms (\sfb)}_j(x)\big)^2\Big] \leq \bbE\bigg[\bbE_{\tilde \omega}\Big[ \Big ( \log Z_{N,\beta}^{A}(x) -\log Z_{N,\beta}^{A,\sfT_{a_j}}(x)  \Big) \ind_{E_j(x)} \Big]^2 \bigg] \, ,
	\end{align*}
	and similarly
	\begin{align*}
		\bbE\Big[\big(D^{\ms (\sfs)}_j(y)\big)^2\Big] \leq \bbE\bigg[\bbE_{\tilde \omega}\Big[ \Big ( \log Z_{N,\beta}^{A}(y) -\log Z_{N,\beta}^{A,\sfT_{a_j}}(y)  \Big) \ind_{E^c_j(y)} \Big]^2 \bigg] \, ,
	\end{align*}
	since it is true that $\bbE\big[ \bbE[X |\mathcal{G}]^2\big ]\leq   \bbE \big[X^2\big ]$ for a random variable $X:(\Omega,\cF,\bbP)\rightarrow \R$ and a $\sigma$-algebra $\cG \subset \cF$. We note here that we will use the notation $\bbE_{\ms \omega,\tilde{\omega}}[\,\cdot\,]$ to denote the expectation with respect to $\omega$ and $\tilde{\omega}$, i.e. the resampled disorder.
	We use Jensen inequality for the expectation $\bbE_{\tilde{\omega}} [\cdot]$ and bound the indicator $\ind_{E_j(x)}\leq 1$ to obtain that
	\begin{align*}
		\bbE\Big[\big(D^{\ms (\sfb)}_j(x)\big)^2\Big] \leq & \,\bbE\bigg[\bbE_{\tilde \omega}\Big[ \Big ( \log Z_{N,\beta}^{A}(x) -\log Z_{N,\beta}^{A,\sfT_{a_j}}(x)  \Big) \ind_{E_j(x)} \Big]^2 \bigg] \notag \\
		\leq                                               & \,\bbE\bigg[\bbE_{ \tilde \omega}\Big[ \Big ( \log Z_{N,\beta}^{A}(x) -\log Z_{N,\beta}^{A,\sfT_{a_j}}(x)  \Big)^2  \Big] \bigg] \notag             \\
		=                                                  & \bbE_{\ms \omega,\tilde{\omega}} \bigg[ \Big ( \log Z_{N,\beta}^{A}(x) -\log Z_{N,\beta}^{A,\sfT_{a_j}}(x)  \Big)^2   \bigg] \notag                 \\
		\leq
		                                                   & 4 \, \bbE\Big[\big( \log Z_{N,\beta}^{A}(x) \big)^2\Big]<\infty    \, .
	\end{align*}
	by using the inequality $(a+b)^2\leq 2(a^2+b^2)$ and the fact that $\log Z_{N,\beta}^{A}(x)$ and $\log Z_{N,\beta}^{A,\sfT_{a_j}}(x)$ have the same distribution. Also, $\bbE\Big[\big( \log Z_{N,\beta}^{A}(x) \big)^2\Big]<\infty  $
	by Proposition \ref{negmom}.

	For $\bbE\Big[\big(D^{\ms (\sfs)}_j(y)\big)^2\Big]$, we have that
	\begin{align*}
		\quad \qquad  \bbE\Big[\big(D^{\ms (\sfs)}_j(y)\big)^2\Big]\leq & \,\bbE\bigg[\bbE_{ \tilde \omega}\Big[ \Big ( \log Z_{N,\beta}^{A}(y) -\log Z_{N,\beta}^{A,\sfT_{a_j}}(y)  \Big) \ind_{E^c_j(y)} \Big]^2 \bigg] \notag                                                          \\
		\leq                                                            & \,\bbE_{\ms \omega,\tilde{\omega}}\Big[ \Big ( \log Z_{N,\beta}^{A}(y) -\log Z_{N,\beta}^{A,\sfT_{a_j}}(y)  \Big)^2 \ind_{E^c_j(y)} \Big] \notag                                                                \\
		\leq                                                            & \,\bbE_{\ms \omega,\tilde{\omega}} \Big[ \Big ( \log Z_{N,\beta}^{A}(y) -\log Z_{N,\beta}^{A,\sfT_{a_j}}(y)  \Big)^4  \Big]^\frac{1}{2} \bbP_{\ms\omega, \tilde{\omega}} \big(E_j^c(y) \big)^\frac{1}{2} \notag \\
		\leq                                                            & \,4 \bbE \Big[ \big ( \log Z_{N,\beta}^{A}(y) \big)^4  \Big]^\frac{1}{2} \bbP_{\ms \omega,\tilde{\omega}} \big(E_j^c(y) \big)^\frac{1}{2} \, .
	\end{align*}
	Last, by a union bound we have that $ \bbP_{\ms \omega,\tilde{\omega}}\big(E_j^c(y) \big) \leq 2 \bbP(Z^A_{N,\beta}(y)<N^{-\sfh})=2 \bbP(Z^A_{N,\beta}(0)<N^{-\sfh})$. Therefore, there exists a constant $\tilde{{C}}_\beta$, such that for all $j \geq 1$,
	\begin{align*}
		\bbE\Big[ D^{\ms (\sfb)}_j(x) D^{\ms (\sfs)}_j(y)\Big]  \leq \tilde{C}_\beta\, \bbP(Z^A_{N,\beta}(0)<N^{-\sfh})^{\frac{1}{4}} \, .
	\end{align*}
	Hence, we have that
	\begin{align*}
		     & N^{\frac{d}{2}-1} \sum_{|x-y|\leq 2N^{\frac{\epsilon}{ 2}+\alpha}} \varphi_N(x,y)\, \sum_{j\, :\, a_j \in A_N^x \cap A_N^y} \bbE\Big[ D^{\ms (\sfb)}_j(x) D^{\ms (\sfs)}_j(y)\Big]  \notag               \\
		\leq & \, \tilde{C}_{\beta}\,  N^{\frac{d}{2}-1} \sum_{|x-y|\leq 2N^{\frac{\epsilon}{ 2}+\alpha}} \varphi_N(x,y)\, \sum_{j\, :\, a_j \in A_N^x \cap A_N^y}  \bbP(Z^A_{N,\beta}(0)<N^{-\sfh})^{\frac{1}{4}} \, .
	\end{align*}
	{From the definition \eqref{Asets}, we can bound
	$| A_N^x \cap A_N^y|\leq  N^{\epsilon(\frac{d}{2}+1)}\leq N^{(\frac{d}{2}+1)}$. }We also have that the probability $\bbP(Z^A_{N,\beta}(0)<N^{-\sfh})$ decays super-polynomially by Proposition \ref{lefttail} and so does $\bbP(Z^A_{N,\beta}(0)<N^{-\sfh})^{\frac{1}{4}}$. Indeed, by Proposition \ref{lefttail}, we have that
	\begin{align*}
		\bbP\Big(Z^{A}_{N,\beta}(x)<N^{-\sfh} \Big )^{\frac{1}{4}} \leq c_{\beta}^{\frac{1}{4}} \exp\bigg(-\frac{(\sfh \log N
				)^\gamma}{4c_{\beta}} \bigg)\, ,   \quad \gamma>1\, ,
	\end{align*}
	Thus, we have that
	\begin{align*}
		     & N^{\frac{d}{2}-1} \sum_{|x-y|\leq 2N^{\frac{\epsilon}{ 2}+\alpha}} \varphi_N(x,y)\, \sum_{j\, :\, a_j \in A_N^x \cap A_N^y} \bbE\Big[ D^{\ms (\sfb)}_j(x) D^{\ms (\sfs)}_j(y)\Big] \notag \\
		\leq & \, \tilde{C}_{\beta}\,  N^{\frac{d}{2}-1} |A_N^x \cap A_N^y| \sum_{|x-y|\leq 2N^{\frac{\epsilon}{ 2}+\alpha}} \varphi_N(x,y)\,   \bbP(Z^A_{N,\beta}(0)<N^{-\sfh})^{\frac{1}{4}} \notag    \\
		\leq & \, \tilde{C}_\beta \, N^{\frac{d}{2}-1} N^{d +1} \sum_{|x-y|\leq 2N^{\frac{\epsilon}{ 2}+\alpha}} \varphi_N(x,y)\, \bbP(Z^A_{N,\beta}(0)<N^{-\sfh})^{\frac{1}{4}} \notag                  \\
		\leq & \, \tilde{C}_\beta\,  \norm{\varphi}^2_1\,  N^{\frac{d}{2}+d }\, \bbP(Z^A_{N,\beta}(0)<N^{-\sfh})^{\frac{1}{4}} = O(N^{2d})\exp\Big(-O\big(\log N)^\gamma\Big)\, .
	\end{align*}
	Since $\gamma>1$, the last bound vanishes and therefore we conclude that
	\begin{align*}
		N^{\frac{d}{2}-1} \sum_{|x-y|\leq 2N^{\frac{\epsilon}{ 2}+\alpha}} \varphi_N(x,y)\, \sum_{j\, :\, a_j \in A_N^x \cap A_N^y} \bbE\Big[ D^{\ms (\sfb)}_j(x) D^{\ms (\sfs)}_j(y)\Big]  \xrightarrow[N \to \infty]{} 0 \, .
	\end{align*}
\end{proof}

We now proceed to the proof of Proposition \ref{remainder}. We will need the following lemma which provides a bound on the rate of decay of $\bbE\Big[\big(\hat{Z}^{A}_{N,\beta}(x)\big)^2\Big]$.

\begin{lemma} \label{hat}
	For every $\beta \in (0,\beta_{L^2})$, there exists a constant $C_\beta$, such that for every
	$\lambda \in (0,\epsilon)$we have that $\bbE\Big[\big(\hat{Z}^{A}_{N,\beta}(x)\big)^2\Big]\leq C_{\beta} N^{-\lambda (\frac{d}{2}-1)}$.
\end{lemma}

\begin{proof}
	Let us fix a positive $\lambda \in (0,\epsilon).$ We  then have that
	\begin{align*}
		\bbE\Big[\big(\hat{Z}^{A}_{N,\beta}(x)\big)^2\Big]=\sum_{k=1}^N \sigma^{2k} \sumthree{1\leq n_1<...<n_k\leq N}{x:=z_0,z_1,...,z_k \in \Z^d}{\exists \, i \, \in \,  \{1,...,k\}: \: (n_i,z_i)\, \notin A^{x}_N} \prod_{i=1}^k q^2_{n_i-n_{i-1}}(z_i-z_{i-1}) \, .
	\end{align*}
	{Since the rightmost summation is over sequences of $k$ space-time points $ (n_i,z_i)_{1 \leq i \leq k}$, such that at least one of the points $ (n_i,z_i)_{1 \leq i \leq k}$ is not in $A_N^x$, for every such sequence, there exists at least one index $i \in \{1,...,k\}$, such that $|n_i-n_{i-1}|>\frac{1}{k}N^{\epsilon}$ or $|z_i-z_{i-1}|>\frac{1}{k}N^{\frac{\epsilon}{2}+\alpha}$; recall the definition of $A_N^x$ from \eqref{Asets}.} Thus, by changing variables $w_i:=z_i-z_{i-1}$, $\ell_i:=n_i-n_{i-1}$ and extending the range of summation from $1\leq \ell_1 +...+\ell_k \leq N$ to $\ell_1,...,\ell_k \in \{1,...,N\}$, we obtain that

	\begin{align*}
		{\bbE\Big[\big(\hat{Z}^{A}_{N,\beta}(x)\big)^2\Big]\leq \sum_{k=1}^N \sigma^{2k} \sumtwo{\ell_1,...,\ell_k \in \{1,...,N\}}{w_1,...,w_k \in \Z^d} \sum_{j=1}^k \big( \ind_{\{\ell_j >\frac{1}{k}N^{\epsilon}\}}+ \ind_{\{ \ell_j\leq \frac{1}{k}N^{\epsilon}, |w_j|>\frac{1}{k}N^{\frac{\epsilon}{2}+\alpha}   \}} \big) \prod_{i=1}^k q^2_{\ell_i}(w_i) \, .}
	\end{align*}
	By changing the order of summation, for each $i \neq j$ we have that $\sum_{\ell_i=1}^N \sum_{w_i \in \Z^d} q^2_{\ell_i}(w_i) = \sum_{\ell_i=1}^N q_{2\ell_i} (0) =R_N $. Thus, we have
	\begin{align} \label{split}
		{\hspace{-0.6cm} \bbE\Big[\big(\hat{Z}^{A}_{N,\beta}(x)\big)^2\Big]\leq \sum_{k=1}^N \sigma^{2k}  R^{k-1}_N k  \sumtwo{\ell \in \{1,...,N\}}{w \in \Z^d} \big( \ind_{\{\ell >\frac{1}{k}N^{\epsilon}\}}+ \ind_{\{ \ell \leq \frac{1}{k}N^{\epsilon}, |w|>\frac{1}{k}N^{\frac{\epsilon}{2}+\alpha}   \}} \big) q^2_{\ell}(w) \, .}
	\end{align}
	Let us consider the contribution of the two indicator functions separately. For the first one, by summing $w \in \Z^d$, one obtains, for $N$ large enough,
	\begin{align} \label{hatestimate}
		\sum_{k=1}^N \sigma^{2k}   R^{k-1}_N  k \sum_{\frac{1}{k}N^{\epsilon} < \ell \leq N } q_{2\ell}(0) & \leq  \sum_{k=1}^{N^{\epsilon-\lambda}} \sigma^{2k} R^{k-1}_N k \sum_{\frac{1}{k}N^{\epsilon} < \ell \leq N } q_{2\ell}(0)+ \sum_{k>N^{\epsilon-\lambda}}^N \sigma^{2k}  R^{k}_N k \notag \\
		                                                                                                   & \leq \sum_{k=1}^{N^{\epsilon-\lambda}} \sigma^{2k} R^{k-1}_N k (R_N-R_{N^{\lambda}})+\sum_{k>N^{\epsilon-\lambda}}^N \sigma^{2k}  R^{k}_N k                                      \notag   \\
		                                                                                                   & \leq (R_N-R_{N^\lambda}) \sum_{k=1}^{\infty} k a(\beta)^k + \sum_{k>N^{\epsilon-\lambda}}^{\infty} k a(\beta)^k \, ,
	\end{align}
	where $a(\beta):=\sigma^2(\beta)R_{\infty}$ and $R_{\infty}=\sum_{\ell=1}^{\infty}q_{2\ell}(0)$. Note that, since $\beta$ lies in the $L^2$-region, we have that $a(\beta)<1$. Therefore, the sum $\sum_{k=1}^{\infty} k a^k$ is finite.
	Using the local limit theorem one obtains that
	\begin{align*}
		R_N-R_{N^\lambda} \leq C \sum_{\ell>N^{\lambda}}^{\infty}  \frac{1}{\ell^{\frac{d}{2}}} =O\big( N^{-\lambda(\frac{d}{2}-1)} \big) \, ,
	\end{align*}
	for some $C>0$. Moreover, since $a(\beta)<1$, there exists $p>1$ very close to $1$, so that $p  a<1$ and for every $k\geq k_0$, for some $k_0 \in \N$, we have that $ka^k \leq (pa)^k $. Therefore,
	\begin{align*}
		\sum_{k>N^{\epsilon-\lambda}}^{\infty} k a^k  \leq  \sum_{k>N^{\epsilon-\lambda}}^{\infty}  (pa)^k  \leq C  (p a)^{N^{\epsilon-\lambda }} = o\, \big( N^{-\lambda(\frac{d}{2}-1)}\big) \, .
	\end{align*}
	As a consequence, the rate at which the sum \eqref{hatestimate} decays, is  at least $N^{-\lambda(\frac{d}{2}-1)}$, for any $ \lambda \in (0,\epsilon)$.
	For the contribution of the second indicator function in \eqref{split}, namely the sum
	\begin{align*}
		\sum_{k=1}^N \sigma^{2k}  R^{k-1}_N k  \sumtwo{\ell \, \in \, \{1,...,N\}}{w \in \Z^d}  \ind_{\big \{ \ell \leq \frac{1}{k}N^{\epsilon}, |w|>\frac{1}{k}N^{\frac{\epsilon}{2}+\alpha}   \big\}} q^2_{\ell}(w) \, ,
	\end{align*}
	one can see using moderate deviation estimates for the simple random walk and  following the proof of equation (3.4) in \cite{CSZ18b}, that it can be bounded by $\tilde{C}_{\beta}  (\eta_\beta)^{N^\alpha}=o\, \big( N^{-\lambda(\frac{d}{2}-1)}\big)$, for some $\eta_\beta<1$ and some constant $\tilde{C}_\beta>0$
	(we remind that the exponent $\alpha$ is defined in \eqref{Asets}).
	Therefore, one obtains that there exists a constant $C_\beta>0$ such that $\bbE\Big[\big(\hat{Z}^{A}_{N,\beta}(x)\big)^2\Big]\leq C_{\beta} N^{-\lambda (\frac{d}{2}-1)}$.
\end{proof}

\begin{proof}[Proof of Proposition \ref{remainder}]
	\smallskip

	It suffices to prove that:
	\begin{align*}
		N^{\frac{d-2}{4}}\bbE\Big[\big|O_N(x)\big|\Big] \longrightarrow 0 \, ,
	\end{align*}
	as $N \longrightarrow \infty$.

	\justify As in \cite{CSZ18b} {this is a careful Taylor estimate. We define }
	\begin{align*}
		D_N^{\pm}:=\bigg \{\pm \frac{\hat{Z}^{A}_{N,\beta}(x)}{Z^{A}_{N,\beta}(x)}> N^{-p}\bigg \} \hspace{0.5cm} \text{and} \hspace{0.5cm} D_N:=D_N^{+}\cup D_N^{-} =\bigg\{\, \bigg|\frac{\hat{Z}^{A}_{N,\beta}(x)}{Z^{A}_{N,\beta}(x)}\bigg|>N^{-p} \bigg\} \, ,
	\end{align*}
	for $ p=\frac{d-2}{4}p^{*}$, with $0<p^{*}<1$ to be defined later. For  $q=\frac{d-2}{4}q^{*}$  with $0<q^{*}<1$, also to be specified later, we have that
	\begin{align} \label{preDN}
		\bbP(D_N) & \leq  \bbP\Big(D_N\cap \Big \{Z^A_{N,\beta}(x)\geq N^{-q}\big \}\big)+ \bbP\Big(D_N\cap \Big \{Z^A_{N,\beta}(x)< N^{-q}\big \}\big) \notag \\
		          & \leq
		\bbP\Big(\big|\hat{Z}^{A}_{N,\beta}(x)\big|>N^{-(p+q)}\Big)
		+\bbP\Big(Z^A_{N,\beta}(x)<N^{-q}\Big) \notag                                                                                                          \\
		          & \leq  N^{2(p+q)}\bbE\Big[(\hat{Z}^{A}_{N,\beta}(x))^2\Big]+\bbP\Big(Z^A_{N,\beta}(x)<N^{-q}\Big) \, .
	\end{align}
	{For the last inequality we used Chebyshev inequality.}	By Lemma \ref{hat} we have that $\bbE\Big[(\hat{Z}^{A}_{N,\beta}(x))^2\Big] \leq C_\beta N^{-\lambda (\frac{d}{2}-1)} $ for some constant $C_\beta$ and for every $\lambda \in (0,\epsilon)$. {By Proposition \ref{lefttail} we have that $\bbP\Big(Z^A_{N,\beta}(x)<N^{-q}\Big)$ vanishes super-polynomially i.e.}
	\begin{align*}
		\bbP \bigg(Z^A_{N,\beta}(x)<N^{-q}\bigg) \leq c_{\beta} \exp \bigg(\frac{-q^{\gamma} (\log N)^{\gamma}}{c_{\beta}}  \bigg) \,  , \qquad \gamma>1\, .
	\end{align*}
	Therefore, by plugging those estimates into \eqref{preDN} we get that for a constant $\hat{C}_\beta>C_\beta$,
	\begin{align} \label{DN}
		\bbP(D_N) \leq \hat{C}_\beta N^{2(p+q)-\lambda (\frac{d}{2}-1)}\, .
	\end{align}
	For {a} constant $C < \infty$, it is true that,
	\begin{align*}
		|\log (1+y) - y| \le C \cdot \begin{cases}
			\sqrt{\frac{|y|}{1+y}} & \text{if } -1 < y < 0                         \\
			y^2                    & \text{if } -\frac{1}{2} \le y \le \frac{1}{2} \\
			|y|                    & \text{if } 0 < y < \infty
		\end{cases} \,.
	\end{align*}
	Hence,
	\begin{align} \label{ON}
		\bbE\Big[\big|O_N(x)\big|\Big]\leq \bbE\bigg[\bigg(\frac{\hat{Z}^{A}_{N,\beta}(x)}{Z^{A}_{N,\beta}(x)}\bigg)^2\ind_{D_N^{c}}\bigg]+
		\bbE\bigg[\bigg|\frac{\hat{Z}^{A}_{N,\beta}(x)}{Z^{A}_{N,\beta}(x)}\bigg|\ind_{D_N^{+}}\bigg]+
		\bbE\bigg[\sqrt{\bigg|\frac{\hat{Z}^{A}_{N,\beta}(x)}{Z_{N,\beta}(x)}\bigg|}\ind_{D_N^{-}}\bigg] \,.
	\end{align}
	Let us deal with each term separately. We have that
	\begin{align} \label{term1}
		\bbE\bigg[\bigg(\frac{\hat{Z}^{A}_{N,\beta}(x)}{Z^{A}_{N,\beta}(x)}\bigg)^2\ind_{D_N^{c}}\bigg] \leq N^{-2p} \, ,
	\end{align}
	by the definition of $D_N$. We split the second term as follows:

	\begin{align} \label{term2}
		 & \bbE\bigg[\bigg|\frac{\hat{Z}^{A}_{N,\beta}(x)}{Z^{A}_{N,\beta}(x)}\bigg|\ind_{D_N^{+}}\bigg] = \bbE\bigg[\bigg|\frac{\hat{Z}^{A}_{N,\beta}(x)}{Z^{A}_{N,\beta}(x)}\bigg|\ind_{D_N^{+} \cap \{Z^A_{N,\beta}(x)\geq N^{-q} \}}\bigg] + \bbE\bigg[\bigg|\frac{\hat{Z}^{A}_{N,\beta}(x)}{Z^{A}_{N,\beta}(x)}\bigg|\ind_{D_N^{+} \cap \{Z^A_{N,\beta}(x)< N^{-q} \}}\bigg]\,.
	\end{align}
	For the first summand of \eqref{term1} we have that
	\begin{align*}
		\bbE\bigg[\bigg|\frac{\hat{Z}^{A}_{N,\beta}(x)}{Z^{A}_{N,\beta}(x)}\bigg|\ind_{D_N^{+} \cap \{Z^A_{N,\beta}(x)\geq N^{-q} \}}\bigg]
		\leq    & N^q \, \bbE\Big[\big|\hat{Z}^{A}_{N,\beta}(x)\big|\ind_{D_N^{+} \cap \{Z^A_{N,\beta}(x)\geq N^{-q} \}}\Big] \notag \\
		\leq \, & N^q \, \bbE\Big[\big|\hat{Z}^{A}_{N,\beta}(x)\big|\ind_{D_N^{+}}\Big] \notag                                       \\
		\leq\,  & N^q \, \bbE\Big[\big(\hat{Z}^{A}_{N,\beta}(x)\big)^2 \Big]^{\frac{1}{2}} \bbP(D_N)^{\frac{1}{2}} \, ,
	\end{align*}
	by Cauchy-Schwarz. By Lemma \ref{hat}, we get that $\bbE\Big[(\hat{Z}^{A}_{N,\beta}(x))^2\Big] \leq C_\beta N^{-\lambda (\frac{d}{2}-1)} $ and $ \bbP(D_N) \leq \hat{C}_\beta N^{2(p+q)-\lambda (\frac{d}{2}-1)}$ by \eqref{DN}. Hence,
	\begin{align*}
		\bbE\bigg[\bigg|\frac{\hat{Z}^{A}_{N,\beta}(x)}{Z^{A}_{N,\beta}(x)}\bigg|\ind_{D_N^{+} \cap \{Z^A_{N,\beta}(x)\geq N^{-q} \}}\bigg] \leq\, & \,\hat{C}_{\beta} N^q N^{- \lambda (\frac{d-2}{4})} N^{p+q - \lambda (\frac{d-2}{4})} \notag \\
		=\,                                                                                                                                        & \, \hat{C}_\beta  N^{p+2q-2\lambda (\frac{d-2}{4})} \, .
	\end{align*}
	For the second summand of \eqref{term1} we use Hölder inequality with exponents $a=\frac{1}{2}, b=c=\frac{1}{4}$ to obtain that
	\begin{align*}
		\bbE\bigg[\bigg|\frac{\hat{Z}^{A}_{N,\beta}(x)}{Z^{A}_{N,\beta}(x)}\bigg|\ind_{D_N^{+} \cap \{Z^A_{N,\beta}(x)< N^{-q} \}}\bigg] \leq \bbE\Big[ (\hat{Z}^{A}_{N,\beta}(x))^2 \Big ]^{\frac{1}{2}} \bbE\Big[ \frac{1}{(Z^A_{N,\beta}(x))^4}\Big]^{\frac{1}{4}} \bbP(Z^A_{N,\beta}(x)<N^{-q})^{\frac{1}{4}} \, .
	\end{align*}
	The term $\bbP(Z^A_{N,\beta}(x)<N^{-q})^{\frac{1}{4}}$ vanishes super-polynomially therefore, {recalling \eqref{term2}} we conclude that
	\begin{align} \label{secondtermest}
		\bbE\bigg[\bigg|\frac{\hat{Z}^{A}_{N,\beta}(x)}{Z^{A}_{N,\beta}(x)}\bigg|\ind_{D_N^{+}}\bigg]  \leq C_{1,\beta} \,  N^{p+2q-2\lambda (\frac{d-2}{4})} \, .
	\end{align}
	for some constant $C_{1,\beta}>0$. The second summand of \eqref{ON} can be treated similarly. In particular, we split it as follows
	\begin{align} \label{secondsummandsplit}
		\bbE\bigg[{\bigg|\frac{\hat{Z}^{A}_{N,\beta}(x)}{Z_{N,\beta}(x)}\bigg|}^{\frac{1}{2}}\ind_{D_N^{-}}\bigg]   = \bbE\bigg[{\bigg|\frac{\hat{Z}^{A}_{N,\beta}(x)}{Z_{N,\beta}(x)}\bigg|}^{\frac{1}{2}}\ind_{D_N^{-}\cap\{Z_{N,\beta}(x)\geq N^{-q} \} }\bigg]+
		\bbE\bigg[{\bigg|\frac{\hat{Z}^{A}_{N,\beta}(x)}{Z_{N,\beta}(x)}\bigg|}^{\frac{1}{2}}\ind_{D_N^{-}\cap \{Z_{N,\beta}(x)< N^{-q} \}}\bigg] \, .
	\end{align}
	For the first term we have that
	\begin{align} \label{term1secsa}
		\bbE\bigg[{\bigg|\frac{\hat{Z}^{A}_{N,\beta}(x)}{Z_{N,\beta}(x)}\bigg|}^{\frac{1}{2}}\ind_{D_N^{-}\cap\{Z_{N,\beta}(x)\geq N^{-q} \} }\bigg]
		\leq & N^{\frac{q}{2}} \bbE\Big[ |\hat{Z}^{A}_{N,\beta}(x)|^\frac{1}{2} \ind_{D^{-}_N} \Big] \notag             \\
		\leq & N^{\frac{q}{2}} \bbE\Big[ |\hat{Z}^{A}_{N,\beta}(x)|^\frac{1}{2} \ind_{D_N} \Big] \notag                 \\
		\leq & N^{\frac{q}{2}} \bbE\Big[(\hat{Z}^{A}_{N,\beta}(x))^{2} \Big]^{\frac{1}{4}} \bbP(D_N)^{\frac{3}{4}} \, .
	\end{align}
	by H\"{o}lder inequality. By Lemma \ref{hat} we have that $ \bbE\Big[(\hat{Z}^{A}_{N,\beta}(x))^{2} \Big ] \leq C_\beta N^{-\lambda(\frac{d}{2}-1)}$ { for $\lambda \in (0,\epsilon)$} and by bound \eqref{DN} we have that $\bbP(D_N) \leq \hat{C}_\beta N^{2(p+q)-\lambda (\frac{d}{2}-1)}$. Combining {these} two estimates we get that
	\begin{align} \label{term1secsb}
		N^{\frac{q}{2}} \bbE\Big[(\hat{Z}^{A}_{N,\beta}(x))^{2} \Big]^{\frac{1}{4}} \bbP(D_N)^{\frac{3}{4}} \leq & \, \hat{C}_{\beta} N^{\frac{q}{2}} N^{-\frac{\lambda}{2} (\frac{d-2}{4})} N^{\frac{3}{2} (p+q-\lambda(\frac{d-2}{4}))} \notag \\
		=                                                                                                        & \,\hat{C}_{\beta} N^{\frac{3}{2}p+2q-2\lambda(\frac{d-2}{4})} \, ,
	\end{align}
	where we used H\"{o}lder inequality for the last inequality as well {as bound} \eqref{DN} and Lemma \ref{hat}. For the second term in \eqref{secondsummandsplit} we can proceed as before, namely
	\begin{align} \label{term2secs}
		\bbE\bigg[{\bigg|\frac{\hat{Z}^{A}_{N,\beta}(x)}{Z_{N,\beta}(x)}\bigg|}^{\frac{1}{2}}\ind_{D_N^{-}\cap \{Z_{N,\beta}(x)< N^{-q} \}}\bigg] \leq \bbE\Big[ (\hat{Z}^{A}_{N,\beta}(x))^2 \Big ]^{\frac{1}{4}} \bbE\Big[ \frac{1}{(Z^A_{N,\beta}(x))^4}\Big]^{\frac{1}{4}} \bbP(Z_{N,\beta}(x)<N^{-q})^{\frac{1}{2}} \, ,
	\end{align}
	by H\"{o}lder inequality. The super-polynomial decay of $\bbP(Z_{N,\beta}(x)<N^{-q})$ together with the bounds \eqref{DN}, \eqref{term1secsa}, \eqref{term1secsb},\eqref{term2secs} and Proposition \ref{lefttail}, allows us to conclude that
	\begin{align} \label{thirdtermest}
		\bbE\bigg[{\bigg|\frac{\hat{Z}^{A}_{N,\beta}(x)}{Z_{N,\beta}(x)}\bigg|}^{\frac{1}{2}}\ind_{D_N^{-}}\bigg] \leq C_{2,\beta} N^{\frac{3}{2}p+2q-2\lambda(\frac{d-2}{4})} \, ,
	\end{align}
	for some constant $ C_{2,\beta}>0$.
	Recall now that we wanted to prove that $N^{\frac{d-2}{4}}\bbE \big[ |O_N(x)| \big] \to 0$ as $N \to  \infty$.   By the estimates \eqref{term1}, \eqref{secondtermest} and \eqref{thirdtermest} {respectively}, we see that it suffices to find exponents $p^{*},q^{*}$ and $\lambda$, so that
	\begin{align*}
		1-2p^{*}<0 ,             \qquad 1-2\lambda+p^{*}+2q^{*}<0, \qquad      1-2\lambda+\frac{3}{2}p^{*}+2q^{*}<0 \, .
	\end{align*}
	{Since we can consider $\lambda \in (0,\epsilon)$ arbitrarily close to $\epsilon$ and also because the second inequality is implied by the third, it suffices to find exponents $p^{*},q^{*}$ and $\epsilon$, so that}
	\begin{align*}
		{1-2p^{*}<0 ,     \qquad    1-2\epsilon+\frac{3}{2}p^{*}+2q^{*}<0 \, .}
	\end{align*}
	{This would lead to $\epsilon>\frac{1}{2}(1+\frac{3}{2}p^{*}+2q^*)$ and since we can take $p^*>\frac{1}{2}$ arbitrarily close to $\frac{1}{2}$ and $q^* >0$ arbitrarily small, it suffices to choose $\epsilon>\frac{7}{8}$} {in the definition of the sets $A_N^x$, recall \eqref{Asets}}.
\end{proof}
We proceed now to the proof of Proposition \ref{factorisation}.
\begin{proof}[Proof of Proposition \ref{factorisation}]
	We need to prove that
	\begin{align} \label{quotfluct}
		N^{\frac{d-2}{4}}  \sum_{x \in \Z^d} \varphi_N(x)\, \bigg(\frac{\hat{Z}^{A}_{N,\beta}(x)}{Z^{A}_{N,\beta}(x)}- \big(Z_{N,\beta}^{B^{\geq}}(x)-1\big) \bigg) \xrightarrow[N \to \infty]{L^1(\bbP)} 0 \, .
	\end{align}
	We remind the reader that $ B^{\geq}_N:=\big( (N^\rho,N] \cap \N \big) \times \Z^d$ for some $\rho \in (\epsilon,1)$, the choice of which is specified by \eqref{finalrho}. We also define the sets
	\begin{align*}
		 & B_N:=\big( (N^\epsilon,N] \cap \N \big) \times \Z^d \, ,                                                                \notag   \\
		 & C^{x}_N:=\big\{ (n,z) \in \N\times \Z^d: \: 1\leq n\leq N^{\epsilon}, \: \: |z-x|\geq N^{\frac{\epsilon}{2}+\alpha} \big \} \, .
	\end{align*}
	We decompose $\hat{Z}^{A}_{N,\beta}(x)$ into two parts
	\begin{align*}
		\hat{Z}^{A}_{N,\beta}(x)=Z^{A,B}_{N,\beta}(x)+Z^{A,C}_{N,\beta}(x) \, ,
	\end{align*}
	where
	\begin{align} \label{ZAC}
		 & Z^{A,B}_{N,\beta}(x):=\sum_{\tau \subset A^x_N \cup B_N: \: \tau \cap B_N \neq \eset} \sigma^{|\tau|}{q}^{(0,x)}(\tau) {\eta}(\tau)     \, ,         \notag \\
		 & Z^{A,C}_{N,\beta}(x):=\sum_{\tau \subset\{1,\ldots, N\} \times \Z^d: \: \tau \cap C_N^x \neq \eset} \sigma^{|\tau|}{q}^{(0,x)}(\tau) {\eta}(\tau) \, .
	\end{align}
	and if $\tau=(n_i,z_i)_{1\leq i \leq k}$,
	\begin{align*}
		q^{(0,x)}(\tau) :=q_{n_1}(z_1-x)\prod_{i=2}^k q_{n_i-n_{i-1}}(z_i-z_{i-1})\, .
	\end{align*}
	The proof will consist of three steps.\\

	\textbf{(Step 1)} The first task will be to show that $Z^{A,C}_{N,\beta}(x)$ has a negligible contribution to~\eqref{quotfluct}. The proof of this is based on the fact that $Z^{A,C}_{N,\beta}(x)$ consists of random walk paths which are super-diffusive: the walk will have to travel at distance greater than $N^{\frac{\epsilon}{2}+\alpha}$ from $x$ within time $N^{\epsilon}$. Therefore, by standard moderate deviation estimates one can show that
	\begin{align*}
		N^{\frac{d-2}{4}}  \sum_{x \in \Z^d} \varphi_N(x)\, \cdot \frac{Z^{A,C}_{N,\beta}(x)}{Z^{A}_{N,\beta}(x)} \xrightarrow[N \to \infty]{L^2(\bbP)} 0\, ,
	\end{align*}
	super-polynomially.The proof follows the same lines of the proof of  Prop. 2.3. in \cite{CSZ18b} and for this reason we omit the details.

	\textbf{ (Step 2)}
	The second step will be to show that in the chaos expansion of $Z_{N,\beta}^{A,B}(x)$,
	the contribution {from sampling disorder $\eta_{r,z}$, with  $r<N^{\rho}$ }is negligible, for every $\rho \in (\epsilon,1)$.
		{In} particular, let us denote by $\Strip$ the set $\Strip:= \Big \{(n,z) \in (N^{\epsilon},N^{\rho}]\times \Z^d \Big \}$. We can decompose $Z_{N,\beta}^{A,B}(x)$ into two parts $Z_{N,\beta}^{A,B}(x)=Z_{N,\beta}^{A,B^{<}}(x)+Z_{N,\beta}^{A,B^{\geq}}(x)$
	such that
	\begin{align}  \label{Zstrip}
		Z_{N,\beta}^{A,B^<}(x) := \sum_{k=1}^N \sigma^k \sumthree{0:=n_0< n_1<...<n_k \leq N}{x:=z_0, z_1,...,z_k \in \Z^d}{(n_i,z_i)_{1 \leq i \leq k}\cap \Strip \neq \eset}  \prod_{i=1}^k q_{n_i-n_{i-1}}(z_i-z_{i-1})   \eta_{n_i,z_i} \, .
	\end{align}
	and
	\begin{align}  \label{Znostrip}
		Z_{N,\beta}^{A,B^\geq}(x) := \sum_{k=1}^N \sigma^k \sumthree{0:=n_0< n_1<...<n_k \leq N}{x:=z_0, z_1,...,z_k \in \Z^d}{(n_i,z_i)_{1 \leq i \leq k}\cap \Strip = \eset}  \prod_{i=1}^k q_{n_i-n_{i-1}}(z_i-z_{i-1})   \eta_{n_i,z_i} \, .
	\end{align}
	In this step we will show that
	\begin{align*}
		N^{\frac{d-2}{4}}  \sum_{x \in \Z^d} \varphi_N(x)\, \cdot \frac{Z^{A,B^{<}}_{N,\beta}(x)}{Z^{A}_{N,\beta}(x)} \xrightarrow[N \to \infty]{L^2(\bbP)} 0\, ,
	\end{align*}
	or equivalently
	\begin{align} \label{L2strip}
		N^{\frac{d}{2}-1} \sum_{x,y \in \Z^d} \varphi_N(x,y) \: \bbE \Bigg[ \frac{Z_{N,\beta}^{A,B^<}(x)}{Z_{N,\beta}^A(x)} \cdot \frac{Z_{N,\beta}^{A,B^<}(y)}{Z_{N,\beta}^A(y)} \Bigg] \,\xrightarrow[N\to\infty]\, 0 \, .
	\end{align}
	Let us denote by $S^x,S^y$ the paths of two independent random walks starting from $x,y$ respectively. Let us also use the following notation
	\begin{align} \label{FNdef}
		F_N(x,y):=\E_{x,y}\big[(e^{\sfH^x(\omega)}-1)(e^{\sfH^y(\omega)}-1) \ind_{{S^x \cap S^y} \cap \Strip \neq \eset}\big]\, ,
	\end{align}
	{where
	\begin{align*}
		\sfH^x(\omega):=\sum_{(n,z) \in \N\times \Z^d} (\beta \omega_{n,z}-\lambda(\beta))\ind_{S^x_n=z} \, ,
	\end{align*}}
	\noindent and
	\begin{align}\label{FNnsdef}
		F^{\sfn \sfs}_N(x,y):=\E_{x,y}\big[(e^{\sfH_{\sfn \sfs}^x(\omega)}-1)(e^{\sfH_{\sfn \sfs}^y(\omega)}-1) \ind_{S^x \cap S^y \cap \Strip \neq \eset}\big] \, ,
	\end{align}
	where
	\begin{align*}
		\sfH_{\sfn \sfs}^x(\omega):=\sum_{(n,z) \in (\Strip)^c} (\beta \omega_{n,z}-\lambda(\beta))\ind_{S^x_n=z}\,.
	\end{align*} is the energy which does not contain disorder indexed by space-time points in the region $\Strip$. Note that, even though in the definition \eqref{FNdef} of $F^{\sfn \sfs}_N(x,y)$, the energies $\sfH_{\sfn \sfs}^x(\omega),\sfH_{\sfn \sfs}^y(\omega)$ do not contain disorder indexed by $\Strip$, there is still the constraint that the two random walks $S^x,S^y$ meet at some point in $\Strip.$

	We will control \eqref{L2strip}, by showing that
	\begin{align} \label{hatequalsdif}
		\bbE \Bigg[ \frac{Z_{N,\beta}^{A,B^<}(x)}{Z_{N,\beta}^A(x)} \cdot \frac{Z_{N,\beta}^{A,B^<}(y)}{Z_{N,\beta}^A(y)} \Bigg] = \bbE\Bigg[\frac{F_N(x,y)-F^{\sfn \sfs}_N(x,y)}{Z_{N,\beta}^{A}(x)Z_{N,\beta}^A(y)} \Bigg] \, ,
	\end{align}
	{and then showing that when the right-hand side is inserted into \eqref{L2strip}, then it leads to vanishing contribution.} {Let us check first the equality \eqref{hatequalsdif}.}
	The chaos expansion of $F_N(x,y)$ is
	\begin{align*}
		F_N(x,y)= & \E_{x,y}\big[(e^{\sfH^x(\omega)}-1)(e^{\sfH^y(\omega)}-1) \ind_{S^x \cap S^y \cap \Strip \neq \eset}\big] \notag                                                                                                                                                               \\
		=         & \sum_{1 \leq k,\ell \leq N} \sigma^{k+\ell} \sumtwo{(n_i,z_i)_{1\leq i \leq k}}{(m_j,{w_j})_{1\leq j \leq \ell} } \E_{x,y}\big[  \prodtwo{1 \leq i \leq k}{1\leq j \leq \ell} \ind_{S^{x}_{n_i}=z_i} \ind_{S^{y}_{m_j}=w_j}   \ind_{S^x \cap S^y \cap \Strip \neq \eset} \big] \\ \times & \prodtwo{1 \leq i \leq k}{1\leq j \leq \ell}\eta_{n_i,z_i} \eta_{m_j,w_j} \, .
	\end{align*}
	Similarly,
	\begin{align*}
		F^{\sfn \sfs}_N(x,y)= & \E_{x,y}\big[(e^{\sfH_{\sfn \sfs}^x(\omega)}-1)(e^{\sfH_{\sfn \sfs}^y(\omega)}-1) \ind_{S^x \cap S^y \cap \Strip \neq \eset}\big] \notag                                                                                                                                                                                   \\
		=                     & \sum_{1 \leq k,\ell \leq N} \sigma^{k+\ell} \sumtwo{(n_i,z_i)_{1\leq i \leq k}\cap \Strip =\eset}{(m_j,{w_j})_{1\leq j \leq \ell} \cap \Strip =\eset } \E_{x,y}\big[  \prodtwo{1 \leq i \leq k}{1\leq j \leq \ell} \ind_{S^{x}_{n_i}=z_i} \ind_{S^{y}_{m_j}=w_j}   \ind_{S^x \cap S^y \cap \Strip \neq \eset} \big] \notag \\
		\times                & \prodtwo{1 \leq i \leq k}{1\leq j \leq \ell}\eta_{n_i,z_i} \eta_{m_j,w_j} \, .
	\end{align*}
	The constraints $(n_i,z_i)_{1\leq i \leq k}\cap \Strip =\eset$ and $(m_j,w_j)_{1\leq j \leq \ell} \cap \Strip =\eset$ come from the fact that the energies $H^x_{\sfn \sfs}(\omega),H^y_{\sfn \sfs}(\omega)$ do not sample points from $\Strip$.
	The chaos expansion of the difference, $F_N(x,y)-F^{\sfn \sfs}_N(x,y)$, is then
	\begin{align*}
		 & F_N(x,y)-F^{\sfn \sfs}_N(x,y)=  \notag                                                                                                                                                                                                                                                                                                                                                                                                   \\
		 & \sum_{1 \leq k,\ell \leq N} \sigma^{k+\ell} \sumthree{(n_i,z_i)_{1\leq i \leq k}\cap \Strip \neq \eset \vspace{0.1cm}}{\textbf{or}}{(m_j,{w_j})_{1\leq j \leq \ell} \cap \Strip \neq \eset } \E_{x,y}\big[  \prodtwo{1 \leq i \leq k}{1\leq j \leq \ell} \ind_{S^{x}_{n_i}=z_i} \ind_{S^{y}_{m_j}=w_j}   \ind_{S^x \cap S^y \cap \Strip \neq \eset} \big] \prodtwo{1 \leq i \leq k}{1\leq j \leq \ell}\eta_{n_i,z_i} \eta_{m_j,w_j} \, .
	\end{align*}
	Therefore, the expansion of $\displaystyle \bbE\Bigg[\frac{F_N(x,y)-F^{\sfn \sfs}_N(x,y)}{Z_{N,\beta}^{A}(x)Z_{N,\beta}^A(y)} \Bigg]$ is
	\begin{align} \label{stripest3}
		\bbE\Bigg[\frac{F_N(x,y)-F^{\sfn \sfs}_N(x,y)}{Z_{N,\beta}^{A}(x)Z_{N,\beta}^A(y)} \Bigg] = & \bbE \Bigg [ \frac{1}{Z_{N,\beta}^{A}(x)Z_{N,\beta}^A(y)} \sum_{1 \leq k,\ell \leq N} \sigma^{k+\ell} \sumthree{(n_i,z_i)_{1\leq i \leq k}\cap \Strip \neq \eset \vspace{0.1cm}}{\textbf{or}}{(m_j,{w_j})_{1\leq j \leq \ell} \cap \Strip \neq \eset }  \notag \\
		\times                                                                                      & \E_{x,y}\big[  \prodtwo{1 \leq i \leq k}{1\leq j \leq \ell} \ind_{S^{x}_{n_i}=z_i} \ind_{S^{y}_{m_j}=w_j}   \ind_{S^x \cap S^y \cap \Strip \neq \eset} \big]  \prodtwo{1 \leq i \leq k}{1\leq j \leq \ell}\eta_{n_i,z_i} \eta_{m_j,w_j} \Bigg ] \, .
	\end{align}
	Note that if for example $(n_i,z_i)_{1\leq i \leq k}\cap \Strip \neq \eset$, the expectation $\bbE[\cdot]$ will impose that also $(m_j,z_j)_{1\leq j \leq \ell} \cap \Strip \neq \eset$ and in particular, $(n_i,z_i)_{1\leq i \leq k}\cap \Strip= (m_j,{w_j})_{1\leq j \leq \ell} \cap \Strip$, due to the fact that the $\eta$ variables indexed by space-time points with time index $t>N^{\epsilon}$ appearing in the expansion of $F_N(x,y)-F^{\sfn \sfs}_N(x,y)$ have to match pairwise, because they are independent of $Z_{N,\beta}^A(x),Z_{N,\beta}^A(y)$, and so if a disorder variable $\eta_{n_i,z_i}$ or $\eta_{m_j,w_j}$ is unmatched, their mean zero property will lead to vanishing of the whole expectation $\bbE[\cdot]$. Thus, the indicator $\ind_{S^x \cap S^y \cap \Strip \neq \eset}$ will always be equal to $1$ for every summand of the last expansion, since we are summing space-time sequences, such that $(n_i,z_i)_{1\leq i \leq k}\cap(m_j,z_j)_{1\leq j \leq \ell}\cap \Strip \neq \eset$. Therefore, the expansion of $\displaystyle \bbE\Bigg[\frac{F_N(x,y)-F^{\sfn \sfs}_N(x,y)}{Z_{N,\beta}^{A}(x)Z_{N,\beta}^A(y)} \Bigg]$ is actually equal to
	\begin{align*}
		\bbE\Bigg[\frac{F_N(x,y)-F^{\sfn \sfs}_N(x,y)}{Z_{N,\beta}^{A}(x)Z_{N,\beta}^A(y)} \Bigg] = & \bbE \Bigg [ \frac{1}{Z_{N,\beta}^{A}(x)Z_{N,\beta}^A(y)} \sum_{1 \leq k,\ell \leq N} \sigma^{k+\ell} \sumtwo{(n_i,z_i)_{1\leq i \leq k}\cap \Strip \neq \eset }{(m_j,{w_j})_{1\leq j \leq \ell} \cap \Strip \neq \eset }  \notag \\
		\times                                                                                      & \E_{x,y}\big[  \prodtwo{1 \leq i \leq k}{1\leq j \leq \ell} \ind_{S^{x}_{n_i}=z_i} \ind_{S^{y}_{m_j}=w_j}  \big]  \prodtwo{1 \leq i \leq k}{1\leq j \leq \ell}\eta_{n_i,z_i} \eta_{m_j,w_j} \Bigg ] \, .
	\end{align*}
	Recalling \eqref{Zstrip}, we have that
	\begin{align*}
		\bbE \Bigg[ \frac{Z_{N,\beta}^{A,B^<}(x)}{Z_{N,\beta}^A(x)} \cdot \frac{Z_{N,\beta}^{A,B^<}(y)}{Z_{N,\beta}^A(y)} \Bigg] = & \bbE \Bigg [ \frac{1}{Z_{N,\beta}^{A}(x)Z_{N,\beta}^A(y)} \sum_{1 \leq k,\ell \leq N} \sigma^{k+\ell} \sumtwo{(n_i,z_i)_{1\leq i \leq k}\cap \Strip \neq \eset }{(m_j,{w_j})_{1\leq j \leq \ell} \cap \Strip \neq \eset }  \notag \\
		\times                                                                                                                     & \E_{x,y}\big[  \prodtwo{1 \leq i \leq k}{1\leq j \leq \ell} \ind_{S^{x}_{n_i}=z_i} \ind_{S^{y}_{m_j}=w_j}  \big]  \prodtwo{1 \leq i \leq k}{1\leq j \leq \ell}\eta_{n_i,z_i} \eta_{m_j,w_j} \Bigg ]\, .
	\end{align*}
	Therefore, we conclude that
	\begin{align*}
		\bbE \Bigg[ \frac{Z_{N,\beta}^{A,B^<}(x)}{Z_{N,\beta}^A(x)} \cdot \frac{Z_{N,\beta}^{A,B^<}(y)}{Z_{N,\beta}^A(y)} \Bigg]=\bbE\Bigg[\frac{F_N(x,y)-F^{\sfn \sfs}_N(x,y)}{Z_{N,\beta}^{A}(x)Z_{N,\beta}^A(y)} \Bigg]\, .
	\end{align*}
	\smallskip

	Having established this equality,
	to finish the proof of  \eqref{L2strip}, we will prove that
	\begin{align} \label{stripest4}
		N^{\frac{d}{2}-1} \sum_{x,y \in \Z^d} \varphi_N(x,y) \: \bbE \Bigg[ \frac{F_N(x,y)}{Z_{N,\beta}^A(x)Z_{N,\beta}^A(y)}  \Bigg]   \,\xrightarrow[N\to\infty]\, 0 \, ,
	\end{align} and
	\begin{align} \label{stripest5}
		N^{\frac{d}{2}-1} \sum_{x,y \in \Z^d} \varphi_N(x,y) \: \bbE \Bigg[ \frac{F^{\sfn \sfs}_N(x,y)}{Z_{N,\beta}^A(x)Z_{N,\beta}^A(y)}  \Bigg]\,\xrightarrow[N\to\infty]\, 0 \,.
	\end{align}
	We start by showing the validity of \eqref{stripest4}, since \eqref{stripest5} can be treated with the same arguments. In view of \eqref{FNdef} we have that
	\begin{align} \label{FNdecomp}
		F_N(x,y)= & \E_{x,y}\big[(e^{\sfH^x(\omega)}-1)(e^{\sfH^y(\omega)}-1) \ind_{S^x \cap S^y \cap \Strip \neq \eset}\big] \notag                           \\
		=         & \E_{x,y}\big[e^{\sfH^x(\omega)+\sfH^y(\omega)} \ind_{S^x \cap S^y \cap \Strip \neq \eset}\big]
		-\E_{x,y}\big[e^{\sfH^x(\omega)} \ind_{S^x \cap S^y \cap \Strip \neq \eset}\big]\notag                                                                 \\
		-         & \E_{x,y}\big[e^{\sfH^y(\omega)} \ind_{S^x \cap S^y \cap \Strip \neq \eset}\big]+\P_{x,y}\big(S^x \cap S^y \cap \Strip \neq \eset \big)\, .
	\end{align}
	We begin by showing that
	\begin{align*}
		N^{\frac{d}{2}-1} \sum_{x,y \in \Z^d} \varphi_N(x,y) \: \bbE \Bigg[ \frac{\E_{x,y}\big[e^{\sfH^x(\omega)+\sfH^y(\omega)} \ind_{S^x \cap S^y \cap \Strip \neq \eset}\big]}{Z_{N,\beta}^A(x)Z_{N,\beta}^A(y)}  \Bigg]   \,\xrightarrow[N\to\infty]\, 0 \,.
	\end{align*}
	The main point here will be to remove the denominators. Consider the set $E_N:=\{ Z^A_{N,\beta}(x),  Z^A_{N,\beta}(y) \geq{ N^{-\sfh} }\} $ for {some $\sfh \in (0,\frac{1-\rho}{2})$}. We have that
	\begin{align} \label{ENdecomposition}
		\bbE\bigg[ \frac{\E_{x,y}\big[e^{\sfH^x(\omega)+\sfH^y(\omega)} \ind_{S^x \cap S^y \cap \Strip \neq \eset}\big]}{Z_{N,\beta}^A(x)Z_{N,\beta}^A(y)} \bigg]
		= & \bbE\bigg[ \frac{\E_{x,y}\big[e^{\sfH^x(\omega)+\sfH^y(\omega)} \ind_{S^x \cap S^y \cap \Strip \neq \eset}\big]}{Z_{N,\beta}^A(x)Z_{N,\beta}^A(y)}  \ind_{E_N} \bigg] \notag   \\
		+ & \bbE\bigg[ \frac{\E_{x,y}\big[e^{\sfH^x(\omega)+\sfH^y(\omega)} \ind_{S^x \cap S^y \cap \Strip \neq \eset}\big]}{Z_{N,\beta}^A(x)Z_{N,\beta}^A(y)}  \ind_{E^{c}_N} \bigg] \, .
	\end{align}
	We can bound the first summand using the definition of the sets $E_N$, as follows
	\begin{align} \label{stripest1}
		\bbE\bigg[ \frac{\E_{x,y}\big[e^{\sfH^x(\omega)+\sfH^y(\omega)} \ind_{S^x \cap S^y \cap \Strip \neq \eset}\big]}{Z_{N,\beta}^A(x)Z_{N,\beta}^A(y)}  \ind_{E_N} \bigg]
		\leq & \, N^{2\sfh} \, \bbE\bigg[  \E_{x,y}\big[e^{\sfH^x(\omega)+\sfH^y(\omega)} \ind_{S^x \cap S^y \cap \Strip \neq \eset}\big] \bigg]\notag \\
		=    & \, N^{2\sfh}\, \E_{x,y}\big[e^{\lambda_2(\beta)\cL_N(x,y)}\ind_{S^x \cap S^y \cap \Strip \neq \eset}\big] \, .
	\end{align}
	We condition on the first time, $\tau_{x,y}$, that the two random walk paths meet,  to obtain that
	\begin{align*}
		\E_{x,y}\big[e^{{\lambda_2(\beta)} \cL_N(x,y)}\ind_{S^x \cap S^y \cap \Strip \neq \eset}\big]= & \sum_{n=1}^{N^\rho} \E_{x,y}\big[e^{\lambda_2(\beta)\cL_N(x,y)} \ind_{S^x \cap S^y \cap \Strip \neq \eset} \big| \tau_{x,y}=n  \big] \P(\tau_{x,y}=n) \notag \\
		\leq                                                                                           & \sum_{n=1}^{N^\rho} \E_{x,y}\big[e^{\lambda_2(\beta)\cL_N(x,y)}  \big| \tau_{x,y}=n  \big] \P(\tau_{x,y}=n) \, .
	\end{align*}
	By the Markov property
	\begin{align} \label{stripest2}
		\sum_{n=1}^{N^\rho} \E_{x,y}\big[e^{\lambda_2(\beta)\cL_N(x,y)}  \big| \tau_{x,y}=n  \big] \P_{x,y}(\tau_{x,y}=n)= & \sum_{n=1}^{N^\rho} \E\big[e^{\lambda_2(\beta)(\cL_{N-n}+1)}  \big] \P_{x,y}(\tau_{x,y}=n) \notag                    \\
		=                                                                                                                  & \sum_{n=1}^{N^\rho} e^{\lambda_2(\beta)}\, \E\big[e^{\lambda_2(\beta)\cL_{N-n}}  \big] \P_{x,y}(\tau_{x,y}=n) \notag \\
		\leq                                                                                                               & \, e^{\lambda_2(\beta)} \,\E\big[e^{\lambda_2(\beta)\cL_{\infty}}  \big] \sum_{n=1}^{N^\rho}  q_{2n}(x-y) \, .
	\end{align}
	We set $\tilde{C}_{\beta}:=e^{\lambda_2(\beta)}\, \E\big[e^{\lambda_2(\beta)\cL_{\infty}}  \big]$
	and remind the reader that $\E\big[e^{\lambda_2(\beta)\cL_{\infty}}\big]<\infty$ because $\beta \in (0,\beta_{L^2})$. Therefore, if we combine \eqref{stripest1},\eqref{stripest2}, we deduce the estimate
	\begin{align*}
		     & N^{\frac{d}{2}-1} \sum_{x,y \in \Z^d}  \varphi_N(x,y) \, \bbE\bigg[ \frac{\E_{x,y}\big[e^{\sfH^x(\omega)+\sfH^y(\omega)} \ind_{S^x \cap S^y \cap \Strip \neq \eset}\big]}{Z_{N,\beta}^A(x)Z_{N,\beta}^A(y)}  \ind_{E_N} \bigg] \notag \\
		\leq & \, \tilde{C}_{\beta} \, N^{\frac{d}{2}-1+2\sfh} \sum_{x,y \in \Z^d}  \varphi_N(x,y) \sum_{n=1}^{N^\rho} q_{2n}(y-x) \, .
	\end{align*}
	The last bound vanishes because $\sfh \in (0,\frac{1-\rho}{2})$, see \eqref{erbound} for the derivation of this fact.

	We now deal with the complementary event $E^c_N$ in \eqref{ENdecomposition}. Recall that
	\begin{align*}
		E^c_N=\{ Z^A_{N,\beta}(x) <N^{-\sfh} \} \cup  \{ Z^A_{N,\beta}(y) <N^{-\sfh} \} \, .
	\end{align*}
	By Proposition \ref{lefttail} and a union bound we obtain that
	\begin{align} \label{superpolynomial}
		{\bbP(E^c_N) \leq 2 \bbP \big(Z^A_{N,\beta}(x)<N^{-\sfh}\big) \leq 2c_{\beta} \exp \bigg(\frac{-\sfh^{\gamma} (\log N)^{\gamma}}{c_{\beta}}  \bigg)} \, .
	\end{align}
	Recall that we need to show that
	\begin{align*}
		N^{\frac{d}{2}-1} \sum_{x,y \in \Z^d} \varphi_N(x,y) \: \bbE \Bigg[ \frac{\E_{x,y}\big[e^{\sfH^x(\omega)+\sfH^y(\omega)} \ind_{S^x \cap S^y \cap \Strip \neq \eset}\big]}{Z_{N,\beta}^A(x)Z_{N,\beta}^A(y)} \ind_{E^c_N} \Bigg]   \,\xrightarrow[N\to\infty]\, 0 \,.
	\end{align*}
	We have that
	\begin{align*}
		\bbE \Bigg[ \frac{\E_{x,y}\big[e^{\sfH^x(\omega)+\sfH^y(\omega)} \ind_{S^x \cap S^y \cap \Strip \neq \eset}\big]}{Z_{N,\beta}^A(x)Z_{N,\beta}^A(y)} \ind_{E^c_N} \Bigg]
		\leq & \, \bbE \Bigg[ \frac{\E_{x,y}\big[e^{\sfH^x(\omega)+\sfH^y(\omega)} \big]}{Z_{N,\beta}^A(x)Z_{N,\beta}^A(y)} \ind_{E^c_N} \Bigg] \notag \\
		=    & \, \bbE \Bigg[ \frac{Z_{N,\beta}(x)}{Z_{N,\beta}^A(x)}\cdot \frac{Z_{N,\beta}(y)}{Z_{N,\beta}^A(y)}\ind_{E^c_N} \Bigg] \, .
	\end{align*}
	In order to bound the last expectation, we use H\"{o}lder inequality with exponents $p,p,q>1$, so that
	$\frac{2}{p}+\frac{1}{q}=1$, with
	$p\in (2,\infty)$ sufficiently close to $2$ so that $\bbE\big[(Z_{N,\beta}(x))^p\big]<\infty$, {thanks to} Proposition \ref{hyperc}. In particular, we obtain that
	\begin{align*}
		\bbE \Bigg[ \frac{Z_{N,\beta}(x)}{Z_{N,\beta}^A(x)}\cdot \frac{Z_{N,\beta}(y)}{Z_{N,\beta}^A(y)}\ind_{E^c_N} \Bigg] \leq \bbE \Bigg[ \bigg(\frac{Z_{N,\beta}(0)}{Z^A_{N,\beta}(0)}\bigg)^p \Bigg]^{\frac{2}{p}} \bbP(E_N^c)^{\frac{1}{q}} \, .
	\end{align*}
	We apply H\"{o}lder inequality again {on the first term}, with exponents $r,s>1$, so that $\frac{1}{r}+\frac{1}{s}=1$
	and $r>1$ is sufficiently close to $1$ so that we have $\bbE\bigg[\Big(Z_{N,\beta}(0)\Big)^{pr}\bigg]<\infty$, by Proposition \ref{hyperc}. This way, we obtain that
	\begin{align*}
		\bbE \Bigg[ \bigg(\frac{Z_{N,\beta}(0)}{Z^A_{N,\beta}(0)}\bigg)^p \Bigg]^{\frac{2}{p}} \leq \bbE\bigg[\Big(Z_{N,\beta}(0)\Big)^{pr}\bigg]^{\frac{2}{p r}}  \bbE\bigg[\Big(Z^A_{N,\beta}(0)\Big)^{-ps}\bigg]^{\frac{2}{ps}} \, .
	\end{align*}
	By Proposition \ref{negmom}, we also have that $\displaystyle \bbE\bigg[\Big(Z^A_{N,\beta}(0)\Big)^{-ps}\bigg]< \infty$. Therefore,
	we have showed that there exists a constant $\hat{C}_\beta$, such that
	\begin{align*}
		\bbE \Bigg[ \frac{\E_{x,y}\big[e^{\sfH^x(\omega)+\sfH^y(\omega)} \ind_{S^x \cap S^y \cap \Strip \neq \eset}\big]}{Z_{N,\beta}^A(x)Z_{N,\beta}^A(y)} \ind_{E^c_N} \Bigg] \leq \hat{C}_\beta P(E_N^c)^{\frac{1}{q}} \, .
	\end{align*}
	for some $q>1$. Thus,
	\begin{align*}
		     & N^{\frac{d}{2}-1} \sum_{x,y \in \Z^d} \varphi_N(x,y) \: \bbE \Bigg[ \frac{\E_{x,y}\big[e^{\sfH^x(\omega)+\sfH^y(\omega)} \ind_{S^x \cap S^y \cap \Strip \neq \eset}\big]}{Z_{N,\beta}^A(x)Z_{N,\beta}^A(y)} \ind_{E^c_N} \Bigg] \notag \\
		\leq & \hat{C}_\beta c_\beta   N^{\frac{d}{2}-1} \sum_{x,y \in \Z^d} \varphi_N(x,y) \: \exp \bigg(\frac{- \sfh^{\gamma} {(\log N)}^{\gamma}}{q \,c_{\beta}}  \bigg) \xrightarrow[N \to \infty]{}0 \, ,
	\end{align*}
	because $\gamma>1$. {Recall now decomposition} \eqref{FNdecomp}. We have shown that
	\begin{align} \label{firstpartFN}
		N^{\frac{d}{2}-1} \sum_{x,y \in \Z^d} \varphi_N(x,y) \: \bbE \Bigg[ \frac{\E_{x,y}\big[e^{\sfH^x(\omega)+\sfH^y(\omega)} \ind_{S^x \cap S^y \cap \Strip \neq \eset}\big]}{Z_{N,\beta}^A(x)Z_{N,\beta}^A(y)}  \Bigg]   \,\xrightarrow[N\to\infty]\, 0 \,.
	\end{align}
	{Similarly, we can show that}
	\begin{align} \label{3estimates}
		 & N^{\frac{d}{2}-1} \sum_{x,y \in \Z^d} \varphi_N(x,y) \: \bbE \Bigg[ \frac{\E_{x,y}\big[e^{\sfH^x(\omega)}\ind_{S^x \cap S^y \cap \Strip \neq \eset}\big]}{Z_{N,\beta}^A(x)Z_{N,\beta}^A(y)}  \Bigg]   \,\xrightarrow[N\to\infty]\, 0 \,, \notag \\
		 & N^{\frac{d}{2}-1} \sum_{x,y \in \Z^d} \varphi_N(x,y) \: \bbE \Bigg[ \frac{\E_{x,y}\big[e^{\sfH^y(\omega)}\ind_{S^x \cap S^y \cap \Strip \neq \eset}\big]}{Z_{N,\beta}^A(x)Z_{N,\beta}^A(y)}  \Bigg]   \,\xrightarrow[N\to\infty]\, 0 \,, \notag \\
		 & N^{\frac{d}{2}-1} \sum_{x,y \in \Z^d} \varphi_N(x,y) \: \bbE \Bigg[\frac{1}{Z_{N,\beta}^A(x)Z_{N,\beta}^A(y)}   \Bigg]  \P_{x,y}\big( S^x \cap S^y \cap \Strip \neq \eset \big) \,\xrightarrow[N\to\infty]\, 0 \,.
	\end{align}
	The steps to do that are quite similar to the steps we followed to prove \eqref{firstpartFN}. Therefore, the proof of \eqref{stripest4} has been completed.  Then, the proof of \eqref{stripest5} follows exactly the same lines, since $F_N^{\sfn \sfs}(x,y)$ admits a similar decomposition to \eqref{FNdecomp}.

	\textbf{ (Step 3)}
	Recall from \eqref{quotfluct} that we have to show that
	\begin{align*}
		N^{\frac{d-2}{4}}  \sum_{x \in \Z^d} \varphi_N(x)\, \bigg(\frac{\hat{Z}^{A}_{N,\beta}(x)}{Z^{A}_{N,\beta}(x)}- \big(Z_{N,\beta}^{B^{\geq}}(x)-1\big) \bigg) \xrightarrow[N \to \infty]{L^1(\bbP)} 0 \, .
	\end{align*}
	In Steps $1$ and $2$ we showed that if one decomposes $\hat{Z}^{A}_{N,\beta}(x)$ as
	$\hat{Z}^{A}_{N,\beta}(x)=Z_{N,\beta}^{A,C}(x)+Z_{N,\beta}^{A,B^{<}}(x)+Z_{N,\beta}^{A,B^{\geq}}(x)$ (recall their definitions from \eqref{ZAC},\eqref{Zstrip},\eqref{Znostrip}) then one has that
	\begin{align*}
		N^{\frac{d-2}{4}}  \sum_{x \in \Z^d} \varphi_N(x)\,  \frac{Z^{A,C}_{N,\beta}(x)}{Z^{A}_{N,\beta}(x)} \xrightarrow[N \to \infty]{L^2(\bbP)} 0 \,.
	\end{align*}
	and
	\begin{align*}
		N^{\frac{d-2}{4}}  \sum_{x \in \Z^d} \varphi_N(x)\,  \frac{Z^{A,B^{<}}_{N,\beta}(x)}{Z^{A}_{N,\beta}(x)} \xrightarrow[N \to \infty]{L^2(\bbP)} 0 \,.
	\end{align*}
	Therefore, this last step will be devoted {to showing that}
	\begin{align*}
		N^{\frac{d-2}{4}}  \sum_{x \in \Z^d} \varphi_N(x)\, \bigg(\frac{{Z}^{A,B^{\geq}}_{N,\beta}(x)}{Z^{A}_{N,\beta}(x)}- \big(Z_{N,\beta}^{B^{\geq}}(x)-1\big) \bigg) \xrightarrow[N \to \infty]{L^1(\bbP)} 0 \,.
	\end{align*}
	We can rewrite the expansion of $	Z_{N,\beta}^{A,B^\geq}(x)$, according to the last point that the polymer samples inside $A_N^x$ and the first point that it samples in $B_N^{\geq }$, where we recall the definition of $B_N^{\geq }$, from \eqref{rhodef}. In particular,
	\begin{align} \label{Znostripalt}
		Z_{N,\beta}^{A,B^\geq}(x) =
		\sum_{(t,w)\in A_N^x, \, (r,z)\in B_N^\geq}
		Z_{0,t,\beta}^A(x,w) \cdot\, q_{r-t}(z-w) \, \cdot \sigma \, \eta_{r,z}
		\cdot \,\,Z_{r,N,\beta} (z) \,.
	\end{align}
	where $Z_{0,t,\beta}^A(x,w)$ is the point-to-point partition function from $(0,x)$ to $(t,w)$, defined by $Z_{0,t,\beta}^A(x,w):=1$ if $(t,w)=(0,x)$ and by
	\begin{align} \label{ptppartf}
		Z_{0,t,\beta}^A(x,w):=\sum_{\tau \subset A_N^x \cap([0,t]\times \Z^d): \tau \ni (t,w)} \sigma^{|\tau|}q^{(0,x)}(\tau)\eta(\tau) \, .
	\end{align}
	We will show that if we replace $q_{r-t}(z-w)$ by $q_r(z-x)$ in the expansion of
	\begin{align*}
		N^{\frac{d-2}{4}}  \sum_{x \in \Z^d} \varphi_N(x)\, \cdot \frac{{Z}^{A,B^{\geq}}_{N,\beta}(x)}{Z^{A}_{N,\beta}(x)} \, ,
	\end{align*}
	{via \eqref{Znostripalt}}, then the corresponding error vanishes in $L^{1}(\bbP)$, as $N \to \infty.$ Note that if we perform this replacement, then the right hand side of \eqref{Znostripalt} becomes exactly equal to $Z^{A}_{N,\beta}(x)(Z^{B^\geq}_{N,\beta}(x) -1)$ {and this will lead to the cancellation of the corresponding denominator.}
	We define the set
	\begin{align*}
		B_{N}^\geq (x) :=\big\{(r,z) \in B^\geq_N\colon \, |z-x|<r^{\tfrac{1}{2}+\alpha}\big\} \, .
	\end{align*}
	{where $\alpha$ is defined in~\eqref{Asets}.}
	Then by first restricting to $(r,z) \in B_{N}^\geq(x)$, we want to show that the $L^{1}(\bbP)$ norm of
	\begin{align} \label{l1est}
		 & N^{\frac{d-2}{4}}\sum_{x\in \Z^d} \varphi_N(x)\,
		\sumtwo{(t,w)\in A_N^x}{(r,z)\in B_{N}^\geq(x)}
		\frac{Z_{0,t,\beta}^A(x,w)}{Z_{N,\beta}^A(x)} \Big( q_{r-t}(z-w)-q_r(z-x)\Big) \cdot
		\sigma\,\eta_{r,z} \cdot Z_{r,N,\beta}(z) \,,
	\end{align}
	vanishes as $N \to \infty$.{ We note that the rightmost sum in \eqref{l1est} is essentially over points $(t,w) \in A_N^x$, so that $q_t(x-w)\neq 0$, because otherwise the point to point partition function $Z_{0,t,\beta}^A(x,w)$ is zero. In that case, we observe that if due to the periodicity of the random walk, $q_{r-t}(z-w)=0$ then we also have that $q_r(z-x)=0$, since $q_t(x-w)\neq 0$. Therefore, we shall assume that $q_{r-t}(z-w), q_{r}(z-x)\neq 0$ from now on.}  By Theorem 2.3.11 in \cite{LL10}, we have that for $(r,z) \in B_N^{\geq}(x)$,
	\begin{align} \label{refined1}
		q_r(z-x)  = & {2g_{\frac{r}{d}}(z-x)}\exp\Big(O\big(\tfrac{1}{r}+\tfrac{|z-x|^4}{r^3}\big)\Big)\cdot \ind_{q_r(z-x)\neq 0} \notag \\
		=           & {2g_{\frac{r}{d}}(z-x)} \exp\big(O(r^{-1+4\alpha})\big) \cdot \ind_{q_r(z-x)\neq 0} \, .
	\end{align}
	Furthermore, for $(t,w) \in A_N^x$ we have that
	\begin{align} \label{refined2}
		\quad \quad q_{r-t}(z-w)= & {2g_{\frac{r-t}{d}}(z-w)}\exp\Big(O\big(\tfrac{1}{r-t}+\tfrac{|z-w|^4}{(r-t)^3}\big)\Big)\cdot \ind_{q_{r-t}(z-w)\neq 0} \notag \\
		=                         & {2g_{\frac{r-t}{d}}(z-w)} \exp\big(O(r^{-1+4\alpha})\big)\cdot \ind_{q_{r-t}(z-w)\neq 0} \, ,
	\end{align}
	because we have that $|z-w|\leq |z-x|+|x-w|\leq r^{\frac{1}{2}+\alpha}+N^{\frac{\epsilon}{2}+\alpha} \leq 2r^{\frac{1}{2}+\alpha}$,
	for large $N$ since $r\in (N^\epsilon,N^\rho)$. Also, we have that for large $N$, $|r-t|\geq \frac{1}{2} r$, since $t\leq N^\epsilon$. It is a matter of simple computations to see that
	\begin{align}\label{kernel_est}
		\sup\Bigg\{\bigg|\frac{g_{\frac{r}{d}}(z-x)}{g_{\frac{r-t}{d}}(z-w)} -1 \bigg|  \colon  r>N^{\rho} \,,\, t\leq N^{\epsilon}\,,
		|w-x|<N^{\frac{\epsilon}{2}+\alpha},
		|z-x|<r^{\tfrac{1}{2}+\alpha}  \Bigg\}
		= O\Big(N^{c\,(\epsilon-\rho)}\Big) \, ,
	\end{align}
	for some positive constant $c>0$, {by choosing $\alpha$ sufficiently small. }By Cauchy-Schwarz we obtain the following estimate for the $L^1$-norm of \eqref{l1est},
	\begin{align*}
		     & N^{\frac{d-2}{4}} \sum_{x\in \Z^d} |\varphi_N(x)| \,
		\bbE \,\Bigg[ \frac{1}{Z_{N,\beta}^A(x)} \notag                                                                                                     \\
		     & \qquad \times \Big|\sumtwo{(t,w)\in A_N^x}{(r,z)\in  B_{N}^\geq(x)}  \!\!\! Z_{0,t,\beta}^A(x,w) \Big(q_r(z-x)-q_{r-t}(z-w)\Big)
		\cdot \sigma\,\eta_{r,z} \cdot Z_{r,N,\beta}(z) \Big|\,\Bigg] \notag                                                                                \\
		\leq & \, N^{\frac{d-2}{4}} \sum_{x\in \Z^d} |\varphi_N(x)|
		\, \bbE  \Big[\frac{1}{Z_{N,\beta}^A(x)^2 }\Big]^{1/2}                                                                                              \\
		     & \qquad \times  \bbE \Bigg[\Bigg(\sumtwo{(t,w)\in A_N^x}{(r,z)\in B_{N}^\geq(x)}  \!\!\! Z_{0,t,\beta}^A(x,w) \Big(q_r(z-x)-q_{r-t}(z-w)\Big)
		\cdot \sigma\,\eta_{r,z} \cdot Z_{r,N,\beta}(z) \Bigg)^2\,\Bigg]^{1/2} \notag \, .
	\end{align*}
	By the negative moment estimate, i.e. Proposition \ref{negmom} we have that $\bbE \Big[\big(Z_{N,\beta}^A(x)\big)^{-2}\Big]<\infty$. By expanding the square in the second expectation we have that it is equal to
	\begin{align*}
		        & \sumtwo{(t,w)\in A_N^x}{(r,z)\in B_{N}^\geq(x)}  \bbE \big[\,Z_{0,t,\beta}^A(x,w)^2\big]\,\,  \Big(q_r(z-x)-q_{r-t}(z-w)\Big)^2
		\, \sigma^2 \,\, \bbE\big[Z_{r,N,\beta}(z)^2 \,\big]                                                                                                                                                                         \\
		=
		        & \sumtwo{(t,w)\in A_N^x}{(r,z)\in B_{N}^\geq(x)}  \bbE \big[\,Z_{0,t,\beta}^A(x,w)^2\big]\,\,  \Bigg\{ 1-\frac{q_r(z-x)}{q_{r-t}(z-w)}\Bigg\}^2\, \,q^2_{r-t}(z-w)  \sigma^2 \, \bbE\big[Z_{r,N,\beta}(z)^2 \,\big] \\
		\leq \, & O(N^{2c(\epsilon-\rho)}) \sumtwo{(t,w)\in A_N^x}{(r,z)\in B_{N}^\geq(x)}  \bbE \big[\,Z_{0,t,\beta}^A(x,w)^2\big]\, \,q^2_{r-t}(z-w)  \sigma^2 \, \bbE\big[Z_{r,N,\beta}(z)^2 \,\big] \, ,
	\end{align*}
	by using estimate \eqref{kernel_est} and \eqref{refined1},\eqref{refined2}.
	The last sum is bounded by $\bbE \Big[\big(Z_{N,\beta}^{A,B^\geq}(0)\big)^2\Big]$. By adapting the proof of Lemma \ref{hat}, one can show that	$\bbE \Big[\big(Z_{N,\beta}^{A,B^\geq}(0)\big)^2\Big] =O\big(N^{-\theta(\frac{d}{2}-1)}\big)$, for every $\theta<\rho$. Therefore,
	\begin{align*}
		     & N^{\frac{d-2}{4}} \sum_{x\in \Z^d}  \quad|\varphi_N(x)| \,
		\bbE \,\Bigg[ \frac{1}{Z_{N,\beta}^A(x)} \notag                                                                                                                             \\
		     & \quad  \times \Big|\sumtwo{(t,w)\in A_N^x}{(r,z)\in  B_{N}^\geq(x)}  \!\!\! Z_{0,t,\beta}^A(x,w) \,\,\Bigg\{ 1-\frac{q_r(z-x)}{q_{r-t}(z-w)}\Bigg\}\, \,q_{r-t}(z-w)
		\cdot \sigma\,\eta_{r,z} \cdot Z_{r,N,\beta}(z) \Big|\,\Bigg] \notag                                                                                                        \\
		\leq & \, C \norm{\varphi}_1 \bbE \Big[\big(Z_{N,\beta}^A(x)\big)^{-2} \Big]^{\frac{1}{2}} N^{\frac{d-2}{4}} N^{c(\epsilon-\rho)} N^{-\theta(\frac{d-2}{4})} \, .
	\end{align*}
	In order for the last bound to vanish we need that
	\begin{align*}
		{ (1-\theta)\frac{d-2}{4} +c(\epsilon-\rho)<0} \, .
	\end{align*}
	Since, $\theta \in (0,\rho)$ can be chosen arbitrarily close to $\rho$, it suffices that
	\begin{align*}
		{ (1-\rho)\frac{d-2}{4} +c(\epsilon-\rho)<0} \, .
	\end{align*}
	Rearranging this inequality, we need that
	\begin{align} \label{finalrho}
		\frac{\, c \, \epsilon+\tfrac{d-2}{4} \,}{\, c+\tfrac{d-2}{4} \, }<\rho \, .
	\end{align}
	This is possible since, given a choice of $\epsilon \in (0,1)$, {we proved in Step 2 that \eqref{L2strip} is valid for any $\rho \in (\epsilon,1)$}, therefore we can choose $\rho$, {large enough}, so that \eqref{finalrho} is satisfied.
	To complete Step 3, one needs to show that we can lift the restriction $(r,z) \in B_{N}^\geq(x)$, that is, allow $(r,z) \in B^{\geq}_{N}$, such that $|z-x|\geq r^{\frac{1}{2}+\alpha}$ but this follows by standard moderate deviation estimates and is quite to similar to the proof of \cite{CSZ18b}, thus we omit the details.
\end{proof}

In order to complete the steps needed to prove Theorem \ref{loggaussianity}, one has to show that also Proposition \ref{fourth2} is valid.
But, this is a corollary of Theorem \ref{gaussianity}. Since we are using the diffusive scaling, the fact that $Z^{B \geq}_{N,\beta}(x)$ is the partition function of a polymer which starts sampling noise after time $N^{\rho}$ for some $\rho \in (0,1)$, does not change the asymptotic distribution.
\begin{proof}[Proof of Proposition \ref{fourth2}]
	This Proposition is a corollary of Theorem  \ref{gaussianity}, since one can see that the difference of
	\begin{align*}
		N^{\frac{d-2}{4}} \sum_{x \in \Z^d} \varphi_N(x)\,\big( Z_{N,\beta}(x)-1 \big) \quad \text{      and      }\quad  N^{\frac{d-2}{4}} \sum_{x \in \Z^d} \varphi_N(x)\,\big( Z^{B\geq}_{N,\beta}(x)-1 \big) \, .
	\end{align*}
	vanishes in $L^2(\bbP)$. More specifically, we have that
	\begin{align*}
		     & \norm{ N^{\frac{d-2}{4}} \sum_{x \in \Z^d} \varphi_N(x)\,\big( Z_{N,\beta}(x)-1 \big)- N^{\frac{d-2}{4}} \sum_{x \in \Z^d} \varphi_N(x)\,\big( Z^{B\geq}_{N,\beta}(x)-1 \big)}^{2}_{L^2(\bbP)} \notag \\
		\leq & \,  N^{\frac{d}{2}-1}  \sum_{n=1}^{N^{\rho}} \sigma^2\sum_{x,y \in \Z^d} \varphi_N(x,y)\,q_{2n}(x-y)  \E[e^{\lambda_2(\beta)\mathcal{L}_{N-n}}] \, .
	\end{align*}
	by recalling expression \eqref{convul}. We can bound the last quantity as follows
	\begin{align*}
		     & N^{\frac{d}{2}-1}  \sum_{n=1}^{N^{\rho}} \sigma^2\sum_{x,y \in \Z^d} \varphi_N(x,y)\,q_{2n}(x-y)  \E[e^{\lambda_2(\beta)\mathcal{L}_{N-n}}] \notag      \\
		\leq & \, \E[e^{\lambda_2(\beta)\mathcal{L}_{\infty}}] \,N^{\frac{d}{2}-1}  \sum_{n=1}^{N^{\rho}}\sigma^2\sum_{x,y \in \Z^d} \varphi_N(x,y)\,q_{2n}(x-y)  \, .
	\end{align*}
	By Lemma  \ref{conver} the main contribution to the sum
	\begin{align*}
		N^{\frac{d}{2}-1}  \sum_{n=1}^{N^{\rho}} \sigma^2\sum_{x,y \in \Z^d} \varphi_N(x,y)\,q_{2n}(x-y) \, .
	\end{align*}
	comes from $n \in [\theta N,N]$  for $\theta$ small, therefore it converges to $0$ as $N \to \infty$.
\end{proof}

\bigskip

\noindent \textbf{Acknowledgements.} \,
We thank Francesco Caravenna and Rongfeng Sun for useful comments and Cl\'ement
Cosco for bringing to our attention work \cite{CNN20}. {D. L.} acknowledges financial support from EPRSC through grant EP/HO23364/1 as part of the MASDOC DTC at the University of Warwick.
	{N. Z.
		acknowledges support from EPRSC through grant EP/R024456/1. }

\end{document}